\documentclass[aap]{imsart}

\RequirePackage{amsthm,amsmath,amsfonts,amssymb}
\usepackage{enumitem}
\usepackage{graphicx}
\usepackage{booktabs}
\usepackage{multirow}
\usepackage{threeparttable}
\usepackage{bm}
\RequirePackage[numbers,sort&compress]{natbib}
\RequirePackage[colorlinks,citecolor=blue,urlcolor=blue]{hyperref}
\usepackage{caption}
\startlocaldefs

\theoremstyle{plain}

\newtheorem{theorem}{Theorem}[section]
\newtheorem{remark}{Remark}[section]
\newtheorem{lemma}[theorem]{Lemma}
\newtheorem{corollary}[theorem]{Corollary}
\theoremstyle{definition}

\newtheorem{assumption}[theorem]{Assumption}
\newtheorem{example}{Example}

\newcommand{\non}{\nonumber}

\endlocaldefs

\begin{document}
	
	\begin{frontmatter}
		\title{Asymptotics for Reinforced Stochastic Processes on Hierarchical Networks}
		\runtitle{Asymptotics for Hierarchical Networks}
		
		\begin{aug}
			\author[A]{\fnms{Li}~\snm{Yang}\ead[label=e1]{yangl@xjtu.edu.cn}}
			\author[A]{\fnms{Dandan}~\snm{Jiang}*\ead[label=e2]{jiangdd@mail.xjtu.edu.cn}}
			\author[B]{\fnms{Jiang}~\snm{Hu}\ead[label=e3]{huj156@nenu.edu.cn}}
			\author[B]{\fnms{Zhidong}~\snm{Bai}\ead[label=e4]{baizd@nenu.edu.cn}}
			\address[A]{School of Mathematics and Statistics, Xi'an Jiaotong University\printead[presep={,\ }]{e1,e2}}
			
			\address[B]{KLASMOE and School of Mathematics and Statistics, Northeast Normal University\printead[presep={,\ }]{e3,e4}}
		\end{aug}
		
		\begin{abstract}
        In this paper, we analyze the asymptotic behavior of a system of interacting reinforced stochastic processes $({\bf Z}_n, {\bf N}_n)_n$ on a directed network of $N$ agents. The system is defined by the coupled dynamics ${\bf Z}_{n+1}=(1-r_{n}){\bf Z}_{n}+r_{n}{\bf X}_{n+1}$ and ${\bf N}_{n+1}=(1-\frac{1}{n+1}){\bf N}_n+\frac{1}{n+1}{\bf X}_{n+1}$, where agent actions $\mathbb{P}(X_{n+1,j}=1\mid{\cal F}_n)=\sum_{h} w_{hj}Z_{nh}$ are governed by a column-normalized adjacency matrix ${\bf W}$, and $r_n \sim cn^{-\gamma}$ with $\gamma \in (1/2, 1]$. Existing asymptotic theory has largely been restricted to irreducible and diagonalizable ${\bf W}$. We extend this analysis to the broader and more practical class of reducible and non-diagonalizable matrices ${\bf W}$ possessing a block upper-triangular form, which models hierarchical influence. We first establish synchronization, proving $({\bf Z}^\top_n, {\bf N}^\top_n)^\top \to Z_\infty {\bf 1}$ almost surely, where the distribution of the limit $Z_\infty$ is shown to be determined solely by the internal dynamics of the leading subgroup. Furthermore, we establish a joint central limit theorem for $({\bf Z}_n,{\bf N}_n)_n$, revealing how the spectral properties and Jordan block structure of ${\bf W}$ govern second-order fluctuations. We demonstrate that the convergence rates and the limiting covariance structure exhibit a phase transition dependent on $\gamma$ and the spectral properties of ${\bf W}$. Crucially, we explicitly characterize how the non-diagonalizability of ${\bf W}$ fundamentally alters the asymptotic covariance and introduces new logarithmic scaling factors in the critical case ($\gamma=1$). These results provide a probabilistic foundation for statistical inference on such hierarchical network structures.
		\end{abstract}
		
		\begin{keyword}[class=MSC]
			\kwd[Primary ]{60F05}
			\kwd{60F15}
			\kwd{60K35}
			\kwd[; secondary ]{91D30}
		\end{keyword}

		\begin{keyword}
			\kwd{Multi-urn model}
            \kwd{Hierarchical network}
			\kwd{Interacting systems}
			\kwd{Synchronization}
		\end{keyword}
		
	\end{frontmatter}
	
		\section{Introduction}
    Complex systems composed of interacting components have attracted significant attention across scientific disciplines, owing to their rich theoretical structures and diverse applications (see, e.g., \cite{r11,r12}). In neuroscience, the brain is not merely a network in the anatomical sense but a dynamic and intelligent system for information processing, where network structure directly impacts cognitive function (see, e.g., \cite{r13,r14}). In the life sciences, researchers have utilized multilayer networks to model interactions from molecular to species levels, uncovering fundamental principles that govern the organization and evolution of biological systems (see, e.g., \cite{r15,r16}). Similarly, in economic systems, network-based models of interconnections among financial institutions and firms have enabled the identification of pathways for systemic risk propagation, providing quantitative support for financial stability policies (see, e.g., \cite{r17,r34}). The study of social networks has revealed universal patterns in human relationships and the network-driven mechanisms behind social phenomena like cultural transmission and behavioral diffusion (see, e.g., \cite{r18},\cite{r19},\cite{r20}). Collectively, these findings underscore the power of a network-based perspective in reshaping our understanding of complex systems.

    A notable feature of many such systems is the emergence of similar macroscopic behaviors, a phenomenon known as {\it synchronization}, despite substantial heterogeneity among agents and the complex structure of their interactions.
    Understanding the microscopic mechanisms that generate synchronization is fundamental for both prediction and intervention. This work addresses this challenge by focusing on a particular class of networked stochastic systems driven by {\it reinforcement}, a fundamental feedback mechanism that amplifies frequently occurring events.

    Reinforcement describes the tendency for the probability of an event to increase with the frequency of its past occurrences. It underlies diverse natural and social phenomena, such as the amplification of gene expression in biology, the formation of preferences in economics, and the consolidation of behavioral patterns in social networks. Formally, it is classically modeled by the P\'{o}lya urn process (\cite{r21}), in which drawing a ball and returning it with an additional one of the same color formalizes self-reinforcing feedback. Over the past century, this paradigm has inspired a broad family of reinforced stochastic processes (some variants can be found in \cite{r22},\cite{r23},\cite{r24},\cite{r25},\cite{r35},\cite{r26},\cite{r27},\cite{r28},\cite{r29},\cite{r30},\cite{r31},\cite{r32},\cite{r33}), which in turn have characterized the long-term behavior of these dynamics in a single-agent setting.

    To capture the dynamics of systems composed of multiple interacting components, demands a shift from these single-agent frameworks to multi-agent models. In this context, each agent can be represented by an urn whose composition encodes its internal state, leading naturally to interacting urn systems. A well-studied model is the {\it mean-field interacting urn system}. 
    For example, the reference \cite{r36} investigated a system of countably many exponentially reinforced urns, introducing interactions via a Bernoulli($p$) sampling mechanism. The reference \cite{r37} further developed the theory in a system of interacting urns with mean-field interactions, proving synchronization to and establishing a central limit theorem (CLT) for the empirical average as the number of urns  tends to infinity. Subsequent work within this mean-field paradigm, such as \cite{r38}, examined second-order asymptotics, showing how the convergence rate depends on an interaction parameter $\alpha\in(0,1]$. Additional analyses of interacting systems can be found in \cite{r39}, \cite{r40}, \cite{r41} and \cite{r9}.

    While mean-field models capture global interactions, many systems exhibit more localized and heterogeneous influence patterns. To account for such structures, the paper \cite{r1} proposed a framework in which $N$ reinforced agents interact via a weighted directed graph ${\cal G}=({\cal V},{\cal E})$, with vertex set ${\cal V}=\{1,2,\dots,N\}$ and edge set ${\cal E}\subseteq{\cal V}\times{\cal V}$. The network influence structure is described by a nonnegative, column-normalized adjacency matrix ${\bf W}=[w_{hj}]_{h,j\in{\cal V}}$ satisfying $\sum_{h=1}^N w_{hj}=1$ for all $j$. The diagonal entry $w_{jj}$ quantifies self-reinforcement, while the off-diagonal entries $w_{hj}$ capture the influence exerted by agent $h$ on agent $j$. Each agent $j\in{\cal V}$ is associated with a binary action sequence $(X_{nj})_{n\ge1}\in\{0,1\}$ and an inclination process $(Z_{nj})_{n\ge0}$. Let ${\cal F}_n=\sigma({\bf Z}_0,{\bf X}_1,\dots,{\bf X}_n)$ denote the natural filtration. Conditional on ${\cal F}_n$, the actions at time $n+1$ are independent across agents, with
    \begin{equation}\label{eq49}
	   \mathbb{P}(X_{n+1,j}=1\mid{\cal F}_n)=\sum_{h=1}^N w_{hj}Z_{nh}, \quad j\in{\cal V},
    \end{equation}
    and the inclinations evolve as
    \begin{equation}\label{eq15}
	   Z_{nh}=(1-r_{n-1})Z_{n-1,h}+r_{n-1}X_{nh}, \quad h\in{\cal V},
    \end{equation}
    where $r_n\in[0,1)$ is a decaying step size, with initial state ${\bf Z}_0\in[0,1]^{N}$. Under the assumptions that ${\bf W}$ is irreducible and diagonalizable, and $r_n\sim cn^{-\gamma}$ with $\gamma\in(1/2,1]$, the paper \cite{r1} proved almost sure synchronization of the inclination vector ${\bf Z}_n$ and established its corresponding  CLTs.
    Subsequent developments by \cite{r2} extended the analysis to the joint process of inclinations and empirical actions $({\bf Z}_n,{\bf N}_n)_n$ governed by
    \begin{equation}\label{eq3}
	\begin{cases}
		{\bf Z}_{n+1}=(1-r_{n}){\bf Z}_{n}+r_{n}{\bf X}_{n+1},\\[3pt]
		{\bf N}_{n+1}=\left(1-\dfrac{1}{n+1}\right){\bf N}_n+\dfrac{1}{n+1}{\bf X}_{n+1}.
	\end{cases}
    \end{equation}
    Under the same structural conditions on ${\bf W}$, they obtained almost sure synchronization and corresponding CLTs for this coupled system. 

    Later, the reference \cite{r4} removed the diagonalizability assumption and, under the irreducibility condition on the adjacency matrix, derived necessary and sufficient conditions that fully characterize the first-order asymptotic behavior of \eqref{eq3}. Related work by \cite{r5} investigated asymptotic polarization phenomena, identifying regimes in which the common limiting inclination takes extreme values with positive or zero probability.

    Taken together, previous studies have established a coherent asymptotic theory for irreducible network structures, where the assumption of diagonalizability is essential for deriving second-order results. Yet, as noted in \cite{r4}, the naive diagonalizability assumption is difficult to verify in practice and may limit the range of applicable models.
    In many empirical settings, the underlying network often exhibits reducible or non-diagonalizable structures that capture asymmetric or unidirectional influence, such as those observed in hierarchical organizations (see, e.g., \cite{rr1}, \cite{rr3}, \cite{rr2}). Consequently, the asymptotic behavior of systems with these more general interaction structures remains a largely open question.


    To bridge this gap, this paper develops an asymptotic framework for reinforced dynamics on hierarchical networks. We model such systems by first partitioning the population into $S$ subgroups ${\cal G}_1,\dots,{\cal G}_S$, which induces a hierarchical adjacency matrix ${\bf W}$ of block upper-triangular form:
    \begin{equation}\label{eq12}
    {\bf W}=\begin{pmatrix}
    {\bf W}_{11} & {\bf W}_{12} & \cdots & {\bf W}_{1S}\\
    {\bf 0} & {\bf W}_{22} & \cdots & {\bf W}_{2S}\\
    \vdots & \vdots & \ddots & \vdots\\
    {\bf 0} & {\bf 0} & \cdots & {\bf W}_{SS}
    \end{pmatrix}.
    \end{equation}
    This structure encodes a unidirectional hierarchy: influence propagates from upstream groups $h$ to downstream groups $j$ via the blocks ${\bf W}_{hj}$ ($h<j$), while the zero blocks below the diagonal indicate that downstream groups do not feed back to upstream groups. The leading block ${\bf W}_{11}$ is assumed irreducible, and each downstream diagonal block ${\bf W}_{hh}$ ($h\ge2$) satisfies $\|{\bf W}_{hh}\|_1<1$. The special case $S=1$ recovers the classical irreducible setting of \cite{r1,r2}.

    A central objective of this work is to establish  the first- and second-order asymptotic properties for the joint process $({\bf Z}_n, {\bf N}_n)$ in \eqref{eq3}. We achieve this for hierarchical interaction matrices \eqref{eq12} under a general setting that permits both reducibility and non-diagonalizability. The main contributions are as follows:
	
	a) {\it First-order synchronization.} We develop the first-order asymptotic theory for general reducible hierarchical networks. This extends classical synchronization results, which focused primarily on irreducible structures to settings with top-down influence. We show that the entire system with its downstream subgroups achieves almost sure synchronization. Remarkably, this synchronization exhibits a hierarchical dominance, where the limit $Z_\infty$ is dictated entirely by the dynamics of the leading irreducible subgroup ${\cal G}_1$.
    
	 b) {\it Second-order asymptotics.} We establish the joint CLTs for $({\bf Z}_n,{\bf N}_n)_n$ in the general reducible and non-diagonalizable setting. The spectral characteristics of ${\bf W}$ together with the step-size parameters $(\gamma, c)$ give rise to a dynamic phase transition in the system's second-order behavior, reflected in qualitative changes to the asymptotic covariance and convergence rates. Notably, in our model, the Jordan block structure of ${\bf W}$ modifies the asymptotic covariance, which introduces additional leading components and slows convergence in certain spectral regimes relative to the classical diagonalizable case. Moreover, we explicitly express the limiting covariance matrix as a function of the spectral characteristics of ${\bf W}$, specifically its eigenvalues and the size and structure of the Jordan blocks associated with the corresponding generalized eigenvectors.

    c) {\it Statistical inference for hierarchical networks.} Building on the derived CLTs, we develop a principled framework for statistical inference in hierarchical systems. The framework provides confidence intervals for the synchronization limit $Z_\infty$, confidence regions for structural parameters, and formal hypothesis tests for the adjacency matrix ${\bf W}$. These results provide principled and flexible statistical tools for validation and uncertainty quantification in complex hierarchical networks.
   
   In summary, this work lays a complete asymptotic and statistical foundation for hierarchical reinforced networks, bridging theoretical limits with a practical inference framework that delivers quantifiable uncertainty for both predictions and structural discoveries. Beyond theoretical interest, these results are particularly relevant for real-world systems where behavior evolves under both self-reinforcement and structured interactions. For instance, in social networks, repeated individual choices strengthen personal preferences, while exposure to friends' or opinion leaders' actions shapes collective dynamics. Similarly, in biological systems, decisions or signals propagate along hierarchical pathways, while local feedback loops simultaneously reinforce existing tendencies. By capturing the key interplay between reinforcement and network influence, our framework provides a structured approach to analyzing emergent collective behavior and understanding the dynamics of complex networked systems.

	The remainder of this paper is organized as follows. Section \ref{sec2} introduces the notations and formal assumptions underlying our model. Section \ref{sec3} presents the main theoretical results, including both first- and second-order asymptotic properties of the stochastic process $({\bf Z}_n,{\bf N}_n)_n$. Section \ref{sec4} is devoted to the proofs. Section \ref{sec4.1} outlines the overall proof strategy, and Section \ref{sec4.2} provides detailed proofs for the main results. Building upon this theory, we develop a framework for statistical inference in Section \ref{sec5}, which includes the construction of hypothesis tests for network structures and confidence regions for key parameters. We then use simulation studies in Section \ref{sec7} to illustrate our theoretical findings and explore the behavior of the model under various settings. Technical lemmas and auxiliary results used throughout the paper are collected in Appendix.

	\section{Notation and Assumptions}\label{sec2}
	In this section, we introduce the notation and key assumptions that underpin the analytical framework of this paper. 
	
	In the sequel, we adopt the following notational conventions: random variables are denoted by uppercase letters (e.g., $ X, Z, N \dots$), constants are represented by lowercase letters (e.g., \( a, b, c, \dots \)), vectors and matrices are indicated by bold letters (e.g., $ \mathbf{X}, \mathbf{Z}, \mathbf{N}, \mathbf{W}, \dots$), sets are denoted by calligraphic letters $\{{\cal G},{\cal F}\}$, and functions are represented by script letters $\{\mathbb{E}, \mathbb{P}\}$. Let $(\cdot)^{-1}$, $(\cdot)^\top$, and $\overline{(\cdot)}^\top$ denote the matrix inverse, transpose, and conjugate transpose, respectively. For vectors, $\|\cdot\|$ denotes the $L^2$ norm of a vector. For matrices, $\|\cdot\|_1, \|\cdot\|_2$ and $\|\cdot\|_\infty$ denote  the $L_1$ norm, spectral norm and $L_\infty$ norm, respectively.
	Moreover, we denote by $|\cdot|$ the sum of the modulus of its entries for vectors and matrices. Finally, the notation $f(n)\sim g(n)$ indicates that $\lim_{n\to\infty}f(n)/g(n)=1$.

	Throughout this paper, we need the following assumptions. The first assumption concerns the convergence behavior of the step size sequence $(r_n)_n$.

	\begin{assumption}\label{as1}
		There exist real constants $c>0$ and $1/2<\gamma\le 1$ such that
		\begin{equation}\label{eq11}
			\lim\limits_{n\to\infty}n^\gamma r_n=c.
		\end{equation}
		Furthermore, when $\gamma=1$, we require the following stronger condition for further analyses,
		\begin{equation*}
			nr_n-c=O(n^{-1}).
		\end{equation*}
	\end{assumption}
	
	\begin{assumption}\label{as2}
		The adjacency matrix ${\bf W}$ satisfies the following conditions:
		
		(1) The adjacency matrix $\mathbf{W}$ is column-normalized, i.e., $\sum_{h=1}^N {\bf W}_{hj}=1$ for all $j\in\{1,2,\dots,N\}$.
		
		(2) The submatrix ${\bf W}_{11}$ is irreducible, and $\max_{h\in\{2,\cdots,S\}}\|{\bf W}_{hh}\|_1<1$.
	\end{assumption}

    Under Assumption \ref{as2}, the adjacency matrix ${\bf W}$ has a simple largest eigenvalue, which equals 1. The Jordan decomposition of ${\bf W}$ is given by
    \begin{equation}
	   \widetilde{\bf P}{\bf W}\widetilde{\bf P}^{-1}=\widetilde{\bf J}= 
	\begin{pmatrix}
		1 & \bf 0 & \cdots & \bf 0 \\
		\bf 0 & {\bf J}_1 & \cdots & \bf 0 \\
		\vdots & \vdots & \ddots & \vdots \\
		\bf 0 & \bf 0 & \cdots & {\bf J}_T
	\end{pmatrix},
    \end{equation}
    where the transformation matrices are explicitly defined by their columns as
    \begin{gather*}
	   \widetilde{\bf P}^\top=({\bf p}_1, {\bf p}_2, \dots, {\bf p}_N),\quad
	   \widetilde{\bf Q}=\widetilde{\bf P}^{-1}=({\bf q}_1, {\bf q}_2, \dots, {\bf q}_N),
    \end{gather*}
    The rows of $\widetilde{\bf P}$ are the generalized left eigenvectors, and the columns of $\widetilde{\bf Q}$ are the generalized right eigenvectors.
    Let ${\bf J}=\text{Diag}({\bf J}_1, \dots, {\bf J}_T)$ be the matrix of Jordan blocks associated with the non-dominant spectrum $\mbox{Sp}({\bf W})\setminus\{1\}=\{\lambda_1,\dots,\lambda_T\}$, and let $\rho_t$ be the order of the block ${\bf J}_t$ for $t\in\{1,\cdots,T\}$. For clarity, we define
    \begin{gather*}
	   \tau=\max_{t}\mbox{Re}(\lambda_t),\quad \tau^\ast=\min_{t}\mbox{Re}(\lambda_t),\quad \text{and}\quad 
	   \rho=\max_t\{\rho_t:\mbox{Re}(\lambda_t)=\tau\}.
    \end{gather*}

    We partition the transformation matrices to isolate the dominant eigenvector associated with eigenvalue 1 from the remaining eigenvectors:
    \begin{equation}
	   \widetilde{\bf P}=\begin{pmatrix} {\bf p}_1^\top \\ {\bf P}^\top \end{pmatrix}, \quad
	   \widetilde{\bf Q}=({\bf q}_1,{\bf Q}).
    \end{equation}
    For normalization, we set
    \begin{equation}\label{eq54}
	   {\bf p}_1=N^{-1/2} {\bf 1}.
    \end{equation}
    The identity $\widetilde{\bf Q}\widetilde{\bf P} = {\bf I}$ and ${\bf W} = \widetilde{\bf Q}\widetilde{\bf J}\widetilde{\bf P}$ yield 
    \begin{align}\label{eq55}
        {\bf p}_1^\top {\bf q}_1=1, \quad
        {\bf p}_1^\top {\bf Q}={\bf 0}, \quad
        {\bf P}^\top {\bf q}_1={\bf 0}, \quad
        {\bf P}^\top {\bf Q}={\bf I}.\\
        {\bf I}={\bf q}_1 {\bf p}_1^\top+{\bf Q}{\bf P}^\top,\quad {\bf W}={\bf q}_1{\bf p}_1^\top+{\bf Q}{\bf J}{\bf P}^\top.\non
    \end{align}
    These definitions and identities provide the formal framework for the subsequent analysis.

	\section{Main Results}\label{sec3}
	
	This section presents the main results of this paper, focusing on the first- and second-order convergence properties of the joint process $({\bf Z}_n, {\bf N}_n)_n$. The first-order convergence characterizes synchronization, while the second-order convergence quantifies the convergence rate and synchronization rate.
	
	\subsection{Almost Sure Convergence of $({\bf Z}_n,{\bf N}_n)_n$}
	The first result establishes the strong convergence of the stochastic process $({\bf Z}_n,{\bf N}_n)_n$.
	\begin{theorem}[Synchronization]\label{th7}
		Under Assumptions \ref{as1} and \ref{as2}, there exists a random variable $Z_\infty$ taking values in $[0, 1]$ such that
		\begin{equation}
			\begin{pmatrix}
				{\bf Z}_n\\
				{\bf N}_n
			\end{pmatrix}\overset{a.s.}{\to}Z_\infty{\bf 1}.
		\end{equation}
	\end{theorem}
	This result shows that, regardless of the initial states ${\bf Z}_0$ of agents in the population ${\cal G}$, the proposed interaction dynamics and enhanced decision mechanism ensure that both the agent inclinations and the empirical means converge almost surely to a common limit $Z_\infty$. This implies that the entire population achieves asymptotic synchronization. The asymptotic behavior of $Z_\infty$ is characterized by the following Theorem \ref{th11}, Corollary \ref{co1} and Theorem \ref{th1}.
	
	\begin{theorem}\label{th11}
		Suppose $S\ge2$. Under Assumptions \ref{as1} and \ref{as2}, the distribution of the synchronization limit $Z_\infty$ is determined solely by the interaction structure within the leading subgroup ${\cal G}_1$, and is independent of the remaining subgroups ${\cal G}_k$ for $k\in\{2,\dots,S\}$.
	\end{theorem}

	The first moment of the synchronization limit $Z_\infty$ is given by the following corollary.
	
	\begin{corollary}\label{co1}
		Under Assumptions \ref{as1} and \ref{as2}, the mathematical expectation of the synchronization limit, $\mathbb{E}[Z_\infty]$, is a weighted average of the initial states of group ${\cal G}_1$, given by
		\begin{equation}
			\mathbb{E}(Z_\infty)=N^{-1/2}{\bf q}^\top_{11}\mathbb{E}({\bf Z}^{(1)}_0),
		\end{equation}
		where ${\bf Z}^{(1)}_n$ denotes the agent inclination corresponding to the subgroup ${\cal G}_1$, and ${\bf q}_{11}$ is the right eigenvector of ${\bf W}_{11}$ corresponding to the eigenvalue $1$.
	\end{corollary}
	\begin{remark}
		This corollary clarifies how the synchronization limit is formed on average. The expected limit $\mathbb{E}(Z_\infty)$ is a weighted average of the initial states in the leading subgroup ${\cal G}_1$. The weights for this average are the components of the normalized vector $N^{-1/2}{\bf q}^\top_{11}$. The related vector ${\bf q}^\top_{11}$ is the dominant right eigenvector of ${\bf W}_{11}$ and represents the relative intrinsic influence of each agent within that group. Therefore, the initial states of more influential agents in the leading group have a greater impact on the synchronization limit of the entire network.
	\end{remark}

	\begin{remark}
		When $S=1$, i.e., ${\bf W}={\bf W}_{11}$, the properties of $Z_\infty$ coincide with those established in Theorem {\rm 3.1} of {\rm \cite{r1}}.
	\end{remark}
	
    Under the classical assumptions of irreducibility and diagonalizability, the prior work \cite{r1} established two key properties of the synchronization limit $Z_\infty$. While our setting, which focuses on the more complex hierarchical structure \eqref{eq12}, relaxes these assumptions, Theorem~\ref{th11} reveals a fundamental insight: the distribution of $Z_\infty$ is solely governed by the leading subgroup ${\cal G}1$ via the submatrix ${\bf W}_{11}$. Consequently, the following  Theorem~\ref{th1} demonstrates that these same asymptotic properties emerge naturally even in this generalized context. This shows that the foundational laws identified in \cite{r1} are not artifacts of their idealized assumptions but are, in fact, robust phenomena driven primarily by the network's leading echelon.

	\begin{theorem}\label{th1}
		Under Assumptions \ref{as1} and \ref{as2}, the following holds: 
		
		{\rm (a)} If the initial state set ${\bf Z}_0$ satisfies
		\begin{equation}\label{eq60}
			{\mathbb P}\bigg(\bigcap\limits_{j=1}^{N_1}\{Z_{0,j}=0\}\bigg)+{\mathbb P}\bigg(\bigcap\limits_{j=1}^{N_1}\{Z_{0,j}=1\}\bigg)<1,
		\end{equation}
		where $N_1$ denotes the order of the matrix ${\bf W}_{11}$, corresponding to the number of agents in the leading subgroup ${\cal G}_1$. Then, the limit of synchronization $Z_\infty$ satisfies ${\mathbb P}(Z_\infty=0)+{\mathbb P}(Z_\infty=1)<1$.
		
		{\rm (b)} ${\mathbb P}(Z_\infty=z)=0$ for any $z\in(0,1)$.
	\end{theorem}
	Part (a) asserts that if the initial states of all agents in ${\cal G}_1$ are not almost surely degenerate (i.e., not all 0 or all 1), then ${\mathbb P}(Z_\infty\in(0,1))>0$. Part (b) shows that $Z_\infty$ has no point masses within the interval $(0,1)$. These properties imply that $Z_\infty$ remains genuinely random. In the following, we will study the second-order convergence of $({\bf Z}_n, {\bf N}_n)_n$, where the limiting covariance structure depends on $Z_\infty$. Consequently, the second-order fluctuations of $({\bf Z}_n,{\bf N}_n)_n$ are governed by a non-degenerate, stochastic covariance matrix.

	\subsection{Central Limit Theorem for $({\bf Z}_n,{\bf N}_n)_n$}
	For clarity in subsequent discussions, we define the cumulative sum
	\begin{equation}
		{\cal I}_0=0,\quad \text{and}\quad {\cal I}_t = \sum_{k=1}^{t} \rho_k \quad \text{for \  } t\in\{1,2,\dots,T\}.
	\end{equation}
	
	We begin by analyzing the convergence rate of the process $({\bf Z}_n, {\bf N}_n)_n$ in the regime $1/2<\gamma<1$.
	\begin{theorem}[Convergence Rate for $1/2<\gamma<1$]\label{th3}
		Under Assumptions \ref{as1} and \ref{as2}, when $N\ge1$, $1/2<\gamma<1$, it holds that:
		\begin{equation*}
			n^{\gamma-\frac{1}{2}}\begin{pmatrix}
				{\bf Z}_n-Z_\infty{\bf 1}\\ {\bf N}_n-Z_\infty{\bf 1}
			\end{pmatrix}\to{\cal N}\left({\bf 0},Z_\infty(1-Z_\infty)\begin{pmatrix}
				\widetilde{\bm\Sigma}_\gamma & \;\;\widetilde{\bm\Sigma}_\gamma\\
				\widetilde{\bm\Sigma}_\gamma & \;\;\widetilde{\bm\Sigma}_\gamma+\widehat{\bm\Gamma}_\gamma
			\end{pmatrix}\right)\ \ \ \ \ \ \text{\textit{stably}},
		\end{equation*}
		where
		\begin{equation}\label{eq4}
			\widetilde{\bm \Sigma}_\gamma=\widetilde{\sigma}^2_\gamma{\bf 1}{\bf 1}^\top\ \ {\rm and}\ \ \widetilde{\sigma}^2_\gamma=\frac{c^2\|{\bf q}_1\|^2}{N(2\gamma-1)},
		\end{equation}
		and
		\begin{equation}\label{eq63}
			\widehat{\bm\Gamma}_\gamma=\widehat{\sigma}^2_\gamma{\bf 1}{\bf 1}^\top\ \ {\rm and}\ \ \widehat{\sigma}^2_\gamma=\frac{c^2\|{\bf q}_1\|^2}{N(3-2\gamma)}.
		\end{equation}
	\end{theorem}
	
	\begin{remark}
		Based on the linear invariance of the normal distribution, the asymptotic normality of $\begin{pmatrix}
			Z_{ni}-Z_{nj}\\ N_{ni}-N_{nj}
		\end{pmatrix}$ can be directly derived from the asymptotic normality of $\begin{pmatrix}
			{\bf Z}_n-Z_\infty{\bf 1}\\ {\bf N}_n-Z_\infty{\bf 1}
		\end{pmatrix}$. Let ${\bf e}_i$ and ${\bf e}_j$ be $N$ dimensional vectors where the $i$th and $j$th components are 1, respectively, and all other components are 0. Since
		\begin{equation*}
			\begin{pmatrix}
				Z_{ni}-Z_{nj}\\ N_{ni}-N_{nj}
			\end{pmatrix}=\begin{pmatrix}
				{\bf e}^\top_i-{\bf e}^\top_j & {\bf 0}\\
				{\bf 0} & {\bf e}^\top_i-{\bf e}^\top_j
			\end{pmatrix}\begin{pmatrix}
				{\bf Z}_n-Z_\infty{\bf 1}\\ {\bf N}_n-Z_\infty{\bf 1}
			\end{pmatrix}
		\end{equation*}
		and	
		\begin{equation*}
			\begin{pmatrix}
				{\bf e}^\top_i-{\bf e}^\top_j & {\bf 0}\\
				{\bf 0} & {\bf e}^\top_i-{\bf e}^\top_j
			\end{pmatrix}\begin{pmatrix}
				\widetilde{\bm\Sigma}_\gamma & \;\;\widetilde{\bm\Sigma}_\gamma\\
				\widetilde{\bm\Sigma}_\gamma & \;\;\widetilde{\bm\Sigma}_\gamma+\widehat{\bm\Gamma}_\gamma
			\end{pmatrix}\begin{pmatrix}
				{\bf e}_i-{\bf e}_j & {\bf 0}\\
				{\bf 0} & {\bf e}_i-{\bf e}_j
			\end{pmatrix}=0,
		\end{equation*}
		it follows that when $1/2<\gamma<1$, the synchronization rate between any two agents in the population is faster than population synchronization rate $n^{\gamma-1/2}$, as subsequently detailed in Theorem \ref{th4}.
	\end{remark}
	
	We now consider the regime $\gamma=1$. For $N=1$, the convergence of the process $({\bf Z}_n, {\bf N}_n)_n$ has been established in Theorem 3.3 of \cite{r2}. When $N\ge2$, the second-order asymptotic behavior of $({\bf Z}_n,{\bf N}_n)_n$ is governed by $\tau$, the second largest real part among the eigenvalues of ${\bf W}$. Different values of $\tau$ yield distinct convergence rates. The following result addresses the regime where $\tau<1-(2c)^{-1}$.

	\begin{theorem}[Convergence Rate for $\gamma=1$, $\tau<1-(2c)^{-1}$]\label{th5}
		Under Assumptions \ref{as1} and \ref{as2}, when $N\ge2$, $\gamma=1$ and $\tau<1-(2c)^{-1}$, we have
		\begin{equation*}
			\sqrt{n}\begin{pmatrix}
				{\bf Z}_n-Z_\infty{\bf 1}\\ {\bf N}_n-Z_\infty{\bf 1}
			\end{pmatrix}\to{\cal N}\left({\bf 0},Z_\infty(1-Z_\infty)\begin{pmatrix}
				\widetilde{\bm\Sigma}_1+\widehat{\bm\Sigma}_{\bf ZZ} & \;\;\widetilde{\bm\Sigma}_1+\widehat{\bm\Sigma}_{\bf ZN}\\
				\widetilde{\bm\Sigma}_1+\widehat{\bm\Sigma}^\top_{\bf ZN} & \;\;\widetilde{\bm\Sigma}_1+\widehat{\bm\Sigma}_{\bf NN}
			\end{pmatrix}\right)\ \ \ \ \ \ \text{\textit{stably}},
		\end{equation*}
		where $\widetilde{\bm\Sigma}_1$ is given by \eqref{eq4} with $\gamma=1$. For any $u, v \in \{1, 2, \dots, T\}$, and for indices $i$, $j$ satisfying ${\cal I}_{u-1}< i \le {\cal I}_u$ and ${\cal I}_{v-1}< j \le {\cal I}_v$, respectively, the matrix $\widehat{\bm \Sigma}_{\bf ZZ}$ is given by
		\begin{align}\label{eq1}
			\widehat{\bm \Sigma}_{\bf ZZ}=&{\bf P}\widehat{\bf S}_{\bf ZZ}{\bf P}^\top,\ \ {\rm and}\\
			[\widehat{\bf S}_{\bf ZZ}]_{i,j}=&\sum\limits_{s=0}^{j-{\cal I}_{v-1}-1}\sum\limits_{t=0}^{i-{\cal I}_{u-1}-1}\frac{c^{t+s+2}(t+s)!}{[-1+c(2-\lambda_u-\lambda_v)]^{t+s+1}}{\bf q}^\top_{i-t+1}{\bf q}_{j-s+1}.\non
		\end{align}
		The matrix $\widehat{\bm\Sigma}_{\bf ZN}$ is given by
		\begin{align}\label{eq65}
			&\qquad \widehat{\bm \Sigma}_{\bf ZN}={\bf P}\widehat{\bf S}_{\bf ZN}\widetilde{\bf P}^\top,\ \ {\rm and}\\
			&[\widehat{\bm S}_{\bf ZN}]_{i,1}=(1-c)\sum\limits_{t=0}^{i-{\cal I}_{u-1}-1}\frac{t!}{(1-\lambda_u)^{t+1}}{\bf q}^\top_{i-t+1}{\bf q}_{1},\non\\
			&[\widehat{\bm S}_{\bf ZN}]_{i,j+1}=\sum\limits_{s=1}^{j-{\cal I}_{v-1}-1}\sum\limits_{t=0}^{i-{\cal I}_{u-1}-1}{\bf q}^\top_{i-t+1}{\bf q}_{j-s+1}c^{t+s+1}\cdot\non\\
			&\qquad\qquad\bigg\{\frac{(t+s-1)!{\cal N}_2(t+s-1,\lambda_u,\lambda_v,c)}{{\cal D}_2(t+s-1,t+s,\lambda_u,\lambda_v,c)}+\frac{\lambda_vc(t+s)!{\cal N}_2(t+s,\lambda_u,\lambda_v,c)}{{\cal D}_2(t+s,t+s+1,\lambda_u,\lambda_v,c)}   \bigg\}\non\\
			&\qquad+\sum\limits_{t=0}^{i-{\cal I}_{u-1}-1}{\bf q}^\top_{i-t+1}{\bf q}_{j+1}c^{t+1}t!\frac{(c-1){\cal N}_2(t,\lambda_u,\lambda_v,c)+{\cal N}_3(t,\lambda_u,c)}{{\cal D}_2(t,t+1,\lambda_u,\lambda_v,c)}.\non
		\end{align}
		And the matrix $\widehat{\bm \Sigma}_{\bf NN}$ is given by
		\begin{align}\label{eq66}
			&\qquad\widehat{\bm \Sigma}_{\bf NN}=\widetilde{\bf P}\widehat{\bf S}_{\bf NN}\widetilde{\bf P}^\top,\ \ {\rm and}\\
			&[\widehat{\bf S}_{\bf NN}]_{1,1}=(c-1)^2\|{\bf q}_1\|^2,  \non\\
			&[\widehat{\bf S}_{\bf NN}]_{1,j+1}=[\widehat{\bm S}_{\bf NN}]_{j+1,1}={\bf q}^\top_{j+1}{\bf q}_1\frac{1-c}{1-\lambda_v}+
			\sum\limits_{s=1}^{j-{\cal I}_{v-1}-1}{\bf q}^\top_{j-s+1}{\bf q}_1c^{s-1}(c^{-1}-1)\cdot\non\\ 
            &\qquad\qquad\qquad\qquad\qquad\qquad\left\{c^2(s-1)!\frac{{\cal N}_4(s-1,\lambda_v,c)}{{\cal N}_3(s-1,\lambda_v,c)}+c^3\lambda_vs!\frac{{\cal N}_4(s,\lambda_v,c)}{{\cal N}_3(s,\lambda_v,c)}\right\}.\non\\
			&[\widehat{\bm S}_{\bf NN}]_{i+1,j+1}=\sum\limits_{t=1}^{i-{\cal I}_{u-1}-1}\sum\limits_{s=1}^{j-{\cal I}_{v-1}-1}{\bf q}^\top_{i-t+1}{\bf q}_{j-s+1} c^{t+s}\big\{{\cal H}(t+s-2,\lambda_u,\lambda_v,c;1,1,0) \non\\
			&\qquad\qquad\qquad\qquad\qquad\qquad + c(\lambda_u+\lambda_v){\cal H}(t+s-1, \lambda_u, \lambda_v,c;1,1,0)\non\\
            &\qquad\qquad\qquad\qquad\qquad\qquad + c^2\lambda_u\lambda_v{\cal H}(t+s, \lambda_u,\lambda_v,c;1,1,0)\big\} \non\\
			&+\sum\limits_{s=1}^{j-{\cal I}_{v-1}-1}{\bf q}^\top_{i+1}{\bf q}_{j-s+1} c^{s+1}\big\{{\cal H}(s-1,\lambda_v,\lambda_u,c;(1-c^{-1}),(1-c^{-1}),c^{-1}) \non\\
			&\qquad\qquad\qquad\qquad\qquad\qquad+\lambda_v{\cal H}(s,\lambda_v,\lambda_u,c;(c-1),(c-1),1) \big\} \non\\
			&+\sum\limits_{t=1}^{i-{\cal I}_{u-1}-1}{\bf q}^\top_{j+1}{\bf q}_{i-t+1}c^{t+1}\big\{{\cal H}(t-1,\lambda_u,\lambda_v,c;(1-c^{-1}),(1-c^{-1}),c^{-1})\non\\
			&\qquad\qquad\qquad\qquad\qquad\qquad+\lambda_u{\cal H}(t,\lambda_u,\lambda_v,c;(c-1),(c-1),1) \big\} \non\\
			&+{\bf q}^\top_{i+1}{\bf q}_{j+1}\frac{(c-1)(2-\lambda_u-\lambda_v)+(1-\lambda_u)(1-\lambda_v)}{(1-\lambda_u)(1-\lambda_v)[-1+c(2-\lambda_u-\lambda_v)]}.\non
		\end{align}
        The auxiliary functions $\mathcal{H}(\cdot)$, $\mathcal{N}_i(\cdot)$, and $\mathcal{D}_i(\cdot)$, which depend on the eigenvalues ($\lambda_u, \lambda_v$) and the step-size constant $c$, are provided in Appendix \ref{sec0}.
	\end{theorem}

\begin{remark}
The complexity of the asymptotic covariance matrix in Theorem \ref{th5}, particularly in $\widehat{\bm \Sigma}_{\bf NN}$, arises directly from the presence of Jordan blocks of order greater than 1 in $\mathbf{W}$. The off-diagonal elements within these blocks induce coupling between the dynamics associated with generalized eigenvectors for the same eigenvalue. Consequently, the calculation of second moments necessitates tracking these dependencies, leading inherently to the combinatorial terms encapsulated by the auxiliary function $\mathcal{H}(\cdot)$. Thus, $\mathcal{H}(\cdot)$ precisely represents the computational structure emerging from these higher-order Jordan blocks.
\end{remark}

	Under the setting $\gamma=1$, we now turn to the threshold case $\tau=1-(2c)^{-1}$. The following theorem describes the corresponding asymptotic behavior.
	\begin{theorem}[Convergence Rate for $\gamma=1$, $\tau=1-(2c)^{-1}$]\label{th13}
		Under Assumptions \ref{as1} and \ref{as2}, when $N\ge2$, $\gamma=1$, $\tau=1-(2c)^{-1}$, we have
		\begin{equation*}
			\frac{\sqrt{n}}{(\log n)^{\rho-1/2}}\begin{pmatrix}
				{\bf Z}_n-Z_\infty{\bf 1}\\ {\bf N}_n-Z_\infty{\bf 1}
			\end{pmatrix}\to{\cal N}\left({\bf 0},Z_\infty(1-Z_\infty)\begin{pmatrix}
				\widehat{\bm\Sigma}^\ast_{\bf ZZ} & \;\;\widehat{\bm\Sigma}^\ast_{\bf ZN}\\
				\widehat{\bm\Sigma}^{\ast\top}_{\bf ZN} & \;\;\widehat{\bm\Sigma}^\ast_{\bf NN}
			\end{pmatrix}\right)\ \ \ \ \ \ \text{\textit{stably}},
		\end{equation*}
		For any $u, v \in \{1, 2, \dots, T\}$, and for indices $i$, $j$ satisfying ${\cal I}_{u-1}< i \le {\cal I}_u$ and ${\cal I}_{v-1}< j \le {\cal I}_v$, respectively, the matrix $\widehat{\bm\Sigma}^\ast_{\bf ZZ}$ is given by
		\begin{align}\label{eq68}
			&\qquad \widehat{\bm\Sigma}^\ast_{\bf ZZ}={\bf P}\widehat{\bm S}^\ast_{\bf ZZ}{\bf P}^\top,\ \ {\rm and}\\
			& [\widehat{\bm S}^\ast_{\bf ZZ}]_{i,j}=\begin{cases}
				\frac{c^{2\rho}}{2\rho-1}{\bf q}^\top_{i-\rho+2}{\bf q}_{j-\rho+2}, & \text{if}\ (i,j)=({\cal I}_u,{\cal I}_v),\ \rho_u=\rho_v=\rho,\ \text{and}\\
				& \lambda_u+\lambda_v= 2-c^{-1};\\
				0, & \text{for}\ (i,j)\neq({\cal I}_u,{\cal I}_v),\ \text{or}\ \rho_u\rho_v<\rho^2,\ \text{or}\\
				& \lambda_u+\lambda_v\neq 2-c^{-1}. \non
			\end{cases}
		\end{align}
		The matrix $\widehat{\bm\Sigma}^\ast_{\bf ZN}$ is given by
		\begin{align}\label{eq67}
			&\qquad\widehat{\bm \Sigma}^\ast_{\bf ZN}={\bf P}\widehat{\bf S}^\ast_{\bf ZN}\widetilde{\bf P}^\top,\ \ {\rm and}\\
			&[\widehat{\bm S}^\ast_{\bf ZN}]_{i,1}=0,\non\\
			&[\widehat{\bm S}^\ast_{\bf ZN}]_{i,j+1}=\begin{cases}
				\frac{c^{2\rho-1}}{2\rho-1}\frac{\lambda_v{\bf q}^\top_{i-\rho+2}{\bf q}_{j-\rho+2}}{1-\lambda_u}, & \text{for}\ (i,j)=({\cal I}_u,{\cal I}_v),\ \rho_u=\rho_v=\rho,\ \text{and}\\ &\lambda_u+\lambda_v= 2-c^{-1};\\
				0, & \text{for}\ (i,j)\neq({\cal I}_u,{\cal I}_v),\ \text{or}\ \rho_u\rho_v<\rho^2,\  \text{or}\\
				&\lambda_u+\lambda_v\neq 2-c^{-1}. \non
			\end{cases}
		\end{align}
		And the matrix $\widehat{\bm\Sigma}^\ast_{\bf NN}$ is given by
		\begin{align}\label{eq69}
			&\qquad\widehat{\bm \Sigma}^\ast_{\bf NN}=\widetilde{\bf P}\widehat{\bf S}^\ast_{\bf SN}\widetilde{\bf P}^\top,\ \ {\rm and}\\
			&[\widehat{\bm S}^\ast_{\bf NN}]_{1,1}=0,~~~[\widehat{\bm S}^\ast_{\bf NN}]_{1,j+1}=[\widehat{\bm S}^\ast_{\bf NN}]_{j+1,1}=0,\non\\
			&[\widehat{\bm S}^\ast_{\bf NN}]_{i+1,j+1}=\begin{cases}
				\frac{c^{2\rho-2}}{2\rho-1}\frac{\lambda_u\lambda_v{\bf q}^\top_{i-\rho+2}{\bf q}_{j-\rho+2}}{(1-\lambda_u)(1-\lambda_v)}, & \text{for}\ (i,j)=({\cal I}_u,{\cal I}_v),\ \rho_u=\rho_v=\rho,\ \text{and}\\
				& \lambda_u+\lambda_v= 2-c^{-1};\\
				0, & \text{for}\ (i,j)\neq({\cal I}_u,{\cal I}_v),\ \text{\textit{or}}\ \rho_u\rho_v<\rho^2,\ \text{or}\\ &\lambda_u+\lambda_v\neq 2-c^{-1}. \non
			\end{cases}
		\end{align}
	\end{theorem}

	\begin{remark}
		Theorems \ref{th3}, \ref{th5} and \ref{th13} establish that the convergence rate of the stochastic process $({\bf Z}_n,{\bf N}_n)_n$ depends on both the order of the step size $\gamma$ and the second-largest real part of the eigenvalues of the adjacency matrix $\mathbf{W}$. When $\prod_{i=2}^{T} \rho_i=1$, corresponding to Jordan blocks of order one and hence a diagonalizable ${\bf W}$,
 the convergence rates and covariance structures 
coincide with those in Theorems 3.2, 3.4, and 3.5 of \cite{r2}. In contrast, if $\prod_{i=2}^{T}\rho_i>1$, reflecting the non-diagonalizability of {\bf W}, the asymptotic covariance matrices differ across all regimes considered, and the convergence rate is also affected in the case $\gamma=1$, $\tau=1-(2c)^{-1}$. These differences arise because the off-diagonal entries in the Jordan decomposition of ${\bf W}$ contribute additional structural components to the covariance matrix.
	\end{remark}

	To formulate the pairwise synchronization rate, we introduce the following notation. 
	Let $[{\bf P}]_{i,\cdot}$ and $[{\bf P}]_{j,\cdot}$ denote the $i$th and $j$th rows of the matrix ${\bf P}$, respectively. Define
	\begin{align*}
		&{\bf p}_{i,j}=({\bf e}^\top_i-{\bf  e}^\top_j){\bf P}=[{\bf P}]_{i,\cdot}-[{\bf P}]_{j,\cdot},\ \ \text{and}\\
		&\widetilde{\bf p}_{i,j}=({\bf  e}^\top_i-{\bf  e}^\top_j)\widetilde{\bf P}=(0,[{\bf P}]_{i,\cdot}-[{\bf P}]_{j,\cdot}).
	\end{align*}
	Within these notations, the theorem below characterizes the synchronization rate between any two agents in the population ${\cal G}$.

	\begin{theorem}[Synchronization Rate]\label{th4}
		Under Assumptions \ref{as1} and \ref{as2}, for any $i,j\in\{1,2,\dots,N\}$, $i\neq j$, it holds that:
		
		{\rm (a)} If $1/2<\gamma<1$, then stably
		\begin{equation*}
			n^{\frac{\gamma}{2}}(Z_{ni}-Z_{nj})\to{\cal N}(0,Z_\infty(1-Z_\infty){\bm\Sigma}_{\gamma,i,j}),
		\end{equation*}
		where for any $u, v \in \{1, 2, \dots, T\}$, and for indices $i$, $j$ satisfying ${\cal I}_{u-1}< i \le {\cal I}_u$ and ${\cal I}_{v-1}< j \le {\cal I}_v$, respectively, the element
		${\bm\Sigma}_{\gamma,i,j}=[\widehat{\bm\Sigma}_{\gamma}]_{i,i}+[\widehat{\bm\Sigma}_{\gamma}]_{j,j}-2[\widehat{\bm\Sigma}_{\gamma}]_{i,j}$, and the matrix $\widehat{\bm \Sigma}_\gamma$ is given by
		\begin{align}\label{eq64}
			&\widehat{\bm\Sigma}_\gamma={\bf P}\widehat{\bf S}_\gamma{\bf P}^\top,\ \ {\rm and}\\
			&	[\widehat{\bf S}_\gamma]_{i,j}=\sum_{s=0}^{j-{\cal I}_{v-1}-1} \sum_{t=0}^{i-{\cal I}_{u-1}-1} \frac{c(t+s)!}{(2 -\lambda_u-\lambda_v)^{t+s+1}}{\bf q}^\top_{i-t+1}{\bf q}_{j-s+1}.\non
		\end{align}

		{\rm (b)} If $\gamma=1$ and $\tau<1-(2c)^{-1}$, then  stably
		\begin{equation*}
			\sqrt{n}\begin{pmatrix}
				Z_{ni}-Z_{nj}\\
				N_{ni}-N_{nj}
			\end{pmatrix}\to{\cal N}\left(0,Z_\infty(1-Z_\infty)\begin{pmatrix}
				{\bf p}_{i,j}\widehat{\bf S}_{\bf ZZ}{\bf p}^\top_{i,j} & \;\;{\bf p}_{i,j}\widehat{\bf S}_{\bf ZN}\widetilde{\bf p}^\top_{i,j}\\
				\widetilde{\bf p}_{i,j}\widehat{\bf S}^\top_{\bf ZN}{\bf p}^\top_{i,j} & \;\;\widetilde{\bf p}_{i,j}\widehat{\bf S}_{\bf NN}\widetilde{\bf p}^\top_{i,j}
			\end{pmatrix}\right).
		\end{equation*}

		{\rm (c)} If $\gamma=1$, $\tau=1-(2c)^{-1}$, then stably
		\begin{equation*}
			{\frac{\sqrt{n}}{(\log n)^{\rho-1/2}}}\begin{pmatrix}
				Z_{ni}-Z_{nj}\\
				N_{ni}-N_{nj}
			\end{pmatrix}\to{\cal N}\left(0,Z_\infty(1-Z_\infty)\begin{pmatrix}
				{\bf p}_{i,j}\widehat{\bf S}^\ast_{\bf ZZ}{\bf p}^\top_{i,j} & \;\;{\bf p}_{i,j}\widehat{\bf S}^\ast_{\bf ZN}\widetilde{\bf p}^\top_{i,j}\\
				\widetilde{\bf p}_{i,j}\widehat{\bf S}^{\ast\top}_{\bf ZN}{\bf p}^\top_{i,j} & \;\;\widetilde{\bf p}_{i,j}\widehat{\bf S}^\ast_{\bf NN}\widetilde{\bf p}^\top_{i,j}
			\end{pmatrix}\right).
		\end{equation*}
	\end{theorem}
	
	The phase transition observed in the convergence rates is not an isolated phenomenon. Although our process ${\bf Z}_n$ arises from an interacting multi-agent network, its associated stochastic approximation form \eqref{eq3} shares a deep structural resemblance with those found in the classical literature on single-urn models, most notably the Generalized Friedman's Urn (GFU) (\cite{r25, r27, r28}). Notably, the GFU model is known to satisfy the stochastic approximation algorithm
	\begin{equation*}
		{\bf Z}_{n+1}-{\bf Z}_n = -\frac{{\bf I}-{\bf H}}{n+1}{\bf Z}_n + \frac{\Delta {\bf M}_{n+1}}{n+1} + \frac{{\bf r}_{n+1}}{n+1},
	\end{equation*}
	where $(\Delta {\bf M}_{n+1})_n$ is a martingale difference sequence and $({\bf r}_{n+1})_n$ is a remainder sequence. This iterative structure closely matches \eqref{eq3} in this paper,
	\begin{equation*}
		{\bf Z}_{n+1}-{\bf Z}_n=-r_n({\bf I}-{\bf W}^\top){\bf Z}_n + r_n \Delta {\bf M}_{n+1},
	\end{equation*}
	especially under the setting $r_n\sim\frac{1}{n}$ (i.e., $\gamma=1$ and $c=1$), where the two forms appear to be highly consistent. Moreover, both the replacement matrix ${\bf H}$ in \cite{r25} and the adjacency matrix ${\bf W}$ in the present paper are allowed to be non-diagonalizable. Theorem 3.2 in \cite{r25} establishes that the second-order asymptotic behavior of the normalized urn composition ${\bf Z}_n$ depends critically on the spectral properties of ${\bf H}$. Defining $\tau$ as the second-largest real part among the eigenvalues, the results show that when $\tau<1/2$, the convergence rate is $\sqrt{n}$. When $\tau=1/2$, the rate becomes ${\sqrt{n}}/{(\log n)^{\rho-1/2}}$. The proposed Theorems \ref{th5} and \ref{th13} demonstrate analogous behavior in both convergence rates and covariance matrix structure, corresponding precisely to these two cases. This striking parallelism strongly suggests that the observed consistency stems from the shared mathematical structure of the underlying stochastic approximation processes.

	In conclusion, the top-down influence dynamic is the essential feature that distinguishes these hierarchical systems from standard irreducible ones. While the leading group exclusively determines the synchronization limit, our second-order results demonstrate that the structure of the downstream groups and their connections governs the rate and path by which the synchronization is reached.

	\section{Proofs of Main Results}\label{sec4}
	In this section, we provide detailed proofs of the main theoretical results presented in this paper.
	\subsection{Proof Framework}\label{sec4.1}
	Recall the following update process of $({\bf Z}_n,{\bf N}_n)_n$,
	\begin{equation*}
		\begin{cases}
			{\bf Z}_{n+1}=(1-r_{n}){\bf Z}_{n}+r_{n}{\bf X}_{n+1},\\
			{\bf N}_{n+1}=\left(1-\frac{1}{n+1}\right){\bf N}_n+\frac{1}{n+1}{\bf X}_{n+1}.
		\end{cases}
	\end{equation*}
	For the stochastic process $({\bf Z}_n)_n$, it follows that
	\begin{align}\label{eq6}
		{\bf Z}_{n+1}-{\bf Z}_n&=-r_n{\bf Z}_n+r_n{\bf X}_{n+1}\non\\
		&=-r_n{\bf Z}_n+r_n{\mathbb E}({\bf X}_{n+1}|{\cal F}_n)+r_n[{\bf X}_{n+1}-{\mathbb E}({\bf X}_{n+1}|{\cal F}_n)]\non\\
		&=-r_n{\bf Z}_n+r_n{\bf W}^\top{\bf Z}_n+r_n\Delta{\bf M}_{n+1}\non\\
		&=-r_n({\bf I}-{\bf W}^\top){\bf Z}_n+r_n\Delta{\bf M}_{n+1},
	\end{align}
	where $\Delta{\bf M}_{n+1}={\bf X}_{n+1}-{\mathbb E}({\bf X}_{n+1}|{\cal F}_n)$, and $(\Delta{\bf M}_n)_n$ is a martingale difference sequence. Since ${\bf q}^\top_1({\bf I}-{\bf W}^\top)={\bf 0}$, we obtain
	\begin{equation}\label{eq16}
		{\bf q}^\top_1{\bf Z}_{n+1}-{\bf q}^\top_1{\bf Z}_{n}=r_n{\bf q}^\top_1\Delta{\bf M}_{n+1}.
	\end{equation}
	Therefore, the sequence $({\bf q}^\top_1{\bf Z}_n)_n$ forms a martingale. Recalling that ${\bf I}={\bf p}_1{\bf q}^\top_1+{\bf P}{\bf Q}^\top$, we may decompose ${\bf Z}_n$ as
	\begin{align}\label{eq13}
		{\bf Z}_n={\bf p}_1{\bf q}^\top_1{\bf Z}_n+{\bf P}{\bf Q}^\top{\bf Z}_n=N^{-1/2}{\bf q}^\top_1{\bf Z}_n{\bf 1}+{\bf P}{\bf Q}^\top{\bf Z}_n=\widetilde{Z}_n{\bf 1}+\widehat{\bf Z}_n,
	\end{align}
	where
	\begin{equation*}
		\widetilde{Z}_n:=N^{-1/2}{\bf q}^\top_1{\bf Z}_n,\ \ \widehat{\bf Z}_n:={\bf P}{\bf Q}^\top{\bf Z}_n,
	\end{equation*}
	while we decompose the stochastic process $({\bf N}_n)_n$ as
	\begin{equation}\label{eq61}
		{\bf N}_{n}=\widetilde{Z}_n{\bf 1}+\widehat{\bf N}_n,\ \ {\rm with}\ \widehat{\bf N}_n:={\bf N}_{n}-\widetilde{Z}_n{\bf 1}.
	\end{equation}
	Based on the decompositions in \eqref{eq13} and \eqref{eq61}, we aim to establish the first- and second-order asymptotic properties of ${\bf Z}_n$ by analyzing $\widetilde{Z}_n$, $\widehat{\bf Z}_n$ and $\widehat{\bf N}_n$.
	To establish the first-order convergence of $({\bf Z}_n, {\bf N}_n)_n$ in Theorem \ref{th7}, we begin by proving the following result for the convergence of ${\bf Z}_n$. The convergence of ${\bf N}_n$ then follows from that of ${\bf Z}_n$ via the recursive relation linking them.
	
	\begin{theorem}\label{th10}
		Under Assumptions \ref{as1} and \ref{as2}, there exists a random variable $Z_\infty$ taking values in $[0,1]$ such that
		\begin{equation*}
			\widetilde{Z}_n\overset{a.s.}{\to}Z_\infty,\qquad \widehat{\bf Z}_n\overset{a.s.}{\to}{\bf 0}.
		\end{equation*}
	\end{theorem}

	To characterize the second-order asymptotic behavior of the sequence $({\bf Z}_n, {\bf N}_n)_n$, we prove the asymptotic normality of the relevant processes $\widetilde{Z}_n$, $\widehat{\bf Z}_n$, and $\widehat{\bf N}_n$, as established in Theorems \ref{th8}, \ref{th9}, and \ref{th2}. 
	\begin{theorem}\label{th8}
		Under Assumptions \ref{as1} and \ref{as2}, for $1/2<\gamma\le 1$, we have
		\begin{equation*}
			n^{\gamma-\frac{1}{2}}(\widetilde{Z}_n-Z_\infty)\to{\cal N}(0,Z_\infty(1-Z_\infty)\widetilde{\sigma}^2_\gamma)\ \ \ \ \ \ \text{\textit{stably}},
		\end{equation*}
		where $\widetilde{\sigma}^2_\gamma$ is given by \eqref{eq4}.
	\end{theorem}

	\begin{theorem}\label{th9}
		Under Assumptions \ref{as1} and \ref{as2}, when $1/2<\gamma<1$, it holds that:
		
		{\rm (a)} The stochastic process $(\widehat{{\bf Z}}_n)_n$ satisfies
		\begin{equation*}
			n^{\gamma/2}\widehat{{\bf Z}}_n\to {\cal N}({\bf 0},Z_\infty(1-Z_\infty)\widehat{\bm \Sigma}_\gamma)\ \ \ \ \ \ \text{\textit{stably}},
		\end{equation*}
		where $\widehat{\bm \Sigma}_\gamma$ ia given by \eqref{eq64}.

		{\rm (b)} The stochastic process $(\widehat{{\bf N}}_n)_n$ satisfies
		\begin{equation*}
			n^{\gamma-\frac{1}{2}}\widehat{{\bf N}}_n\to {\cal N}({\bf 0},Z_\infty(1-Z_\infty)\widehat{\bm \Gamma}_\gamma)\ \ \ \ \ \ \text{\textit{stably}},
		\end{equation*}
		where $\widehat{\bm \Gamma}_\gamma$ is given by \eqref{eq63}.
	\end{theorem}

	\begin{theorem}\label{th2}
		Under Assumptions \ref{as1} and \ref{as2}, when $\gamma=1$, it holds that:
		
		{\rm (a)} When $\tau<1-(2c)^{-1}$,
		\begin{equation*}
			\sqrt{n}\begin{pmatrix}
				\widehat{{\bf Z}}_n\\ \widehat{{\bf N}}_n
			\end{pmatrix}\to{\cal N}\left({\bf 0},Z_\infty(1-Z_\infty)\begin{pmatrix}
				\widehat{\bm\Sigma}_{\bf ZZ} &\;\; \widehat{\bm\Sigma}_{\bf ZN}\\
				\widehat{\bm\Sigma}^\top_{\bf ZN} & \;\;\widehat{\bm\Sigma}_{\bf NN}
			\end{pmatrix}\right)\ \ \ \ \ \ \text{\textit{stably}},
		\end{equation*}
		where $\widehat{\bm\Sigma}_{\bf ZZ},\ \widehat{\bm\Sigma}_{\bf ZN}$ and $\widehat{\bm\Sigma}_{\bf NN}$ are given by \eqref{eq1}, \eqref{eq65} and \eqref{eq66}, respectively.

		{\rm (b)} When $\tau=1-(2c)^{-1}$,
		\begin{equation*}
			\frac{\sqrt{n}}{(\log n)^{\rho-1/2}}\begin{pmatrix}
				\widehat{{\bf Z}}_n\\ \widehat{{\bf N}}_n
			\end{pmatrix}\to{\cal N}\left({\bf 0},Z_\infty(1-Z_\infty)\begin{pmatrix}
				\widehat{\bm\Sigma}^\ast_{\bf ZZ} & \;\;\widehat{\bm\Sigma}^\ast_{\bf ZN}\\
				\widehat{\bm\Sigma}^{\ast\top}_{\bf ZN} &\;\; \widehat{\bm\Sigma}^\ast_{\bf NN}
			\end{pmatrix}\right)\ \ \ \ \ \ \text{\textit{stably}},
		\end{equation*}
		where $\widehat{\bm\Sigma}^\ast_{\bf ZZ},\ \widehat{\bm\Sigma}^\ast_{\bf ZN}$ and $\widehat{\bm\Sigma}^\ast_{\bf NN}$ are given by \eqref{eq68}, \eqref{eq67} and \eqref{eq69}, respectively.
	\end{theorem}

	Assuming the validity of Theorem \ref{th10}, the second-order convergence analysis of $\widetilde{Z}_n$, $\widehat{\mathbf{Z}}_n$ and $\widehat{\mathbf{N}}_n$ in Theorems \ref{th8}--\ref{th2} requires handling of the conditional second moment properties of the martingale difference sequence. For clarity in the subsequent proof, we first establish the following foundational results:
	\begin{equation}\label{eq19}
		{\mathbb E}[(\Delta{\bf M}_{n+1})(\Delta{\bf M}_{n+1})^\top|{\cal F}_n]\overset{a.s.}{\to}Z_\infty(1-Z_\infty){\bf I}.
	\end{equation}
	Recall that $\Delta{\bf M}_{n+1}={\bf X}_{n+1}-{\mathbb E}({\bf X}_{n+1}|{\cal F}_n)$ and ${\mathbb E}({\bf X}_{n+1}|{\cal F}_n)={\bf W}^\top{\bf Z}_n$, we have
	\begin{align*}
		{\mathbb E}[(\Delta{\bf M}_{n+1})(\Delta{\bf M}_{n+1})^\top|{\cal F}_n]=&{\mathbb E}\{[{\bf X}_{n+1}-{\mathbb E}({\bf X}_{n+1}|{\cal F}_n)][{\bf X}_{n+1}-{\mathbb E}({\bf X}_{n+1}|{\cal F}_n)]^\top|{\cal F}_n\}\\
		=&{\mathbb E}({\bf X}_{n+1}{\bf X}^\top_{n+1}|{\cal F}_n)-{\mathbb E}({\bf X}_{n+1}|{\cal F}_n){\mathbb E}({\bf X}^\top_{n+1}|{\cal F}_n).
	\end{align*}
	For all distinct pairs $i,j \in \{1,\ldots,N\}$, the off-diagonal entries of ${\mathbb E}[(\Delta{\bf M}_{n+1})(\Delta{\bf M}_{n+1})^\top|{\cal F}_n]$ satisfy
	\begin{equation*}
		{\mathbb E}(X_{n+1,i}X_{n+1,j}|{\cal F}_n)-{\mathbb E}(X_{n+1,i}|{\cal F}_n){\mathbb E}(X_{n+1,j}|{\cal F}_n)=0.
	\end{equation*}
	For diagonal entries $i \in\{1,\ldots,N\}$, 
	\begin{equation*}
		{\mathbb E}(X^2_{n+1,i}|{\cal F}_n)-{\mathbb E}^2(X_{n+1,i}|{\cal F}_n)=\sum\limits_{h=1}^Nw_{h,j}Z_{n,h}-\bigg(\sum\limits_{h=1}^Nw_{h,j}Z_{n,h}\bigg)^2\overset{a.s.}{\to}Z_\infty(1-Z_\infty),
	\end{equation*}
	where the convergence holds by $\sum_{h=1}^{N}w_{h,j}=1$ and $Z_{n,j}\overset{a.s.}{\to}Z_\infty$ for all $j\in\{1,2,\cdots,N\}$.

	\subsection{Detailed Proof}\label{sec4.2}
	In this section, we present the proofs of Theorems \ref{th7}, \ref{th11}, \ref{th1}, \ref{th3}, \ref{th5}, \ref{th13}, \ref{th4} and Corollary \ref{co1}.
	
	\begin{proof}[Proof of Theorem \ref{th7}]
		To prove Theorem \ref{th7}, we first establish Theorem \ref{th10}. We begin by proving the first part of Theorem \ref{th10}, which concerns the convergence of $\widetilde{Z}_n$. Since ${\bf Z}_0 \in [0,1]^N$ and ${\bf Z}_n$ satisfies the recurrence relation \eqref{eq15}, it follows that ${\bf Z}_n \in [0,1]^N$ for all $n$. Note that ${\bf q}_1^{\top}{\bf p}_1=1$ and $\mathbf{p}_1=N^{-1/2}\mathbf{1}$, so we have $N^{-1/2}\mathbf{q}_1^{\top}\mathbf{1}=1$. Since the components of $\mathbf{q}_1$ are non-negative, $N^{-1/2} \mathbf{q}_1^{\top}$ can be interpreted as a weight vector. Therefore, for all $n$, we have $\min_hZ_{nh}\leq \widetilde{Z}_n=$ $N^{-1/2} \mathbf{q}_1^{\top} \mathbf{Z}_n \leq \max_hZ_{nh}$, which implies $\widetilde{Z}_n \in[0,1]$. From \eqref{eq16}, we have
		\begin{equation*}
			\widetilde{Z}_{n+1}-\widetilde{Z}_n=r_n N^{-1/2} \mathbf{q}_1^{\top} \Delta \mathbf{M}_{n+1},
		\end{equation*}
		thus, $\widetilde{Z}_n$ is a bounded martingale that converges almost surely to a random variable $Z_{\infty}$ taking values in $[0,1]$.

		To prove the second part of Theorem \ref{th10}, concerning the almost sure convergence $\widehat{\bf Z}_n\overset{a.s.}{\to} {\bf 0}$, we proceed as follows. By Lemma \ref{le7}, there exists an invertible block-diagonal matrix ${\bf D}_\beta$ such that
		\begin{equation}\label{eq21}
			{\bf Q}_\beta = {\bf Q}{\bf D}_\beta, \ \  {\bf W}{\bf Q}_\beta = {\bf Q}_\beta{\bf J}_\beta,\ \  {\rm and}\ \ \|{\bf J}_\beta\|_2 \leq \frac{1+\max_{s\in\{1,\cdots,T\}}|\lambda_s|}{2}<1,
		\end{equation}
		where ${\bf J}_\beta$ is associated with ${\bf D}_\beta$. Then, we have $\widehat{\bf Z}_n={\bf P}{\bf Q}^\top{\bf Z}_n={\bf P}({\bf D}^{-1}_\beta)^\top{\bf Q}^\top_\beta{\bf Z}_n$. Defining ${\bf Z}_{{\bf Q}_\beta,n}={\bf Q}^\top_\beta{\bf Z}_n$, it suffices to show that ${\bf Z}_{{\bf Q}_\beta,n}\overset{a.s.}{\to}0$ to conclude $\widehat{\bf Z}_n\overset{a.s.}{\to}0$. 
		Applying left multiplication by ${\bf Q}^\top_\beta$ to both sides of \eqref{eq6} yields
		\begin{align}\label{eq8}
			{\bf Z}_{{\bf Q}_\beta,n+1}-{\bf Z}_{{\bf Q}_\beta,n}&=-r_n({\bf Q}^\top_\beta-{\bf Q}^\top_\beta{\bf W}^\top){\bf Z}_n+r_n{\bf Q}^\top_\beta\Delta{\bf M}_{n+1}\non\\
			&=-r_n({\bf Q}^\top_\beta-{\bf J}^\top_\beta{\bf Q}^\top_\beta){\bf Z}_n+r_n{\bf Q}^\top_\beta\Delta{\bf M}_{n+1}\non\\
			&=-r_n({\bf I}-{\bf J}^\top_\beta){\bf Z}_{{\bf Q}_\beta,n}+r_n{\bf Q}^\top_\beta\Delta{\bf M}_{n+1}.
		\end{align}
		It follows that
		\begin{align}
			&\mathbb{E}[\|{\bf Z}_{{\bf Q}_\beta,n+1}\|^2|{\cal F}_n]=\mathbb{E}\big[\overline{\bf Z}^\top_{{\bf Q}_\beta,n+1}{\bf Z}_{{\bf Q}_\beta,n+1}|{\cal F}_n\big]\non\\
			=&\overline{\bf Z}^\top_{{\bf Q}_\beta,n}\big({\bf I}-r_n({\bf I}-\overline{\bf J}_\beta)\big)\big({\bf I}-r_n({\bf I}-{\bf J}^\top_\beta)  \big){\bf Z}_{{\bf Q}_\beta,n}+r^2_n\mathbb{E}\big[\Delta{\bf M}^\top_{n+1}\overline{\bf Q}_\beta{\bf Q}^\top_\beta\Delta{\bf M}_{n+1}  |{\cal F}_n\big]\non\\
			=&\|\big((1-r_n){\bf I}+r_n{\bf J}^\top_\beta\big) {\bf Z}_{{\bf Q}_\beta,n}\|^2+r^2_n\xi_n\non   \\
			\le&\big((1-r_n)+r_n\|{\bf J}_{\beta}\|_2  \big)^2\|{\bf Z}_{{\bf Q}_\beta,n}\|^2+r^2_n\xi_n   \label{eq22}\\
			\le&\left[1-\left(1-\frac{1+\max_{s\in\{1,\cdots,T\}}|\lambda_s|}{2}\right)r_n\right]^2\|{\bf Z}_{{\bf Q}_\beta,n}\|^2+r^2_n\xi_n   \label{eq23}\\
			\le&(1-ar_n)\|{\bf Z}_{{\bf Q}_\beta,n}\|^2+r^2_n\xi_n,   \label{eq24}
		\end{align}
		where $\xi_n$ is an $\mathcal{F}_n$-measurable bounded random variable, and $a=(1-\max_{s\in\{1,\cdots,T\}}|\lambda_s|)/2>0$. The inequality \eqref{eq22} follows from the submultiplicative property of matrix norms, 
		\eqref{eq23} is obtained directly from \eqref{eq21}, and \eqref{eq24} holds by Assumption \ref{as2}. Hence, there exists a constant $C$ such that
		\begin{equation}\label{eq25}
			\mathbb{E}[\|{\bf Z}_{{\bf Q}_\beta,n+1}\|^2|{\cal F}_n]\le (1-ar_n)\|{\bf Z}_{{\bf Q}_\beta,n}\|^2+Cr^2_n.
		\end{equation}
		Since $\sum_{n=1}^\infty r^2_n<+\infty$, it follows from \cite{r8} that $(\|\mathbf{Z}_{\mathbf{Q}_{\beta,n}}\|)_n$ 
		forms an almost supermartingale and therefore converges almost surely to a finite random variable. Taking expectations on both sides of \eqref{eq23} yields
		\begin{equation*}
			\mathbb{E}\|{\bf Z}_{{\bf Q}_\beta,n+1}\|^2\le (1-ar_n)\mathbb{E}\|{\bf Z}_{{\bf Q}_\beta,n}\|^2+Cr^2_n,
		\end{equation*}
		Furthermore, since $\sum_{n=1}^\infty r_n=+\infty$, by Lemma \ref{le9}, $\mathbb{E}\|{\bf Z}_{{\bf Q}_\beta,n}\|^2$ converges to $0$. Combining this with the almost sure convergence of $\|{\bf Z}_{{\bf Q}_\beta,n}\|^2$, we conclude that
		\begin{equation*}
			\|{\bf Z}_{{\bf Q}_\beta,n}\|^2\overset{a.s.}{\to}0,\ \ \text{ and\ consequently,}\ \ {\bf Z}_{{\bf Q}_\beta,n}\overset{a.s.}{\to}{\bf 0}.
		\end{equation*}
		Thus, Theorem~\ref{th10} is proved, and it follows that ${\bf Z}_n$ converges almost surely to $Z_\infty {\bf 1}$.
		
		We next establish the almost sure convergence of ${\bf N}_n$. Recall the recursive relation
		\begin{equation*}
			{\bf N}_n=\frac{1}{n}\sum\limits_{k=1}^n {\bf X}_k,
		\end{equation*}
		and note that
		\begin{equation*}
			\mathbb{E}({\bf X}_k|{\cal F}_{k-1})={\bf W}^\top{\bf Z}_k\overset{a.s.}{\to}Z_\infty{\bf W}^\top{\bf 1}=Z_\infty{\bf 1}.
		\end{equation*}
		Applying Lemma \ref{le10} with $Y_k=X_{k,h}$, $v_{n,k}=\frac{1}{n}$ and $c_k = 1$ for all $h\in\{1, 2, \dots, N\}$, we then obtain that ${\bf N}_n$ converges to $Z_\infty {\bf 1}$ almost surely.
	\end{proof}

	\begin{proof}[Proof of Theorem \ref{th11} and Corollary \ref{co1}]
		We analyze the properties of the dominant right eigenvector $\mathbf{q}_1$ of the adjacency matrix $\mathbf{W}$. By partitioning $\mathbf{q}_1$ according to the block structure of the columns of $\mathbf{W}$, we write ${\bf q}_1 = ({\bf q}^\top_{11},{\bf q}^\top_{12},\dots,{\bf q}^\top_{1S})^\top$. Since $\mathbf{W}$ is a block upper triangular matrix and 
		$\max_{h\in\{2,\cdots,S\}}\|{\bf W}_{hh}\|_1<1$, it follows from the eigenvalue equation $\mathbf{W} \mathbf{q}_1=\mathbf{q}_1$ that
		\begin{equation}\label{eq7}
			{\bf W}_{11}{\bf q}_{11}={\bf q}_{11},\ \ {\rm and}\ {\bf q}_{12}=\cdots={\bf q}_{1S}={\bf 0}.
		\end{equation}
		From the almost sure convergence of $\widetilde{Z}_n$, we further obtain
		\begin{equation*}
			N^{-1/2}{\bf q}^\top_{11}{\bf Z}^{(1)}_n\overset{a.s.}{\to}Z_\infty,
		\end{equation*}
		where ${\bf Z}^{(1)}_n$ represents the agent inclination corresponding to the subgroup ${\cal G}_1$. Recalling the update equations \eqref{eq49}--\eqref{eq15}, and focusing on ${\bf Z}^{(1)}_n$, we have
		\begin{equation*}
			{\bf Z}^{(1)}_n=(1-r_{n-1}){\bf Z}^{(1)}_{n-1}+r_{n-1}{\bf X}^{(1)}_n,\ \ {\rm and}\ \mathbb{P}({\bf X}^{(1)}_{n}={\bf 1}|{\cal F}_{n-1})={\bf W}^\top_{11}{\bf Z}^{(1)}_{n-1},
		\end{equation*}	
		where ${\bf X}^{(1)}_n$ represents the decisions of agents in subgroup ${\cal G}_1$ at time $n$. Consequently, the process $({\bf Z}^{(1)}_n)_n$ depends exclusively on interactions within ${\cal G}_1$, implying that the limiting variable $Z_\infty$ is determined entirely by subgroup ${\cal G}_1$, while the influence of other subgroups ${\cal G}_k$ ($k \in {2, \dots, S}$) vanishes asymptotically.
		
		Now we prove Corollary \ref{co1}. By Lebesgue's dominated convergence theorem, we have
		\begin{align*}
			{\mathbb E}(Z_\infty)=&{\mathbb E}(\lim\limits_{n\to\infty}\widetilde{Z}_n)={\mathbb E}(\lim\limits_{n\to\infty}N^{-1/2}{\bf q}^\top_1{\bf Z}_n)=\lim\limits_{n\to\infty}{\mathbb E}(N^{-1/2}{\bf q}^\top_1{\bf Z}_n)\\
			=&\lim\limits_{n\to\infty}{\mathbb E}(N^{-1/2}{\bf q}^\top_{11}{\bf Z}^{(1)}_n)=N^{-1/2}{\bf q}^\top_{11}\mathbb{E}({\bf Z}^{(1)}_0),
		\end{align*}
		where the last equality follows from the martingale property of the sequence $({\bf q}^\top_{11}{\bf Z}^{(1)}_n)_n$. This completes the proof.
	\end{proof}

	\begin{proof}[Proof of Theorem \ref{th1}]	
		Recalling the decomposition ${\bf Z}_n=\widetilde{Z}_n{\bf 1}+\widehat{\bf Z}_n$, which is consistent with that in \cite{r1}.  According to Theorem \ref{th10}, $\widetilde{Z}_n\overset{a.s.}{\to}Z_\infty$ and $\widehat{\bf Z}_n\overset{a.s.}{\to}0$, implying that the synchronization limit $Z_\infty$ is entirely determined by the dynamics of $\widetilde{Z}_n$. Theorems 3.5 and 3.6 in \cite{r1} establish two properties of $Z_\infty$ based solely on the structure of $\widetilde{Z}_n$, and their validity does not rely on the diagonalizability of ${\bf W}$. These results thus remain applicable in the current setting. Detailed proofs are omitted.
	\end{proof}

	Next, we establish the second-order convergence of $({\bf Z}_n,{\bf N}_n)_n$, as stated in Theorems \ref{th3}, \ref{th5}, \ref{th13} and \ref{th4}. To this end, we first prove Theorems \ref{th8}, \ref{th9} and \ref{th2} separately.
	\begin{proof}[Proof of Theorem \ref{th8}]
		Theorem \ref{th8} parallels Lemma 4.1 in \cite{r1}, which characterizes the second-order asymptotic behavior of $\widetilde{Z}_n$. The proof of Lemma 4.1 relies solely on the explicit form of $\widetilde{Z}_n$. Since the expression of $\widetilde{Z}_n$ in our setting is identical to that in \cite{r1}, the corresponding arguments remain valid in the present context.
	\end{proof}

	\begin{proof}[Proof of Theorems \ref{th9}]
		We begin by proving part (a) of Theorem \ref{th9}, which establishes the almost sure convergence of $\widehat{\bf Z}_n$. 
		Since ${\bf P}{\bf Q}^\top\widehat{\bf Z}_n=\widehat{\bf Z}_n$, left-multiplying both sides of \eqref{eq8} by ${\bf P}$ yields
		\begin{align}\label{eq9}
			\widehat{\bf Z}_{n+1}&=\widehat{\bf Z}_n-r_n{\bf P}({\bf I}-{\bf J}^\top){\bf Q}^\top\widehat{\bf Z}_n+r_n{\bf P}{\bf Q}^\top\Delta{\bf M}_{n+1}\non\\
			&=[{\bf I}-r_n{\bf P}({\bf I}-{\bf J}^\top){\bf Q}^\top]\widehat{\bf Z}_n+r_n{\bf P}{\bf Q}^\top\Delta{\bf M}_{n+1}\non\\
			&={\bf P}[{\bf I}-r_n({\bf I}-{\bf J}^\top)]{\bf Q}^\top\widehat{\bf Z}_n+r_n{\bf P}{\bf Q}^\top\Delta{\bf M}_{n+1}.
		\end{align}
		Iterating this relation for $n\ge m_0$, where $m_0$ is chosen sufficiently large such that $(1-\tau)r_j<1/2$ for all $j>m_0$, we obtain
		\begin{equation}\label{eq10}
			\widehat{\bf Z}_{n+1}={\bf A}_{m_0,n}\widehat{\bf Z}_{m_0}+\sum\limits_{k=m_0}^n {\bf A}_{k+1,n}{\bf B}_k,
		\end{equation}
		where the matrices ${\bf A}_{k+1,n}$ and ${\bf B}_k$ are given by
		\begin{equation}\label{eq2}
			{\bf A}_{k+1,n}={\bf P}\prod\limits_{j=k+1}^n[{\bf I}-r_j({\bf I}-{\bf J}^\top)]{\bf Q}^\top,\ \ {\bf A}_{n+1,n}={\bf I},\ \  \rm{and\ \ } {\bf B}_k=r_k{\bf P}{\bf Q}^\top\Delta{\bf M}_{k+1}.
		\end{equation}
        
		To analyze \eqref{eq10}, we introduce the notation
		\begin{align}\label{eq70}
			p_{n,s}=\prod\limits_{j=m_0}^{n}[1-r_j(1-\lambda_s)],\ \ \ l_{n,s}=p^{-1}_{n,s},
		\end{align}
		and define the corresponding transition matrices
		\begin{equation}\label{eq78}
			{\bf T}_{k+1,n}=\prod\limits_{j=k+1}^n[{\bf I}-r_j({\bf I}-{\bf J}^\top)],\ \ \ {\bf T}^{(s)}_{k+1,n}=\prod\limits_{j=k+1}^n[{\bf I}-r_j({\bf I}-{\bf J}^\top_s)].
		\end{equation}
        By Lemma \ref{le1}, for each $s\in\{1,\dots,T\}$ and for every index $t\in\{1,\dots,\rho_s\}$, the diagonal entries of $\mathbf{T}^{(s)}_{k+1,n}$ satisfy
		\begin{equation*}
			[{\bf T}^{(s)}_{k+1,n}]_{t,t}=\prod\limits_{j=k+1}^n[1-r_j(1-\lambda_s)]=p_{n,s}l_{k,s}.
		\end{equation*}
        In contrast, for all $t\in\{1,\dots,\rho_s\}$ and $q\in\{1,\dots,\rho_s-1\}$, the off-diagonal entry $[\mathbf{T}^{(s)}{k+1,n}]_{t,t-q}$ can be expressed as
		\begin{align}\label{eq51}
			[{\bf T}^{(s)}_{k+1,n}]_{t,t-q}=\sum\limits_{k+1\le j_1\neq\cdots\neq j_q\le n}R^{(q,s)}_{n,k}p_{n,s}l_{k,s},
		\end{align}
		where
		\begin{align*}
			R^{(q,s)}_{n,k}=\sum\limits_{k+1\le j_1\neq\cdots\neq j_q\le n}\frac{r_{j_1}\cdots r_{j_q} }{[1-r_{j_1}(1-\lambda_s)]\cdots[1-r_{j_q}(1-\lambda_s)]}.
		\end{align*}
        
		Next, we analyze the convergence of the terms $n^{\gamma/2}{\bf A}_{m_0,n}\widehat{\bf Z}_{m_0}$ and $n^{\gamma/2}\sum_{k=m_0}^n {\bf A}_{k+1,n}{\bf B}_k$ in \eqref{eq10} separately. From Lemma \ref{le1}, we have that for all $1/2<\gamma<1$ and $0<\varepsilon<1$,
		\begin{align}\label{eq18}
			&n^{\gamma/2}\|{\bf A}_{m_0,n}\widehat{\bf Z}_{m_0}\|=O_{a.s.}(\|{\bf T}_{m_0,n}\|_1)\non\\
            &\qquad\qquad\qquad\quad\,\,=
			O_{a.s.}\Big(n^{(1-\gamma)(\rho-1)}\exp \Big[-(1-\varepsilon) \frac{c(1-\tau^\ast)}{1-\gamma} n^{1-\gamma}\Big]\Big)\overset{a.s.}{\to}0.
		\end{align}
        We now turn to the convergence of $n^{\gamma/2}\sum_{k=m_0}^n \mathbf{A}_{k+1,n}\mathbf{B}_k$. To this end, we apply Theorem \ref{th6} with ${\cal G}_{n,k}={\cal F}_{k+1}$. To verify condition (c2), observe that
		\begin{equation*}
			\sum\limits_{k=m_0}^nn^{\gamma/2}{\bf A}_{k+1,n}{\bf B}_k(n^{\gamma/2}{\bf A}_{k+1,n}{\bf B}_k)^\top=n^\gamma\sum\limits_{k=m_0}^{n-1}{\bf A}_{k+1,n}{\bf B}_k({\bf A}_{k+1,n}{\bf B}_k)^\top+n^\gamma{\bf B}_n{\bf B}^\top_n.
		\end{equation*}
		For all $i,j\in\{1,2,\cdots,N\}$, we have $[n^\gamma{\bf B}_n{\bf B}^\top_n]_{i,j}=O(n^\gamma r^2_n)=o(1)$. Hence, it suffices to establish the convergence of $n^\gamma\sum_{k=m_0}^{n-1}{\bf A}_{k+1,n}{\bf B}_k({\bf A}_{k+1,n}{\bf B}_k)^\top$. Define
		\begin{equation}\label{eq72}
			{\bf H}_{k+1}={\bf Q}^\top\Delta{\bf M}_{k+1}\Delta{\bf M}^\top_{k+1}{\bf Q}.
		\end{equation}
		Then we have
		\begin{align*}
			&n^\gamma\sum\limits_{k=m_0}^{n-1}{\bf A}_{k+1,n}{\bf B}_k({\bf A}_{k+1,n}{\bf B}_k)^\top\\
			=&n^\gamma{\bf P}\left[\sum\limits_{k=m_0}^{n-1}r^2_k\prod\limits_{j=k+1}^n[{\bf I}-r_j({\bf I}-{\bf J}^\top)]{\bf Q}^\top\Delta{\bf M}_{k+1}\Delta{\bf M}^\top_{k+1}{\bf Q}\left(\prod\limits_{j=k+1}^n[{\bf I}-r_j({\bf I}-{\bf J}^\top)]\right)\right]{\bf P}^\top\\
			=&{\bf P}\left[n^\gamma\sum\limits_{k=m_0}^{n-1}r^2_k\begin{pmatrix}
				{\bf T}^{(1)}_{k+1,n} & \cdots & {\bf 0}\\
				\vdots & \ddots & \vdots\\
				{\bf 0} & \cdots & {\bf T}^{(T)}_{k+1,n}
			\end{pmatrix}\begin{pmatrix}
				{\bf H}^{(1,1)}_{k+1} & \cdots & {\bf H}^{(1,T)}_{k+1}\\
				\vdots & \ddots & \vdots\\
				{\bf H}^{(T,1)}_{k+1} & \cdots & {\bf H}^{(T,T)}_{k+1}
			\end{pmatrix}
			\begin{pmatrix}
				\big({\bf T}^{(1)}_{k+1,n}\big)^\top & \cdots & {\bf 0}\\
				\vdots & \ddots & \vdots\\
				{\bf 0} & \cdots & \big({\bf T}^{(T)}_{k+1,n}\big)^\top
			\end{pmatrix}\right]{\bf P}^\top.
		\end{align*}
        For notational convenience, we denote the term inside the brackets in the above expression by $\widehat{\mathbf{S}}_n$. Recall that ${\cal I}_0=0$ and ${\cal I}_t=\sum_{k=1}^t\rho_k$. For any $u,v\in{1,\dots,T}$ and indices $i$ and $j$ satisfying ${\cal I}{u-1}<i\le{\cal I}u$ and ${\cal I}_{v-1}<j\le{\cal I}_v$, we introduce the shifted indices
		\begin{equation*}
			\tilde{i}=i-{\cal I}_{u-1},\ \ \tilde{j}=j-{\cal I}_{v-1}.
		\end{equation*}
		Then, by Lemmas \ref{le1} and \ref{le13}, we obtain the $(i,j)$th entry of matrix $\widehat{\bf S}_n$ is
		\begin{align}\label{eq62}
			&[\widehat{\bf S}_n]_{i,j}=n^\gamma\sum\limits_{k=m_0}^nr^2_k\sum\limits_{s=0}^{\tilde{j}-1}\sum\limits_{t=0}^{\tilde{i}-1}[{\bf T}^{(u)}_{k+1,n}]_{\tilde{i},\tilde{i}-t}[{\bf H}^{(u,v)}_{k+1}]_{\tilde{i}-t,\tilde{j}-s}[{\bf T}^{(v)}_{k+1,n}]_{\tilde{j},\tilde{j}-s}\non\\
			=&n^\gamma\sum\limits_{k=m_0}^nr^2_k\sum\limits_{s=0}^{\tilde{j}-1}\sum\limits_{t=0}^{\tilde{i}-1}      p_{n,u}l_{k,u}R^{(t,u)}_{n,k}[{\bf H}^{(u,v)}_{k+1}]_{\tilde{i}-t,\tilde{j}-s}p_{n,v}l_{k,v}R^{(s,v)}_{n,k}\non\\
			=&\sum\limits_{s=0}^{\tilde{j}-1}\sum\limits_{t=0}^{\tilde{i}-1}n^\gamma p_{n,u}p_{n,v}\sum\limits_{k=m_0}^nr^2_kl_{k,u}l_{k,v}R^{(t,u)}_{n,k}R^{(s,v)}_{n,k}[{\bf H}^{(u,v)}_{k+1}]_{\tilde{i}-t,\tilde{j}-s}\non\\
			=&
			\sum\limits_{s=0}^{\tilde{j}-1}\sum\limits_{t=0}^{\tilde{i}-1}n^\gamma p_{n,u}p_{n,v}\sum\limits_{k=m_0}^nr^2_k  \left[c/(1-\gamma)\right]^{s+t}(n^{1-\gamma}-k^{1-\gamma})^{s+t} l_{k,u}l_{k,v}[{\bf H}^{(u,v)}_{k+1}]_{\tilde{i}-t,\tilde{j}-s}.
		\end{align}
		Here, strictly speaking, in the above expression, there should be a factor $\psi(k,n,\gamma,\lambda_u,\lambda_v)$. Since for any fixed $k_0$, the total contribution of terms with $k\le k_0$ is $o(1)$ and the function $\psi$ tends to 1 as $k\to\infty$, we may replace $\psi_k$ by 1.

		We now establish the convergence of $[\widehat{\bf S}_n]_{i,j}$ by applying Lemma \ref{le10}. The expression in \eqref{eq62} for fixed $s$ and $t$ can be written in the form $\sum_{k=m_0}^{n-1}\frac{v_{n,k}Y_{k+1}}{c_k}$, where
		\begin{align}
			Y_{k+1}&=[{\bf H}^{(u,v)}_{k+1}]_{\tilde{i}-t,\tilde{j}-s},\ \ c_k=\frac{1}{k^\gamma r^2_k},\ \ \text{and}\non\\
            v_{n,k}&=
			\Big(\frac{c}{1-\gamma}\Big)^{s+t}\left(\frac{n}{k}\right)^\gamma (n^{1-\gamma}-k^{1-\gamma})^{s+t}p_{n,u}p_{n,v}l_{k,u}l_{k,v}.\label{eq50}
		\end{align}
		Now we verify the conditions of Lemma \ref{le10}. From \eqref{eq19}, we have
		\begin{align*}
			\mathbb{E}(Y_n|{\cal F}_{n-1})=&\mathbb{E}([{\bf H}^{(u,v)}_n]_{\tilde{i}-t,\tilde{j}-s}|{\cal F}_{n-1})\\
			=&{\bf q}^\top_{i-t+1}\mathbb{E}[\Delta{\bf M}_n (\Delta{\bf M}_n)^\top|{\cal F}_{n-1}]{\bf q}_{j-s+1}\\
			=&Z_\infty(1-Z_\infty){\bf q}^\top_{i-t+1}{\bf q}_{j-s+1}.
		\end{align*}
		Applying Lemma \ref{le8}, we obtain
		\begin{equation*}
			\lim\limits_{n}\sum\limits_{k=m_0}^{n-1}\frac{v_{n,k}}{c_k}=
			\frac{c(t+s)!}{(2-\lambda_u-\lambda_v)^{t+s+1}}.
		\end{equation*}
		Moreover, by choosing $u=1$ in part (a) of Lemma \ref{le12} and Lemma \ref{le14}, we directly obtain
		\begin{equation*}
			\lim\limits_{n}v_{n,k}=0,\ \   \sum\limits_{k=1}^n\frac{|v_{n,k}|}{c_k}=O(1),\ \ \sum\limits_{k=1}^{n}|v_{n,k}-v_{n,k-1}|=O(1).
		\end{equation*}
		Thus, all the conditions of Lemma \ref{le10} are satisfied, which ensures the convergence of $[\widehat{\bf S}_n]_{i,j}$. Hence, condition (c2) of Theorem \ref{th6} is verified. We now turn to the verification of condition (c3). Recalling the definitions of ${\bf A}_{k+1,n}$ and ${\bf B}_k$, we note that there exists a constant $K>0$ such that
		\begin{equation*}
			\left|{\bf A}_{k+1,n}{\bf B}_k\right|\le K r_k|{\bf T}_{k+1,n}|.	
		\end{equation*}
		Then, for any $u>1$, we have
		\begin{align*}
			&\Big(\sup_{m_0\leq k\leq n}\left|n^{\gamma/2}{\bf A}_{k+1,n}{\bf B}_k\right|\Big)^{2u}\\
			\le &n^{\gamma u}\sum_{k=m_0}^{n-1}\left|{\bf A}_{k+1,n}{\bf B}_k\right|^{2u}+n^{\gamma u}\left|{\bf A}_{n+1,n}{\bf B}_n\right|^{2u}\le K^{2u}n^{\gamma u}\sum_{k=m_0}^{n-1}r^{2u}_k\left|{\bf T}_{k+1,n}\right|^{2u}+n^{\gamma u}\left|{\bf B}_n\right|^{2u} \\
			\le&\max\limits_{v\in\{1,\cdots,T\}}n^{\gamma u} O\Big(\left|p_{n,v}\right|^{2u} \sum_{k=m_0}^{n-1}r_k^{2u}(n^{1-\gamma}-k^{1-\gamma})^{\rho_vu}\left|l_{k,v}\right|^{2u}\Big)+n^u O\big(r_n^{2u}\big),
		\end{align*}
		where the final bound follows from \eqref{eq51} and Lemma \ref{le13}. Note that the term $n^u O\big(r_n^{2u}\big)=o(1)$. Furthermore, Lemma \ref{le12} yields
		\begin{equation*}
			\Big(\sup_{m_0\le k \le n}\left|n^{\gamma/2}{\bf A}_{k+1,n}{\bf B}_k\right|\Big)^{2u}=O\big(n^{\gamma u}n^{-\gamma(2u-1)}\big).
		\end{equation*}
		which vanishes for all $u>1$. Consequently, all the required conditions of Theorem \ref{th6} are verified. This completes the proof of part (a) of Theorem \ref{th9}.

		We now proceed to the proof of part (b) of Theorem \ref{th9}. This result has already been established in Theorem 4.2 of \cite{r2}, where the following decomposition was obtained,
		\begin{equation*}
			n^{\gamma-1/2}\widehat{\bf N}_n=n^{\gamma-3/2}\sum\limits_{k=1}^n{\bf T}_k+{\bf W}^\top{\bf Q}_n,
		\end{equation*}	
		with
		\begin{equation*}
			{\bf T}_k=\Delta{\bf M}_k+k(\widetilde{Z}_{k-1}-\widetilde{Z}_k){\bf 1}=\Delta{\bf M}_k-N^{-1/2}kr_k\left({\bf v}^\top_1\Delta{\bf M}_k{\bf 1}\right),\ \ \ {\bf Q}_n=n^{\gamma-3/2}\sum\limits_{k=1}^n\widehat{Z}_{k-1}.
		\end{equation*}
		The paper \cite{r2} showed that
		\begin{equation*}
			n^{\gamma-3/2}\sum\limits_{k=1}^n{\bf T}_k\overset{L}{\to}{\cal N}(0,\widehat{\bm \Gamma}_\gamma),
		\end{equation*}
		where the proof relies on the structure of $\Delta{\bf M}_k$ and $\widetilde{Z}_k$. Since both $\Delta{\bf M}_k$ and $\widetilde{Z}_k$ retain the same form under the present model, the convergence of $n^{\gamma-3/2}\sum_{k=1}^n{\bf T}_k$  remains valid. 
		In addition, it was shown that ${\bf Q}_n\overset{P}{\to}0$ under the condition that $\lim\limits_{n\to\infty}\mathbb{E}[\|\widehat{\bf Z}_n\|^2]=O(n^{-\gamma})$ as $n\to\infty$. The required moment condition is ensured by \eqref{eq62} and Lemma \ref{le12}. Hence, by Slutsky's theorem, it follows that
		\begin{equation*}
			n^{\gamma-1/2}\widehat{\bf N}_n\overset{L}{\to}{\cal N}(0,\widehat{\bm \Gamma}_\gamma).
		\end{equation*}
		Theorem \ref{th9} is proved.
	\end{proof}

		We now proceed to prove Theorem \ref{th3} by leveraging the convergence properties established in Theorems \ref{th8} and \ref{th9}.
	\begin{proof}[Proof of Theorems \ref{th3}]
		Recall the decomposition of ${\bf Z}_n$ and ${\bf N}_n$ as
		\begin{equation*}
			{\bf Z}_n=\widetilde{Z}_n{\bf 1}+\widehat{\bf Z}_n,\ \ \ {\bf N}_n=\widetilde{Z}_n{\bf 1}+\widehat{\bf N}_n.
		\end{equation*}
		By Theorem \ref{th8}, we have
		\begin{equation*}
			n^{\gamma-\frac{1}{2}}(\widetilde{Z}_n-Z_\infty)\to{\cal N}(0,Z_\infty(1-Z_\infty)\widetilde{\sigma}^2_\gamma))\ \ \ \ \ \ \text{\textit{stably}}.
		\end{equation*}
		By Theorem \ref{th9}, we have
		\begin{equation*}
			n^{\gamma/2}\widehat{{\bf Z}}_n\to {\cal N}({\bf 0},Z_\infty(1-Z_\infty)\widehat{\bm \Sigma}_\gamma),\ \ \ \text{and}\ \  n^{\gamma-\frac{1}{2}}\widehat{{\bf N}}_n\to {\cal N}({\bf 0},Z_\infty(1-Z_\infty)\widehat{\bm \Gamma}_\gamma)\ \ \ \ \ \ \text{\textit{stably}}.
		\end{equation*}
		Note that
		\begin{align*}
			n^{\gamma-\frac{1}{2}}\begin{pmatrix}
				{\bf Z}_n-Z_\infty{\bf 1}\\ {\bf N}_n-Z_\infty{\bf 1}
			\end{pmatrix}=&n^{\gamma-\frac{1}{2}}\begin{pmatrix}
				(\widetilde{Z}_n-Z_\infty){\bf 1}+\widehat{\bf Z}_n\\ {\bf N}_n-\widetilde{Z}_n{\bf 1}+(\widetilde{Z}_n-Z_\infty){\bf 1}
			\end{pmatrix}\\
			=&n^{\gamma-\frac{1}{2}}\begin{pmatrix}
				(\widetilde{Z}_n-Z_\infty){\bf 1}\\ \widehat{\bf N}_n+(\widetilde{Z}_n-Z_\infty){\bf 1}\end{pmatrix}
			+n^{\frac{\gamma-1}{2}}n^{\frac{\gamma}{2}}\begin{pmatrix}
				\widehat{\bf Z}_n\\ {\bf 0}
			\end{pmatrix},
		\end{align*}
		The second term converges to $0$ in probability. By Lemma \ref{le11}, the first term satisfies
		\begin{equation*}
			n^{\gamma-\frac{1}{2}}\begin{pmatrix}
				(\widetilde{Z}_n-Z_\infty){\bf 1}\\ {\bf N}_n-\widetilde{Z}_n{\bf 1}+(\widetilde{Z}_n-Z_\infty){\bf 1}\end{pmatrix}\to{\cal N}\left({\bf 0},Z_\infty(1-Z_\infty)\begin{pmatrix}
				\widetilde{\bm\Sigma}_\gamma & \;\;\widetilde{\bm\Sigma}_\gamma\\ \widetilde{\bm\Sigma}_\gamma & \;\;\widetilde{\bm\Sigma}_\gamma+\widehat{\bm \Gamma}_\gamma
			\end{pmatrix}\right),\ \ \ \ \ \text{\textit{stably}}.
		\end{equation*}
		Hence, Theorem \ref{th3} follows directly from Slutsky's theorem.
	\end{proof}

	\begin{proof}[Proof of Theorem \ref{th2}]

		From \eqref{eq3} we have
		\begin{equation*}
			{\bf N}_{n+1}-{\bf N}_n=-\frac{1}{n+1}({\bf N}_n-{\bf W}^\top{\bf Z}_n)+\frac{1}{n+1}\Delta{\bf M}_{n+1},
		\end{equation*}
		By substituting equation \eqref{eq61} together with the expression of $\widetilde{Z}_n$ into the above, we obtain
		\begin{align*}
			\widehat{\bf N}_{n+1}-\widehat{\bf N}_n=&-\frac{1}{n+1}(\widehat{\bf N}_n-{\bf W}^\top\widehat{\bf Z}_n)+\frac{1}{n+1}\Delta{\bf M}_{n+1}-(\widetilde{Z}_{n+1}-\widetilde{Z}_n){\bf 1}\\
			=&-\frac{1}{n+1}(\widehat{\bf N}_n-{\bf P}{\bf J}^\top{\bf Q}^\top\widehat{\bf Z}_n)+\frac{1}{n+1}\Delta{\bf M}_{n+1}-(\widetilde{Z}_{n+1}-\widetilde{Z}_n){\bf 1}.
		\end{align*}
		The recursion can be reformulated as
		\begin{equation*}
			\widehat{\bf N}_{n+1}=(1-r_nc^{-1})\widehat{\bf N}_n+r_nc^{-1}{\bf P}{\bf J}^\top{\bf Q}^\top\widehat{\bf Z}_n+r_n(c^{-1}{\bf I}-N^{-1/2}{\bf 1}{\bf v}^\top_1)\Delta{\bf M}_{n+1}+r_n{\bf R}_{n+1},
		\end{equation*}
		where the remainder term ${\bf R}_{n+1}$ is given by
		\begin{equation}\label{eq74}
			{\bf R}_{n+1}=\left(\frac{1}{(n+1)r_n}-\frac{1}{c}\right)(-\widehat{\bf N}_n+{\bf P}{\bf J}^\top{\bf Q}^\top\widehat{\bf Z}_n+\Delta{\bf M}_{n+1}).
		\end{equation}
		We recall that the dynamics of $\widehat{\bf Z}_n$ evolve according to
		\begin{align}
			\widehat{\bf Z}_{n+1}=\left[{\bf I}-r_n{\bf P}({\bf I}-{\bf J}^\top){\bf Q}^\top\right]\widehat{\bf Z}_n+r_n{\bf P}{\bf Q}^\top\Delta{\bf M}_{n+1}.
		\end{align}
		To unify these two updates, we define the joint process
		\begin{equation*}
			{\bm \theta}_n=\begin{pmatrix}
				\widehat{\bf Z}_n\\
				\widehat{\bf N}_n
			\end{pmatrix},\ \ \Delta{\bf M}_{\theta,n}=\begin{pmatrix}
				\Delta{\bf M}_n\\
				\Delta{\bf M}_n
			\end{pmatrix},\ \ \ \text{and  }\  {\bf R}_{\theta,n}=\begin{pmatrix}
				{\bf 0}\\
				{\bf R}_n
			\end{pmatrix}.
		\end{equation*}
		Then, the joint dynamics can be written as
		\begin{equation*}
			{\bm \theta}_{n+1}=({\bf I}-r_n{\bf U}){\bm\theta}_n+r_n({\bf V}\Delta{\bf M}_{{\theta},n+1}+{\bf R}_{{\theta},n+1}),
		\end{equation*}
		where the matrices ${\bf U}$ and ${\bf V}$ are given by
		\begin{equation*}
			{\bf U}=\begin{pmatrix}
				{\bf P}({\bf I}-{\bf J}^\top){\bf Q}^\top & \;\;{\bf 0}\\
				-c^{-1}{\bf P}{\bf J}^\top{\bf Q}^\top & \;\;c^{-1}{\bf I}
			\end{pmatrix}\ \ \text{and}\ \  {\bf V}=\begin{pmatrix}
				{\bf P}{\bf Q}^\top  & \;\;{\bf 0}\\
				{\bf 0} & \;\;(c^{-1}-1){\bf p}_1{\bf q}^\top_1+c^{-1}{\bf P}{\bf Q}^\top
			\end{pmatrix}.
		\end{equation*}
		To further simplify the analysis, we introduce two orthonormal $(2N)\times(2N-1)$ matrices, denoted ${\bf P}_\theta$ and ${\bf Q}_\theta$, which satisfy ${\bf Q}_\theta^\top {\bf P}_\theta={\bf P}_\theta^\top {\bf Q}_\theta={\bf I}$ and are defined as follows,
		\begin{equation*}
			{\bf P}_{\theta}=\begin{pmatrix}
				{\bf P} & {\bf 0} & {\bf 0}\\
				{\bf 0} & {\bf p}_1 & {\bf P}
			\end{pmatrix}\ \ \text{and}\ \ {\bf Q}_{\theta}=\begin{pmatrix}
				{\bf Q} & {\bf 0} & {\bf 0}\\
				{\bf 0} & {\bf q}_1 & {\bf Q}
			\end{pmatrix}.
		\end{equation*}
		Then,
		\begin{equation*}
			{\bf P}_\theta {\bf Q}^\top_\theta=\begin{pmatrix}
				{\bf P}{\bf Q}^\top & {\bf 0}\\
				{\bf 0} & {\bf I}
			\end{pmatrix}.
		\end{equation*}
		The matrices ${\bf U}$ and ${\bf V}$ can be expressed as ${\bf U}={\bf P}_\theta{\bf S}_U{\bf Q}^\top_\theta$, ${\bf V}={\bf P}_\theta{\bf S}_V{\bf Q}^\top_\theta$, where the matrices ${\bf S}_U$ and ${\bf S}_U$ are $(2N)\times(2N-1)$ matrices defined as
		\begin{equation*}
			{\bf S}_{ U}=\begin{pmatrix}
				{\bf I}-{\bf J}^\top & {\bf 0} & {\bf 0}\\
				{\bf 0}^\top & c^{-1} & {\bf 0}^\top\\
				-c^{-1}{\bf J}^\top & {\bf 0} & c^{-1}{\bf I}
			\end{pmatrix}\ \ \text{and}\ \ {\bf S}_V=\begin{pmatrix}
				{\bf I} & {\bf 0} & {\bf 0}\\
				{\bf 0}^\top & c^{-1}-1 & {\bf 0}^\top\\
				{\bf 0} & {\bf 0} & c^{-1}{\bf I}
			\end{pmatrix}.
		\end{equation*}
		Under this transformation, the dynamics of $({\bm \theta}_n)_n$ are given by
		\begin{equation*}
			{\bm \theta}_{n+1}={\bf P}_{\theta}({\bf I}-r_n{\bf S}_{U}){\bf Q}^\top_{\theta}{\bm \theta}_n+r_n{\bf V}\Delta{\bf M}_{{\theta},n+1}+r_n{\bf R}_{{\theta},n+1},
		\end{equation*}
		Note that we have chosen $m_0$ sufficiently large such that $(1-\tau)r_j<1/2$ holds for all $j>m_0$. Then, by iterating the recursion until $m_0$, we obtain
		\begin{equation}\label{eq71}
			{\bm \theta}_{n+1}={\bf P}_{\theta}{\bf C}_{m_0,n}{\bm Q}^\top_\theta{\bm\theta}_{m_0}+\sum\limits_{k=m_0}^nr_k{\bf P}_{\theta}{\bf C}_{k+1,n}{\bm Q}^\top_\theta{\bf V}\Delta{\bf M}_{\theta,k+1}+\sum\limits_{k=m_0}^nr_k{\bf P}_{\theta}{\bf C}_{k+1,n}{\bm Q}^\top_\theta {\bf R}_{\theta,k+1},
		\end{equation}
		where ${\bf C}_{n+1,n}={\bf I}$, and for $m_0-1\le k\le n-1$,
		\begin{equation}\label{eq77}
			{\bf C}_{k+1,n}=\prod\limits_{m=k+1}^n[{\bf I}-r_m{\bf S}_U]=\begin{pmatrix}
				{\bf C}^{11}_{k+1,n} & {\bf 0} & {\bf 0}\\
				{\bf 0}^\top & c^{22}_{k+1,n} & {\bf 0}^\top\\
				{\bf C}^{31}_{k+1,n} & {\bf 0} & {\bf C}^{33}_{k+1,n}
			\end{pmatrix}.
		\end{equation}
		Note that the blocks ${\bf C}^{11}_{k+1,n}$, ${\bf C}^{31}_{k+1,n}$ and ${\bf C}^{33}_{k+1,n}$ are all $(N-1)\times(N-1)$ matrices. For notational convenience, in the sequel we let $\alpha_u=1-\lambda_u$ for all $1\le u\le T$ and $F_{k+1,n}(\alpha_u)=p_{n,u}l_{k,u}$. Then, by Lemma \ref{leww1}, we have that for all $1\le u\le T$, ${\cal I}_{u-1}\le i\le {\cal I}_u$, and $0\le t\le i-1$, $1\le s\le i-1$,
		\begin{align*}
			&[{\bf C}^{11}_{k+1,n}]_{i,i-t}\sim c^t(\log n-\log k)^tF_{k+1,n}(\alpha_u),\\
			&[{\bf C}^{33}_{k+1,n}]_{i,i}=c^{22}_{k+1,n}=F_{k+1,n}(c^{-1}),\\
			&[{\bf C}^{31}_{k+1,n}]_{i,i}=\begin{cases}
				\frac{1-\alpha_u}{c\alpha_u-1}[F_{k+1,n}(c^{-1})-F_{k+1,n}(\alpha_u)]\ \ &\text{for  }c\alpha_j\neq 1,\\
				(1-c^{-1})F_{k+1,n}(c^{-1})(\log n-\log k)+O(n^{-1})\ \ &\text{for  }c\alpha_j=1,
			\end{cases}\\
			&[{\bf C}^{31}_{k+1,n}]_{i,i-s}\sim
			\left[c^{s-1}(\log n-\log k)^{s-1}-(1-\alpha_u)c^s(\log n-\log k)^s\right]\cdot\\
			&\qquad\qquad\qquad	\begin{cases}\frac{1}{c\alpha_u-1}[F_{k+1,n}(c^{-1})-F_{k+1,n}(\alpha_u)]\ \ &\text{for  }c\alpha_j\neq 1,\\
				\frac{1-c^{-1}}{1-\alpha_u}F_{k+1,n}(c^{-1})(\log n-\log k)+O(n^{-1})\ \ &\text{for  }c\alpha_j=1.
			\end{cases}
		\end{align*}
		We set 
		\begin{equation*}
			t_n=\begin{cases}
				\sqrt{n}\ \ &\text{for } \tau<1-(2c)^{-1},\\
				\frac{\sqrt{n}}{(\log n)^{\rho-1/2}}\ \ &\text{for }\tau=1-(2c)^{-1}.
			\end{cases}
		\end{equation*}	
		We next establish the convergence of $t_n{\bm\theta}_n$ by analyzing the limiting behavior of each term in \eqref{eq71}. We show that the terms $t_n{\bf P}_{\theta}{\bf C}_{m_0,n}{\bm Q}^\top_\theta{\bm\theta}_{m_0}$ and $t_n\sum_{k=m_0}^nr_k{\bf P}_{\theta}{\bf C}_{k+1,n}{\bm Q}^\top_\theta {\bf R}_{\theta,k+1}$ both converge to 0 almost surely, so that the asymptotic distribution of $t_n{\bm\theta}_n$ is determined by the martingale term $t_n\sum_{k=m_0}^nr_k{\bf P}_{\theta}{\bf C}_{k+1,n}{\bm Q}^\top_\theta {\bf V}\Delta{\bf M}_{\theta,k+1}$.

		We first verify that $|t_n{\bf P}_{\theta}{\bf C}_{m_0,n}{\bm Q}^\top_\theta{\bm\theta}_{m_0}|{\to}0$. To this end, we begin by bounding the magnitude of ${\bf C}_{k+1,n}$. From Lemma \ref{le3}, we obtain
		\begin{align*}
			|{\bf C}_{k+1,n}|=&O\Big((\log n-\log k)^{\rho-1} \max\limits_{u\in\{1,\cdots,T\}}|F_{k+1,n}(\alpha_u)|\Big)\\
            &\qquad+O((\log n-\log k)^\rho F_{k+1,n}(c^{-1}))+O(n^{-1})\\
			=&O\big((\log n)^{\rho-1}(k/n)^{c(1-\tau)}\big)+O\big((\log n)^\rho k/n\big),
		\end{align*}
		Therefore, we have
		\begin{align*}
			|t_n{\bf P}_{\theta}{\bf C}_{m_0,n}{\bm Q}^\top_\theta{\bm\theta}_{m_0}|=&O(t_n|{\bf C}_{m_0,n}|)\\
			=&O\big(t_n(\log n)^{\rho-1}n^{-c(1-\tau)}\big)+O(t_n(\log n)^\rho n^{-1})\\
			=&\begin{cases}
				O\big(\sqrt{n}(\log n)^{\rho-1}n^{-c(1-\tau)}\big)\ \ &\text{for }\tau<1-(2c)^{-1},\\
				O\Big(\frac{\sqrt{n}}{n^{\rho-1/2}}(\log n)^{\rho-1}n^{-1/2}\Big)\ \ &\text{for }\tau=1-(2c)^{-1}.
			\end{cases}
		\end{align*}
		When $\tau=1-(2c)^{-1}$, the expression converges to $0$. On the other hand, when $\tau<1-(2c)^{-1}$, the condition $c(1-\tau)>1/2$ guarantees that the entire expression again converges to $0$.

		We next show that $|t_n\sum_{k=m_0}^nr_k{\bf P}_{\theta}{\bf C}_{k+1,n}{\bm Q}^\top_\theta {\bf R}_{\theta,k+1}|{\to}0$. From Assumption \ref{as2} and \eqref{eq74}, we have $|{\bf R}_{\theta,k}|=|{\bf R}_k|=O(k^{-1})$, and hence
		\begin{align*}
			&\bigg|t_n\sum_{k=m_0}^nr_k{\bf P}_{\theta}{\bf C}_{k+1,n}{\bm Q}^\top_\theta {\bf R}_{\theta,k+1}\bigg|\\
			=&O\Big(t_n\sum\limits_{k=m_0}^{n-1}r_kk^{-1}|{\bf C}_{k+1,n}|\Big)+O(t_nr_nn^{-1})\\
			=&O\Big(t_nn^{-c(1-\tau)}(\log n)^{\rho-1}\sum\limits_{k=m_0}^{n-1}r_kk^{-1}k^{c(1-\tau)}\Big)+O\Big(t_nn^{-1}(\log n)^\rho\sum\limits_{k=m_0}^{n-1}r_k\Big)\\
			=&O\Big(t_nn^{-c(1-\tau)}(\log n)^{\rho-1}\sum\limits_{k=m_0}^{n-1}k^{-[2-c(1-\tau)]}\Big)\\
			=&\begin{cases}
				n^{1/2-c(1-\tau)}(\log n)^\rho\ \ &\text{for }\tau=1-c^{-1},\\
				n^{1/2-c(1-\tau)}(\log n)^{\rho-1}n^{c(1-\tau)-1}\ \ &\text{for }\tau<1-(2c)^{-1}\ \text{and }\tau\neq 1-c^{-1},\\
				\frac{\sqrt{n}}{(\log n)^{\rho-1/2}}n^{-1/2}(\log n)^{\rho-1}\ \ &\text{for } \tau=1-(2c)^{-1}.
			\end{cases}
		\end{align*}
		In all three cases, the term converges to zero, thus $|t_n\sum_{k=m_0}^nr_k{\bf P}_{\theta}{\bf C}_{k+1,n}{\bm Q}^\top_\theta {\bf R}_{\theta,k+1}|\to0$.

		Now, we establish the convergence of $t_n\sum_{k=m_0}^nr_k{\bf P}_{\theta}{\bf C}_{k+1,n}{\bm Q}^\top_\theta {\bf V}\Delta{\bf M}_{\theta,k+1}$ by Theorem \ref{th6}. For this purpose, we set ${\cal G}_{n,k}={\cal F}_{k+1}$. We first verify condition (c3). There exists a constant $K_1$ such that
		\begin{equation*}
			|t_nr_k{\bf P}_{\theta}{\bf C}_{k+1,n}{\bm Q}^\top_\theta {\bf V}\Delta{\bf M}_{\theta,k+1}|\le K_1t_nr_k|{\bf C}_{k+1,n}|.
		\end{equation*}	
		Then, by Lemma \ref{le12}, we have that for all $u>1$,
		\begin{align*}
			&\Big(\sup\limits_{m_0\le k\le n}\big|t_nr_k{\bf P}_{\theta}{\bf C}_{k+1,n}{\bm Q}^\top_\theta {\bf V}\Delta{\bf M}_{\theta,k+1}\big|\Big)^{2u}\\
			\le& t^{2u}_n\sum\limits_{k=m_0}^{n-1}K^{2u}_1|r_k{\bf C}_{k+1,n}|^{2u}+K_1t_nr_n\\
			=&\begin{cases}
				O(n^{-(u-1)})\ \ &\text{for }\tau<1-(2c)^{-1}\ \text{and }2u c(1-\tau)>2u-1,\\
				O((\log n)^{-u})\ \ &\text{for }\tau=1-(2c)^{-1}.
			\end{cases}
		\end{align*}
		Therefore, both terms on the right-hand side converge to $0$ for any $u>1$. For the first term, we require the condition $2u c(1-\tau)>2u-1$ to hold. Under the assumption $\tau<1-(2c)^{-1}$, this condition is satisfied for all $u$ when $c(1-\tau)\ge 1$. For the case where $1/2<c(1-\tau)<1$, the condition can still be fulfilled by choosing $u$ within the interval  $\Big(1,\frac{1}{2-2c(1-\tau)}\Big)$. Thus, there exists $u>1$ such that
		\begin{equation*}
			\Big(\sup\limits_{m_0\le k\le n}\big|t_nr_k{\bf P}_{\theta}{\bf C}_{k+1,n}{\bm Q}^\top_\theta {\bf V}\Delta{\bf M}_{\theta,k+1}\big|\Big)^{2u}\to0,
		\end{equation*}
		which verifies condition (c3).

		Now, we verify condition (c2). 
		Note that
		\begin{align*}
			&t^2_n\sum_{k=m_0}^nr^2_k{\bf P}_{\theta}{\bf C}_{k+1,n}{\bm Q}^\top_\theta {\bf V}\Delta{\bf M}_{\theta,k+1}\Delta{\bf M}^\top_{\theta,k+1}{\bf V}^\top{\bm Q}_\theta{\bf C}^\top_{k+1,n}{\bf P}^\top_{\theta}\\
			=&{\bf P}_{\theta}\bigg(t^2_n\sum_{k=m_0}^nr^2_k{\bf C}_{k+1,n}{\bm Q}^\top_\theta {\bf V}\Delta{\bf M}_{\theta,k+1}\Delta{\bf M}^\top_{\theta,k+1}{\bf V}^\top{\bm Q}_\theta{\bf C}^\top_{k+1,n}\bigg){\bf P}^\top_{\theta}\\
			=&{\bf P}_{\theta}\bigg(t^2_n\sum_{k=m_0}^nr^2_k{\bf C}_{k+1,n}{\bm Q}^\top_\theta {\bf P}_\theta{\bf S}_V{\bf Q}^\top_\theta\Delta{\bf M}_{\theta,k+1}\Delta{\bf M}^\top_{\theta,k+1}{\bf Q}_\theta{\bf S}^\top_V{\bf P}^\top_\theta {\bm Q}_\theta{\bf C}^\top_{k+1,n}\bigg){\bf P}^\top_{\theta}\\
			=&{\bf P}_{\theta}\bigg(t^2_n\sum_{k=m_0}^nr^2_k{\bf C}_{k+1,n}{\bf S}_V{\bf Q}^\top_\theta\Delta{\bf M}_{\theta,k+1}\Delta{\bf M}^\top_{\theta,k+1}{\bf Q}_\theta{\bf S}^\top_V{\bf C}^\top_{k+1,n}\bigg){\bf P}^\top_{\theta}\\
			=&{\bf P}_{\theta}\bigg(t^2_n\sum_{k=m_0}^{n-1}r^2_k{\bf C}_{k+1,n}{\bf S}_V{\bf Q}^\top_\theta\Delta{\bf M}_{\theta,k+1}\Delta{\bf M}^\top_{\theta,k+1}{\bf Q}_\theta{\bf S}^\top_V{\bf C}^\top_{k+1,n}\bigg){\bf P}^\top_{\theta}+\\
			&\quad{\bf P}_{\theta}\bigg(t^2_nr^2_n{\bf S}_V{\bf Q}^\top_\theta\Delta{\bf M}_{\theta,n+1}\Delta{\bf M}^\top_{\theta,n+1}{\bf Q}_\theta{\bf S}^\top_V\bigg){\bf P}^\top_{\theta},
		\end{align*}
		where the last term equals to
		\begin{equation*}
			O(t^2_nr^2_n)=\begin{cases}
				O(n^{-1})\to0\ \ &\text{for }\ \tau<1-(2c)^{-1}, \\
				O(n^{-1}(\log n)^{2\rho-1})\to0\ \ &\text{for }\ \tau=1-(2c)^{-1}.
			\end{cases}
		\end{equation*}
		Therefore, it remains to prove the convergence of
		\begin{equation}\label{eq73}
			t^2_n\sum_{k=m_0}^{n-1}r^2_k{\bf C}_{k+1,n}{\bf S}_V{\bf Q}^\top_\theta\Delta{\bf M}_{\theta,k+1}\Delta{\bf M}^\top_{\theta,k+1}{\bf Q}_\theta{\bf S}^\top_V{\bf C}^\top_{k+1,n}.
		\end{equation}
		To this end, we define
		\begin{align*}
			{\bf H}_{\theta,k+1}=&{\bf Q}^\top_\theta\Delta{\bf M}_{\theta,k+1}\Delta{\bf M}^\top_{\theta,k+1}{\bf Q}_\theta\\
			=&\begin{pmatrix}
				{\bf Q}^\top\Delta{\bf M}_{k+1}\Delta{\bf M}^\top_{k+1}{\bf Q} & \;\;{\bf Q}^\top\Delta{\bf M}_{k+1}\Delta{\bf M}^\top_{k+1}{\bf q}_1  & \;\;{\bf Q}^\top\Delta{\bf M}_{k+1}\Delta{\bf M}^\top_{k+1}{\bf Q}\\
				{\bf q}^\top_1\Delta{\bf M}_{k+1}\Delta{\bf M}^\top_{k+1}{\bf Q} & \;\;{\bf q}^\top_1\Delta{\bf M}_{k+1}\Delta{\bf M}^\top_{k+1}{\bf q}_1 & \;\;{\bf q}^\top_1\Delta{\bf M}_{k+1}\Delta{\bf M}^\top_{k+1}{\bf Q}\\
				{\bf Q}^\top\Delta{\bf M}_{k+1}\Delta{\bf M}^\top_{k+1}{\bf Q} & \;\;{\bf Q}^\top\Delta{\bf M}_{k+1}\Delta{\bf M}^\top_{k+1}{\bf q}_1 & \;\;{\bf Q}^\top\Delta{\bf M}_{k+1}\Delta{\bf M}^\top_{k+1}{\bf Q}
			\end{pmatrix}.
		\end{align*}
		The term ${\bf Q}^\top\Delta{\bf M}_{k+1}\Delta{\bf M}^\top_{k+1}{\bf Q}$ was denoted by ${\bf H}_{k+1}$ in \eqref{eq72}, and for all $1\le i,j\le N-1$,
		\begin{equation}\label{eq83}
			\mathbb{E}([{\bf H}_{k+1}]_{i,j}|{\cal F}_k)\overset{a.s.}{\to}Z_\infty(1-Z_\infty){\bf q}^\top_{i+1}{\bf q}_{j+1}.
		\end{equation}
		Define ${\bf h}_{k+1}={\bf Q}^\top\Delta{\bf M}_{k+1}\Delta{\bf M}^\top_{k+1}{\bf q}_1$ and $h_{k+1}={\bf q}^\top_1\Delta{\bf M}_{k+1}\Delta{\bf M}^\top_{k+1}{\bf q}_1$, then for all $1\le i\le N-1$,
		\begin{equation}\label{eq84}
			\mathbb{E}([{\bf h}_{k+1}]_{i,1}|{\cal F}_k)\overset{a.s.}{\to}Z_\infty(1-Z_\infty){\bf q}^\top_{i+1}{\bf q}_1,\ \ \mathbb{E}(h_{k+1}|{\cal F}_k)\overset{a.s.}{\to}Z_\infty(1-Z_\infty){\bf q}^\top_1{\bf q}_1.
		\end{equation}
		To simplify notation, we define ${\bf C}^1_{k+1,n}={\bf C}^{11}_{k+1,n}$, ${\bf C}^3_{k+1,n}=({\bf C}^{31}_{k+1,n}+c^{-1}{\bf C}^{33}_{k+1,n})$, and $c^2_{k+1,n}=(c^{-1}-1)c^{22}_{k+1,n}$. Then, the term \eqref{eq73} can be rewritten as
		\begin{equation}\label{eq75}
			{\bf W}_n
			=t^2_n\sum\limits_{k=m_0}^{n-1}\begin{pmatrix}
				{\bf C}^1_{k+1,n}{\bf H}_{k+1}({\bf C}^1_{k+1,n})^\top & c^2_{k+1,n}{\bf C}^1_{k+1,n}{\bf h}_{k+1} & {\bf C}^1_{k+1,n}{\bf H}_{k+1}({\bf C}^3_{k+1,n})^\top\\
				c^2_{k+1,n}{\bf h}^\top_{k+1}({\bf C}^1_{k+1,n})^\top & (c^2_{k+1,n})^2h_{k+1} & c^2_{k+1,n}{\bf h}_{k+1}({\bf C}^3_{k+1,n})^\top\\
				{\bf C}^3_{k+1,n}{\bf H}_{k+1}({\bf C}^1_{k+1,n})^\top & c^2_{k+1,n}({\bf C}^3_{k+1,n})^\top{\bf h}_{k+1} & {\bf C}^3_{k+1,n}{\bf H}_{k+1}({\bf C}^3_{k+1,n})^\top
			\end{pmatrix},
		\end{equation}
		For all $1\le u\le T$, ${\cal I}_{u-1}\le i\le {\cal I}_u$, and for $0\le t\le i-1$, $1\le s\le i-1$, we have the scalar $c^2_{k+1,n}$ and the elements of $(N-1)\times(N-1)$ matrices ${\bf C}^1_{k+1,n}$ and ${\bf C}^3_{k+1,n}$ are equal to
		\begin{align}
			&[{\bf C}^{1}_{k+1,n}]_{i,i-t}\sim c^t(\log n-\log k)^tF_{k+1,n}(\alpha_u),  \label{eq85}\\
			&c^{2}_{k+1,n}=(c^{-1}-1)F_{k+1,n}(c^{-1}),\label{eq86}\\
			&[{\bf C}^{3}_{k+1,n}]_{i,i}\label{eqww1}\\
            &\qquad=\begin{cases}
				\frac{1}{c\alpha_u-1}[(1-c^{-1})F_{k+1,n}(c^{-1})-(1-\alpha_u)F_{k+1,n}(\alpha_u)]\ \ &\text{for  }c\alpha_j\neq 1,  \\
				[(1-c^{-1})(\log n-\log k)+c^{-1}]F_{k+1,n}(c^{-1})+O(n^{-1})\ \ &\text{for  }c\alpha_j=1,
			\end{cases}  \non\\
			&[{\bf C}^{3}_{k+1,n}]_{i,i-s}\sim
			\left[c^{s-1}(\log n-\log k)^{s-1}+(1-\alpha_u)c^s(\log n-\log k)^s\right]\cdot\label{eqww2}\\
			&\qquad\qquad\qquad	\begin{cases}\frac{1}{c\alpha_u-1}[F_{k+1,n}(c^{-1})-F_{k+1,n}(\alpha_u)]\ \ &\text{for  }c\alpha_j\neq 1,\\
				\frac{1-c^{-1}}{1-\alpha_u}F_{k+1,n}(c^{-1})(\log n-\log k)+O(n^{-1})\ \ &\text{for  }c\alpha_j=1.
			\end{cases}\non
		\end{align}

		The convergence of each term in \eqref{eq75} can be established by combining the following results:
        \begin{align}\label{eq76}
            & t^2_n\sum\limits_{k=m_0}^{n-1}r^2_k(\log n-\log k)^q F_{k+1,n}(x)F_{k+1,n}(y)  \\
            & \qquad \qquad  \overset{a.s.}{\to} \begin{cases}
				\frac{c^2q!}{[-1+c(x+y)]^{q+1}}\ \ &\text{for } c({\rm Re}(x)+{\rm Re}(y))>1,\\
				\frac{c^2}{q+1}\ \ &\text{for } c(x+y)=1,\ \text{and}\\
				& \quad m(1-x)=m(1-y)=\rho,\\ 
				0,\ \ &\text{for } c({\rm Re}(x)+{\rm Re}(y))=1,\\
				& \quad \text{and }\ c(x+y)\neq1,\\
				& \quad \text{or } m(1-x)m(1-y)<\rho^2.
			\end{cases}\non
        \end{align}
		Here, $x,y\in\{c,\alpha_j,\ 1\le j\le S\}$, $q$ is a non-negative integer, and $m(1-x)$ denotes the geometric multiplicity of the eigenvalue $1-x$. Moreover, let $\eta\in\{[{\bf H}_{k+1}]{i,j}, {\bf h}_{i,1}, h_{k+1},\ 1\le i,j\le N-1\}$. Note that we omit the $O(n^{-1})$ terms associated with ${\bf C}^3_{k+1,n}$, since for all integers $q\ge0$, we have
        \begin{align*}
        & t^2_n(\log n)^q\sum\limits_{k=m_0}^{n-1}r^2_k|\eta_{k+1}|O(n^{-2})=\begin{cases}
        O(n^{-2}(\log n)^q)\to0\ \ &\text{for } \tau<1-(2c)^{-1},\\
        O(n^{-2}(\log n)^{1-2\rho}(\log n)^q)\to0\ \ &\text{for } \tau=1-(2c)^{-1}.
        \end{cases} \\
        & t^2_n(\log n)^q\sum\limits_{k=m_0}^{n-1}r^2_k|\eta_{k+1}|O(n^{-1})F_{k+1,n}(c^{-1}) \\
        & \qquad\qquad\qquad\qquad\qquad\quad\qquad = \begin{cases}
        O(n^{-1}(\log n)^{q+1}) \to 0,           & \text{for } \tau<1-(2c)^{-1}, \\
        O(n^{-1}(\log n)^{2-2\rho+q}) \to 0,     & \text{for } \tau=1-(2c)^{-1}.
        \end{cases} \\
        & t^2_n(\log n)^q\sum\limits_{k=m_0}^{n-1}r^2_k|\eta_{k+1}|O(n^{-1})F_{k+1,n}(\alpha_u) \\
        & \qquad\qquad\qquad\qquad\qquad\qquad\quad = \begin{cases}
        O(n^{-c{\rm Re}(\alpha_u)}\log n) \to 0, & \text{for } {\rm Re} (\alpha_u)=1-c^{-1}, \\
        O(n^{-1}) \to 0,                         & \text{for } \tau<1-(2c)^{-1} \text{ and } \\
                                                & \quad {\rm Re}(\alpha_u)\neq 1-c^{-1}, \\
        O(n^{-1}(\log n)^{1-2\rho}) \to 0,       & \text{for } {\rm Re}(\alpha_u)=1-(2c)^{-1}.
        \end{cases}
        \end{align*}

		We now establish the convergence of \eqref{eq76} by applying Lemma \ref{le10}. The equation \eqref{eq76} can be written as $\sum_{k=m_0}^{n-1}\frac{v_{n,k+1}Y_k}{c_k}$, where $Y_{k+1}=\eta_{k+1}$,
		\begin{equation*}
			c_k=\begin{cases}
				1/{kr^2_k}\ \ &\text{for } c[2-{\rm Re}(x)-{\rm Re}(y)]>1,\\
				\log k/(kr^2_k)\ \ &\text{for } c(x+y)=1.
			\end{cases}
		\end{equation*}
		and
		\begin{align}\label{eq45}
			v_{n,k}=\begin{cases}
				\frac{n}{k}(\log n-\log k)^qF_{k+1,n}(x)F_{k+1,n}(y)\ \ &\text{for }
				c({\rm Re}(x)+{\rm Re}(y))>1,\\
				\frac{n\log k}{k(\log n)^{2\rho-1}}(\log n-\log k)^qF_{k+1,n}(x)F_{k+1,n}(y)\  &\text{for }c(x+y)=1.
			\end{cases}
		\end{align}
		The conditional convergence of $Y_k$ is given by \eqref{eq83} and \eqref{eq84}. Moreover,
		\begin{equation*}
			\sum\limits_{n}\frac{\mathbb{E}[|Y_n|^2]}{c^2_n}=\begin{cases}
				O\Big(\sum\limits_{n}n^2r^4_n\Big)<\infty\ \ &\text{for }c({\rm Re}(x)+{\rm Re}(y))>1,\\
				O\Big(\sum\limits_{n}n^2r^4_n/(\log n)^2\Big)<\infty\ \ &\text{for }c(x+y)=1.
			\end{cases}
		\end{equation*}
		By Lemma \ref{le8}, we obtain
		\begin{align*}
			\lim\limits_{n}\sum\limits_{k=m_0}^{n-1}\frac{v_{n,k}}{c_k}=\begin{cases}
				\frac{c^{t+s+2}}{[-1+c(2-\lambda_u-\lambda_v)]^{t+s+1}}\ \ &\text{for } c({\rm Re}(x)+{\rm Re}(y))>1,\\
				\frac{c^{2\rho}}{2\rho-1}\ &\text{for } c(x+y)=1,\ \text{and}\\
				&\ \ m(1-x)=m(1-y)=\rho.\\
				0,\ &\text{for } c({\rm Re}(x)+{\rm Re}(y))=1,\\
				&\ \ \text{and } c(x+y)\neq1,\\
				&\ \  \text{or } m(1-x)m(1-y)<\rho^2.
			\end{cases}
		\end{align*}
		And, by choosing $u=1$ in Lemmas \ref{le12} and \ref{le14}, we immediately obtain the validity of
		\begin{equation*}
			\lim\limits_{n}v_{n,k}=0,\ \   \sum\limits_{k=1}^n\frac{|v_{n,k}|}{c_k}=O(1),\ \ \sum\limits_{k=1}^{n}|v_{n,k}-v_{n,k-1}|=O(1).
		\end{equation*}
		Thus, by Lemma \ref{le10}, the convergence of \eqref{eq76} is established. Hence, condition (c2) is satisfied.

		We have now established \eqref{eq76}. Based on this expression and the combinations of $x$, $y$ and $q$, we can derive the convergence of each term in \eqref{eq75}. The results are presented below, with detailed derivations given in Supplementary Materials.
		For $1\le u,v\le T$, ${\cal I}_{u-1}\le i\le {\cal I}_u$, ${\cal I}_{v-1}\le j\le {\cal I}_v$, we have
		\begin{align*}
			&t^2_n\sum_{k=m_0}^{n-1}r^2_k[{\bf C}^1_{k+1,n}{\bf H}_{k+1}({\bf C}^1_{k+1,n})^\top]_{i,j}\overset{a.s.}{\to}\begin{cases}
				Z_\infty(1-Z_\infty)[\widehat{\bf S}_{{\bf ZZ}}]_{i,j}\ \ \text{for  }\tau<1-(2c)^{-1},\\
				Z_\infty(1-Z_\infty)[\widehat{\bf S}^\ast_{{\bf ZZ}}]_{i,j}\ \ \text{for  }\tau=1-(2c)^{-1}.\\
			\end{cases}\\
			&t^2_n\sum_{k=m_0}^{n-1}r^2_k[c^2_{k+1,n}{\bf C}^1_{k+1,n}{\bf h}_{k+1}   ]_i\overset{a.s.}{\to}\begin{cases}
				Z_\infty(1-Z_\infty)[\widehat{\bf S}_{{\bf ZN}}]_{i,1}\ \ \text{for  }\tau<1-(2c)^{-1},\\
				Z_\infty(1-Z_\infty)[\widehat{\bf S}^\ast_{{\bf ZN}}]_{i,1}\ \ \text{for  }\tau=1-(2c)^{-1}.\\
			\end{cases}\\
			&t^2_n\sum_{k=m_0}^{n-1}r^2_k[{\bf C}^1_{k+1,n}{\bf H}_{k+1}({\bf C}^3_{k+1,n})^\top]_{i,j}\overset{a.s.}{\to}\begin{cases}
				Z_\infty(1-Z_\infty)[\widehat{\bf S}_{{\bf ZN}}]_{i,j+1}\ \ \text{for  }\tau<1-(2c)^{-1},\\
				Z_\infty(1-Z_\infty)[\widehat{\bf S}^\ast_{{\bf ZN}}]_{i,j+1}\ \ \text{for  }\tau=1-(2c)^{-1}.\\
			\end{cases}\\
			&t^2_n\sum_{k=m_0}^{n-1}r^2_k(c^2_{k+1,n})^2h_{k+1}\overset{a.s.}{\to}\begin{cases}
				Z_\infty(1-Z_\infty)[\widehat{\bf S}_{{\bf NN}}]_{1,1}\ \ \text{for  }\tau<1-(2c)^{-1},\\
				Z_\infty(1-Z_\infty)[\widehat{\bf S}^\ast_{{\bf NN}}]_{1,1}\ \ \text{for  }\tau=1-(2c)^{-1}.\\
			\end{cases}\\
			&t^2_n\sum_{k=m_0}^{n-1}r^2_k[c^2_{k+1,n}{\bf h}_{k+1}({\bf C}^3_{k+1,n})^\top]_j\overset{a.s.}{\to}\begin{cases}
				Z_\infty(1-Z_\infty)[\widehat{\bf S}_{{\bf NN}}]_{1,j+1}\ \ \text{for  }\tau<1-(2c)^{-1},\\
				Z_\infty(1-Z_\infty)[\widehat{\bf S}^\ast_{{\bf NN}}]_{1,j+1}\ \ \text{for  }\tau=1-(2c)^{-1}.\\
			\end{cases}\\
			&t^2_n\sum_{k=m_0}^{n-1}r^2_k[{\bf C}^3_{k+1,n}{\bf H}_{k+1}({\bf C}^3_{k+1,n})^\top]_{i,j}\overset{a.s.}{\to}\begin{cases}
				Z_\infty(1-Z_\infty)[\widehat{\bf S}_{{\bf NN}}]_{i+1,j+1}\ \ \text{for  }\tau<1-(2c)^{-1},\\
				Z_\infty(1-Z_\infty)[\widehat{\bf S}^\ast_{{\bf NN}}]_{i+1,j+1}\ \ \text{for  }\tau=1-(2c)^{-1}.\\
			\end{cases}
		\end{align*}
		Recall the definition of ${\bf P}_\theta$, we then obtain
		\begin{align*}
			&{\bf P}_\theta{\bf W}_n{\bf P}^\top_\theta\\
            &\qquad\overset{a.s.}{\to}\begin{cases}
				Z_\infty(1-Z_\infty)\begin{pmatrix}
					{\bf P} & \;\;{\bf 0}\\
					{\bf 0} & \;\;\widetilde{\bf P}
				\end{pmatrix}\begin{pmatrix}
					\widehat{\bf S}_{{\bf ZZ}} & \;\;\widehat{\bf S}_{{\bf ZN}}\\
					\widehat{\bf S}^\top_{{\bf ZN}} & \;\;\widehat{\bf S}_{{\bf NN}}
				\end{pmatrix}\begin{pmatrix}
					{\bf P}^\top & {\bf 0}\\
					{\bf 0} & \widetilde{\bf P}^\top
				\end{pmatrix}=Z_\infty(1-Z_\infty)\begin{pmatrix}
					\widehat{\bm\Sigma}_{{\bf ZZ}} & \;\;\widehat{\bm\Sigma}_{{\bf ZN}}\\
					\widehat{\bm\Sigma}^\top_{{\bf ZN}} & \;\;\widehat{\bm\Sigma}_{{\bf NN}}
				\end{pmatrix},\\
				\qquad\qquad\qquad\qquad\qquad\qquad\qquad\qquad\qquad\qquad\qquad\qquad\quad \text{for  } \tau<1-(2c)^{-1},\\
				Z_\infty(1-Z_\infty)\begin{pmatrix}
					{\bf P} &\;\; {\bf 0}\\
					{\bf 0} & \;\;\widetilde{\bf P}
				\end{pmatrix}\begin{pmatrix}
					\widehat{\bf S}^\ast_{{\bf ZZ}} & \;\;\widehat{\bf S}^\ast_{{\bf ZN}}\\
					\widehat{\bf S}^\ast_{{\bf ZN}} & \;\;\widehat{\bf S}^\ast_{{\bf NN}}
				\end{pmatrix}\begin{pmatrix}
					{\bf P}^\top & \;\;{\bf 0}\\
					{\bf 0} & \;\;\widetilde{\bf P}^\top
				\end{pmatrix}=Z_\infty(1-Z_\infty)\begin{pmatrix}
					\widehat{\bm\Sigma}^\ast_{{\bf ZZ}} & \;\;\widehat{\bm\Sigma}^\ast_{{\bf ZN}}\\
					\widehat{\bm\Sigma}^\ast_{{\bf ZN}} & \;\;\widehat{\bm\Sigma}^\ast_{{\bf NN}}
				\end{pmatrix}\\
				\qquad\qquad\qquad\qquad\qquad\qquad\qquad\qquad\qquad\qquad\qquad\qquad\quad\text{for  } \tau=1-(2c)^{-1}.
			\end{cases}
		\end{align*}
		Theorem \ref{th2} is proved.
	\end{proof}

	We now establish Theorems \ref{th5}, \ref{th13} and \ref{th4} using Theorems \ref{th8} and \ref{th2}.
	\begin{proof}[Proof of Theorem \ref{th5}, \ref{th13} and \ref{th4}]
		By Theorem \ref{th8}, we have
		\begin{equation*}
			\sqrt{n}(\widetilde{Z}_n-Z_\infty)\to{\cal N}(0,Z_\infty(1-Z_\infty)\widetilde{\sigma}^2_\gamma))\ \ \ \ \ \ \text{\textit{stably}}.
		\end{equation*}
		By case (a) of Theorem \ref{th2}, it follows that
		\begin{equation*}
			\sqrt{n}\begin{pmatrix}
				\widehat{{\bf Z}}_n\\ \widehat{{\bf N}}_n
			\end{pmatrix}\to{\cal N}\left({\bf 0},Z_\infty(1-Z_\infty)\begin{pmatrix}
				\widehat{\bm\Sigma}_{\bf ZZ} & \;\;\widehat{\bm\Sigma}_{\bf ZN}\\
				\widehat{\bm\Sigma}^\top_{\bf ZN} & \;\;\widehat{\bm\Sigma}_{\bf NN}
			\end{pmatrix}\right)\ \ \ \ \ \ \text{\textit{stably}}.
		\end{equation*}
		Applying Lemma \ref{le11}, we obtain
		\begin{align*}
			\sqrt{n}\begin{pmatrix}
				{\bf Z}_n-Z_\infty{\bf 1}\\ {\bf N}_n-Z_\infty{\bf 1}
			\end{pmatrix}=&\sqrt{n}\begin{pmatrix}
				(\widetilde{Z}_n-Z_\infty){\bf 1}+\widehat{\bf Z}_n\\ (\widetilde{Z}_n-Z_\infty){\bf 1}+\widehat{\bf N}_n
			\end{pmatrix}\\
			\to&{\cal N}\left({\bf 0},Z_\infty(1-Z_\infty)\begin{pmatrix}
				\widetilde{\bm\Sigma}_1+\widehat{\bm\Sigma}_{\bf ZZ} & \;\;\widetilde{\bm\Sigma}_1+\widehat{\bm\Sigma}_{\bf ZN}\\
				\widetilde{\bm\Sigma}_1+\widehat{\bm\Sigma}^\top_{\bf ZN} & \;\;\widetilde{\bm\Sigma}_1+\widehat{\bm\Sigma}_{\bf NN}
			\end{pmatrix}\right)\ \ \ \ \ \ \text{\textit{stably}}.
		\end{align*}
		This establishes Theorem \ref{th5}.
		
		Next, we prove Theorem \ref{th13}. By case (b) of Theorem \ref{th2}, we have
		\begin{equation*}
			\frac{\sqrt{n}}{(\log n)^{\rho-1/2}}\begin{pmatrix}
				\widehat{{\bf Z}}_n\\ \widehat{{\bf N}}_n
			\end{pmatrix}\to{\cal N}\left({\bf 0},Z_\infty(1-Z_\infty)\begin{pmatrix}
				\widehat{\bm\Sigma}^\ast_{\bf ZZ} & \;\;\widehat{\bm\Sigma}^\ast_{\bf ZN}\\
				\widehat{\bm\Sigma}^{\ast\top}_{\bf ZN} & \;\;\widehat{\bm\Sigma}^\ast_{\bf NN}
			\end{pmatrix}\right)\ \ \ \ \ \ \text{\textit{stably}},
		\end{equation*}
		Then it follows that
		\begin{align*}
			\frac{\sqrt{n}}{(\log n)^{\rho-1/2}}\begin{pmatrix}
				{\bf Z}_n-Z_\infty{\bf 1}\\ {\bf N}_n-Z_\infty{\bf 1}
			\end{pmatrix}=&\frac{\sqrt{n}}{(\log n)^{\rho-1/2}}\begin{pmatrix}
				(\widetilde{Z}_n-Z_\infty){\bf 1}+\widehat{\bf Z}_n\\ (\widetilde{Z}_n-Z_\infty){\bf 1}+\widehat{\bf N}_n
			\end{pmatrix}\\
			=&\frac{\sqrt{n}}{(\log n)^{\rho-1/2}}\begin{pmatrix}
				(\widetilde{Z}_n-Z_\infty){\bf 1}\\
				(\widetilde{Z}_n-Z_\infty){\bf 1}
			\end{pmatrix}+\frac{\sqrt{n}}{(\log n)^{\rho-1/2}}\begin{pmatrix}
				\widehat{{\bf Z}}_n\\ \widehat{{\bf N}}_n
			\end{pmatrix}  \\
			\to&{\cal N}\left({\bf 0},Z_\infty(1-Z_\infty)\begin{pmatrix}
				\widehat{\bm\Sigma}^\ast_{\bf ZZ} & \;\;\widehat{\bm\Sigma}^\ast_{\bf ZN}\\
				\widehat{\bm\Sigma}^{\ast\top}_{\bf ZN} & \;\;\widehat{\bm\Sigma}^\ast_{\bf NN}
			\end{pmatrix}\right)\ \ \ \ \ \ \text{\textit{stably}},
		\end{align*}
		where the first term vanishes in probability, and the convergence is thus determined by the second term.
		
		Finally, Theorem \ref{th4} follows from case (a) of Theorem \ref{th9}, Theorem \ref{th8}, and Theorem \ref{th2}, corresponding respectively to cases (a), (b), and (c), together with the linear invariance property of the multivariate normal distribution.
	\end{proof}

	\section{Statistical Inference}\label{sec5}
	
	A key challenge in the study of social networks is to understand and predict viral information diffusion on platforms like Twitter. Information often propagates through multi-level networks, where an agent's decision to share content is a reinforced process. This reinforcement is driven by both self-reinforcement from repeated personal exposure and social influence from the actions of other agents within their social network. From a statistical perspective, this complex dynamic creates a clear need to formally test a hypothesized influence structure against observational data and to estimate the strength of reinforcement effects with a principled measure of uncertainty.
	
	The asymptotic theory developed in the previous section provides a rigorous foundation to address these questions. In this section, we develop two applications of our CLTs as practical inferential tools. First, we construct hypothesis tests capable of statistically evaluating specific network models. Second, we provide methods for constructing confidence regions for the system's key parameters, thereby quantifying the uncertainty.

	\subsection{Hypothesis Testing for Network Structure}
	In various applied settings, agent interactions often follow a hierarchical or directional pattern. For example, in organizational command chains, decisions are passed from superiors to subordinates, who partially adopt their instructions. In information diffusion, public opinion may spread outward from a central authority. These patterns motivate the use of networks with unidirectional influence as a natural structure under the null hypothesis when applying CLTs for inference. We consider the hypothesis testing problem 
	\begin{equation}\label{eq53}
		\mathbf{H}_0:\mathbf{W}=\mathbf{W}_0\ \ {\rm vs.}\ \ \mathbf{H}_1:\mathbf{W}\neq \mathbf{W}_0,	
	\end{equation}
	which aims to verify whether the network adjacency matrix conforms to the specified structure $\mathbf{W}_0$, thereby revealing the pattern of influence within and between subgroups.

	To facilitate statistical inference, we establish the joint asymptotic distribution of $({\bf Z}_n, {\bf N}_n)_n$ in Theorems \ref{th3}, \ref{th5}, and \ref{th13}. For the regime $1/2<\gamma<1$, the asymptotic covariance matrices of both components ${\bf Z}_n$ and ${\bf N}_n$ primarily depend on the norm of ${\bf q}_1$, which provides limited insight into the structural properties of the adjacency matrix ${\bf W}$. By contrast, the limiting covariance matrix of $\widehat{\bf Z}_n$ in case (a) of Theorem \ref{th9} can be expressed as a linear combination of the eigenvalues $\lambda_u(u\in\{1,2,\ldots,T\})$ and the associated left eigenvectors ${\bf p}_i$ and right eigenvectors ${\bf q}_i(i\in\{1,2,\ldots,N\})$ of ${\bf W}$, thereby providing a more informative characterization of the network structure. When $\gamma=1$, under both $\tau<1-(2c)^{-1}$ and $\tau=1-(2c)^{-1}$, the covariance structures of $\widehat{\bf Z}_n$ and $\widehat{\bf N}_n$ in Theorem \ref{th2} likewise share this spectral representation. Since $\widehat{\bf N}_n$ does not convey additional structural information beyond that of $\widehat{\bf Z}_n$, inference based solely on $\widehat{\bf Z}_n$ suffices for effective testing while reducing complexity. Accordingly, we focus on inference procedures grounded in the CLTs for $\widehat{\bf Z}_n$ established above, which form the basis for constructing the relevant test statistics.

	To construct the test statistics, we use the vector ${\bf q}_1$ and matrices ${\bf P}$, ${\bf Q}$ derived from the null hypothesis adjacency matrix ${\bf W}_0$. Recall that
	\begin{equation*}
		\widetilde{Z}_n=N^{-1/2}{\bf q}^\top_1{\bf Z}_n,\ \ \widehat{\mathbf{Z}}_n=\mathbf{P}\mathbf{Q}^\top \mathbf{Z}_n,
	\end{equation*}
	since the synchronization limit $Z_\infty$ is generally unobservable, we replace it by its estimator $\widetilde{Z}_n$ according to Theorem \ref{th10}. When ${\bf W}={\bf W}_0$, let the ranks of the matrices $\widehat{\bm{\Sigma}}_\gamma$, $\widehat{\bm{\Sigma}}_{\bf ZZ}$ and $\widehat{\bm{\Sigma}}_{\bf ZZ}^\ast$ in Theorem \ref{th9} be $R_1$, $R_2$ and $R_3$, respectively. Moreover, denote their Moore--Penrose generalized inverses as $\widehat{\bm{\Sigma}}_\gamma^\dagger$, $\widehat{\bm{\Sigma}}_{\bf ZZ}^\dagger$, and $(\widehat{\bm{\Sigma}}_{\bf ZZ}^\ast)^\dagger$, respectively. 
	
	Then, the test statistics and their asymptotic distributions  for the hypothesis testing problem \eqref{eq53} under the three scenarios are  given in the following corollary.
	\begin{corollary}
	Assume that Assumptions \ref{as1} and \ref{as2} hold. 
    For hypothesis test  in \eqref{eq53}, we define 
	\begin{itemize}
		\item[(a)] for $1/2<\gamma<1$,
		\begin{equation}\label{eq56}
			T_{\gamma,n}=n^{\gamma}[\widetilde{Z}_n(1-\widetilde{Z}_n)]^{-1}{\bf Z}^\top_n{\bf Q}{\bf P}^\top \widehat{\bm{\Sigma}}_\gamma^\dagger {\bf P}{\bf Q}^\top{\bf Z}_n.
		\end{equation}
		
		\item[(b)] for $\gamma=1$ and $\tau<1-(2c)^{-1}$,
		\begin{equation}\label{eq57}
			T_{1,n}=n[\widetilde{Z}_n(1-\widetilde{Z}_n)]^{-1}{\bf Z}^\top_n{\bf Q}{\bf P}^\top \widehat{\bm{\Sigma}}_{\bf ZZ}^\dagger {\bf P}{\bf Q}^\top{\bf Z}_n.
		\end{equation}
		
		\item[(c)] for $\gamma=1$, $\tau=1-(2c)^{-1}$,
		\begin{equation}\label{eq58}
			T_{1,n}^\ast=\frac{n}{(\log n)^{2\rho-1}}[\widetilde{Z}_n(1-\widetilde{Z}_n)]^{-1}{\bf Z}^\top_n{\bf Q}{\bf P}^\top (\widehat{\bm{\Sigma}}_{\bf ZZ}^\ast)^\dagger {\bf P}{\bf Q}^\top{\bf Z}_n.		
		\end{equation}
	\end{itemize}
	Thus, under the null hypothesis ${\bf H}_0:{\bf W}={\bf W}_0$, it follows that
	\begin{equation*}
		T_{\gamma,n}\overset{L}{\to}\chi^2_{R_1},\ \ T_{1,n}\overset{L}{\to}\chi^2_{R_2},\ \ T_{1,n}^\ast\overset{L}{\to}\chi^2_{R_3}.
	\end{equation*}
\end{corollary}
	
	The power of the proposed test statistics largely depends on the structure of the adjacency matrix $\mathbf{W}_1$ under the alternative hypothesis ${\bf H}_1$, particularly because its spectral properties may differ from those of $\mathbf{W}_0$ under the null hypothesis ${\bf H}_0$. For instance, the eigenvectors associated with the dominant eigenvalue $1$ under ${\bf H}_0$ and ${\bf H}_1$, denoted by $\mathbf{v}_1$ and $\mathbf{v}_1^\ast$ respectively, may not coincide. Nevertheless, due to \eqref{eq54}, \eqref{eq55} and Theorem \ref{th10}, the term $\widetilde{Z}_n(1-\widetilde{Z}_n)$ remains a estimator for $Z_\infty(1-Z_\infty)$. Similarly, although the matrices $\mathbf{P}$ and $\mathbf{Q}$ constructed from $\mathbf{W}_0$ may no longer correspond to ${\bf P}^\ast,\ {\bf Q}^\ast$ of $\mathbf{W}_1$ under ${\bf H}_1$. The difference in the underlying covariance structures $\widehat{\bm{\Sigma}}_\gamma$, $\widehat{\bm{\Sigma}}_1$ and $\widehat{\bm{\Sigma}}_1^\ast$ under ${\bf H}_0$ and ${\bf H}_1$ determines the divergence between the asymptotic distributions of the test statistics, thereby governing the power of the test. When this structural deviation is sufficiently pronounced, the proposed procedures achieve high asymptotic power.

	In the following, we discuss two specific configurations of the null adjacency matrix ${\bf W}_0$, each reflecting a distinct form of hierarchical or directional influence.
	
	\begin{example}[Top--Down Influence Cascade]\label{ex1}
		The first example considers a cascade-like structure given by
		\begin{equation*}
			{\bf W}_0={\bf e}_1{\bf e}^\top_1+(1-\alpha)\sum\limits_{i=2}^{N}{\bf e}_i{\bf e}^\top_i+\alpha\sum\limits_{i=2}^{N}{\bf e}_{i-1}{\bf e}^\top_i.
		\end{equation*}
		which models a top-down influence flow commonly observed in hierarchical organizations or information cascades. The eigenvalues and eigenvectors of this matrix can be explicitly computed, with the eigenvalue $1$ being simple and $1-\alpha$ having multiplicity $N-1$. From these, the Jordan decomposition matrices ${\bf P}$, ${\bf Q}$ and vector ${\bf q}_1$ are derived explicitly, with
		\begin{equation*}
			{\bf q}_1=\sqrt{N}{\bf e}_1,\ {\bf p}_h=-\frac{1}{\sqrt{N}}{\bf e}_h\alpha^{h-2}, \ \ {\bf q}_h=\sqrt{N}\alpha^{2-h}({\bf e}_1-{\bf e}_h)\ \ {\rm for}\ h\ge2.
		\end{equation*}
		Substituting these into equations \eqref{eq4}, \eqref{eq1} and \eqref{eq68}, the corresponding covariance matrices can be explicitly calculated, 
		\begin{align*}
			&\widehat{\Sigma}_\gamma={\bf P}\widehat{S}_\gamma{\bf P}^\top,\ \ [\widehat{S}_\gamma]_{i,j}=\sum\limits_{s=0}^{j-1}\sum\limits_{t=0}^{i-1}\frac{c(t+s)}{(2\alpha)^{t+s+1}}{\bf q}^\top_{i-t+1}{\bf q}_{j-s+1},\ \ \\
			&\widehat{\Sigma}_{\bf ZZ}={\bf P}\widehat{S}_{\bf ZZ}{\bf P}^\top,\ \ [\widehat{S}_{\bf ZZ}]_{i,j}=\sum\limits_{s=0}^{j-1}\sum\limits_{t=0}^{i-1}\frac{c^{t+s+2}(t+s)!}{(-1+2\alpha c)^{t+s+1}}{\bf q}^\top_{i-t+1}{\bf q}_{j-s+1},\ \  \\
			&\widehat{\Sigma}^\ast_{\bf ZZ}={\bf P}\widehat{S}^\ast_{\bf ZZ}{\bf P}^\top,\ \ [\widehat{S}^\ast_{\bf ZZ}]_{i,j}=\frac{c^2(N-1)}{2(N-1)-1}{\bf q}^\top_{i+1}{\bf q}_{j+1}.
		\end{align*}
	\end{example}

	\begin{example}[Two-Group Hierarchical Network]\label{ex2}
		The second example focuses on a two-group hierarchical network represented by the block upper-triangular matrix
		\begin{equation*}
			{\bf W}_0=\begin{pmatrix}
				{\bf W}_{11} & {\bf W}_{12}\\
				{\bf 0} & {\bf W}_{22}
			\end{pmatrix},
		\end{equation*}
		where
		\begin{equation*}
			{\bf W}_{11}=\frac{1-\alpha}{N_1}{\bf 1}_{N_1}{\bf 1}^\top_{N_1}+\alpha{\bf I}_{N_1},\ \ {\bf W}_{12}=\frac{1-\beta}{N_1}{\bf 1}_{N_1}{\bf 1}^\top_{N_2},\ \ {\bf W}_{22}=\beta{\bf I}_{N_2}.
		\end{equation*}
		with $\alpha>\beta$.This structure captures the hierarchical influence flowing from one subgroup to another in a unidirectional manner. The eigenvalue $1$ is simple, $\alpha$ has geometric multiplicity $N_1 - 1$, and $\beta$ has multiplicity $N_2$. The associated Jordan decomposition matrices ${\bf P}$, ${\bf Q}$ and vector ${\bf q}_1$ are explicitly given as follows. The vector ${\bf q}_1=\sum_{i=1}^{N_1}{\bf e}_i$. For $2\le h\le N_1$,
		\begin{equation*}
			{\bf p}_h=\frac{1}{\sqrt{N}}\left(\sum\limits_{i=1}^{N_1}{\bf e}_i-N_1{\bf e}_h\right),\ \ {\bf q}_h=\frac{\sqrt{N}}{N_1}({\bf e}_1-{\bf e}_h),
		\end{equation*}
		and for $N_1< h\le N$,
		\begin{equation*}
			{\bf p}_h=\frac{N_1}{\sqrt{N}}{\bf e}_h,\ \ {\bf q}_h=\frac{\sqrt{N}}{N_1}\left(-\frac{1}{N_1}\sum\limits_{i=1}^{N_1}{\bf e}_i+{\bf e}_h\right).
		\end{equation*}
		These bases facilitate the explicit derivation of covariance matrix entries, which depend on subgroup dimensions and parameters, thereby yielding a block-structured covariance 
		\begin{equation*}
			\widehat{\Sigma}_\gamma={\bf P}\widehat{S}_\gamma{\bf P}^\top,\ \ \widehat{\Sigma}_{\bf ZZ}={\bf P}\widehat{S}_{\bf ZZ}{\bf P}^\top,\ \ \widehat{\Sigma}^\ast_{\bf ZZ}={\bf P}\widehat{S}^\ast_{\bf ZZ}{\bf P}^\top.
		\end{equation*}
		where for $1\le i\le N_1-1$ and $1\le j\le N_1-1$,
		\begin{align*}
			&[\widehat{S}_\gamma]_{i,j}=\sum\limits_{s=0}^{j-1}\sum\limits_{t=0}^{i-1}\frac{c(t+s)}{(2-2\alpha)^{t+s+1}}{\bf q}^\top_{i-t+1}{\bf q}_{j-s+1},\ \ [\widehat{S}^\ast_{\bf ZZ}]_{i,j}=\frac{c^2(N_1-1)}{2(N_1-1)-1}{\bf q}^\top_{i+1}{\bf q}_{j+1},\\
			&[\widehat{S}_{\bf ZZ}]_{i,j}=\sum\limits_{s=0}^{j-1}\sum\limits_{t=0}^{i-1}\frac{c^{t+s+2}(t+s)!}{[-1+(2-2\alpha) c]^{t+s+1}}{\bf q}^\top_{i-t+1}{\bf q}_{j-s+1},
		\end{align*}
		for $1\le i\le N_1-1$ and $N_1\le j\le N-1$,
		\begin{align*}
			&[\widehat{S}_\gamma]_{i,j}=\sum\limits_{s=0}^{j-(N_1-1)-1}\sum\limits_{t=0}^{i-1}\frac{c(t+s)}{(2-\alpha-\beta)^{t+s+1}}{\bf q}^\top_{i-t+1}{\bf q}_{j-s+1},\ \ [\widehat{S}^\ast_{\bf ZZ}]_{i,j}=0,\\
			&[\widehat{S}_{\bf ZZ}]_{i,j}=\sum\limits_{s=0}^{j-(N_1-1)-1}\sum\limits_{t=0}^{i-1}\frac{c^{t+s+2}(t+s)!}{[-1+(2-\alpha-\beta) c]^{t+s+1}}{\bf q}^\top_{i-t+1}{\bf q}_{j-s+1}.
		\end{align*}
		and for $N_1\le i\le N-1$ and $N_1\le j\le N-1$,
		\begin{align*}
			&[\widehat{S}_\gamma]_{i,j}=\sum\limits_{s=0}^{j-(N_1-1)-1}\sum\limits_{t=0}^{i-(N_1-1)-1}\frac{c(t+s)}{(2-2\beta)^{t+s+1}}{\bf q}^\top_{i-t+1}{\bf q}_{j-s+1},\ \ [\widehat{S}^\ast_{\bf ZZ}]_{i,j}=0,\\
			&[\widehat{S}_{\bf ZZ}]_{i,j}=\sum\limits_{s=0}^{j-(N_1-1)-1}\sum\limits_{t=0}^{i-(N_1-1)-1}\frac{c^{t+s+2}(t+s)!}{[-1+(2-2\beta) c]^{t+s+1}}{\bf q}^\top_{i-t+1}{\bf q}_{j-s+1}.
		\end{align*}
	\end{example}
	In the above examples, by computing the Moore--Penrose generalized inverses of the covariance matrices and substituting them into \eqref{eq56}--\eqref{eq58}, we obtain the test statistics required for statistical inference. These two examples enhance the theoretical understanding while representing practically relevant hierarchical and unidirectional influence networks, and provide concrete settings for applying the proposed hypothesis testing procedures to hierarchical and unidirectional influence networks in reinforced stochastic systems.

	\subsection{Construction of Confidence Regions}
	
	Beyond hypothesis testing, our asymptotic results provide a framework for constructing confidence regions for the model's key parameters, thereby quantifying the uncertainty of parameter estimates.

	First, we can construct a confidence interval for the synchronization limit $Z_\infty$. The basis for this is the asymptotic normality of its estimator, $\widetilde{Z}_n$. Specifically, Theorem \ref{th3} establishes the following stable convergence,
	\begin{equation*}
		n^{\gamma-\frac{1}{2}}(\widetilde{Z}_n-Z_\infty) \to {\cal N}(0,Z_\infty(1-Z_\infty)\widetilde{\sigma}^2_\gamma),
	\end{equation*}
	where $\widetilde{\sigma}^2_\gamma=\frac{c^2\|{\bf q}_1\|^2}{N(2\gamma-1)}$. Since the true limit $Z_\infty$ in the variance term is unknown, we replace it with its consistent estimator $\widetilde{Z}_n$. This yields an approximate $100(1-\alpha)\%$ confidence interval for $Z_\infty$ given by
	\begin{equation}\label{eq:ci_z_infty}
		\widetilde{Z}_n \pm z_{\alpha/2} \sqrt{n^{-(2\gamma-1)} \widetilde{Z}_n(1-\widetilde{Z}_n) \widetilde{\sigma}^2_\gamma},
	\end{equation}
	where $z_{\alpha/2}$ is the upper $\alpha/2$ quantile of the standard normal distribution.

	Second, we develop a method to construct confidence regions for the network's key structural parameters, such as the influence parameter $\alpha$ in Example \ref{ex1} or the vector ${\bm\theta}=(\alpha,\beta)$ in Example \ref{ex2}. The procedure is based on the fundamental principle of inverting a hypothesis test. Specifically, a $100(1-\alpha)\%$ confidence region for a true parameter vector is the set of all parameter values for which the corresponding null hypothesis would not be rejected at the significance level $\alpha$. To illustrate, consider constructing a confidence region for the parameter vector ${\bm\theta}=(\alpha,\beta)$ from Example \ref{ex2}. The choice of the test statistic depends on the specific parameter regime. Assuming the system falls within the regime where $\gamma=1$ and $\tau<1-(2c)^{-1}$, we use the corresponding test statistic $T_{1,n}$ from \eqref{eq57}. The resulting $100(1-\alpha)\%$ confidence region is then given by
	\begin{equation}
		\left\{{\bm\theta}:T_{1,n}\le\chi^2_{R_2,1-\alpha}\right\},
	\end{equation}
	where $T_{1,n}$ is the test statistic treated as a function of the parameter vector ${\bm\theta}$, and $\chi^2_{R_2,1-\alpha}$
	is the corresponding critical value from the Chi-squared distribution with $R_2$ degrees of freedom. This provides a practical tool for estimating and quantifying the uncertainty of the underlying influence structure from observed data.

	\section{Simulation Studies}\label{sec7}
	While the almost sure convergence of the joint process $({\bf Z}_n, {\bf N}_n)_n$ can be established via martingale arguments, the distribution of the synchronization limit $Z_\infty$ in such models remains largely unexplored. Therefore, in this section, we conduct simulation studies to numerically explore the properties of the limit $Z_\infty$. Our simulations are designed to investigate two main aspects: first, the overall shape of the limit distribution under various initial conditions; second, the prevalence of polarization, i.e., convergence to the boundaries 0 or 1.

	In our simulations, we consider a two-group hierarchical network with a total of $N=4$ agents. The population is partitioned into a leading subgroup ${\cal G}_1$ of size $N_1=2$ and a downstream subgroup ${\cal G}_2$ of size $N_2=2$. The interactions are governed by the following adjacency matrix ${\bf W}$,
	\begin{equation}
		{\bf W}=\begin{pmatrix}
			\alpha & 1-\alpha & (1-\beta)/2 & (1-\beta)/2 \\
			1-\alpha & \alpha & (1-\beta)/2 & (1-\beta)/2 \\
			0 & 0 & \beta & 0 \\
			0 & 0 & 0 & \beta
		\end{pmatrix}.
	\end{equation}
	In this setup, we fix the downstream self-reinforcement parameter at $\beta=0.5$. To investigate the impact of the leading group's internal structure, we compare two scenarios for its self-reinforcement parameter $\alpha$: a strong case with $\alpha=0.8$ and a weak case with $\alpha=0.2$.
	
	For each of these two network structures, we test six initial state ${\bf Z}_0$ configurations. These configurations are designed to explore the system's behavior under different initial states and are formed by pairing three distinct scenarios for the leading subgroup ${\cal G}_1$ with two for the downstream subgroup ${\cal G}_2$. We consider three scenarios for the initial state of ${\cal G}_1$: a consensus scenario with ${\bf Z}^{(1)}_0=(0.5,0.5)^\top$; an asymmetric scenario with ${\bf Z}^{(1)}_0=(0.1,0.5)^\top$; and a random scenario where the components are drawn i.i.d. from ${\rm U}(0,1)$. Each of these leading group configurations is then paired with downstream initial states for ${\cal G}_2$ of either all zeros, ${\bf Z}^{(2)}_0=(0,0)^\top$, or all ones, ${\bf Z}^{(2)}_0=(1,1)^\top$. The parameters for the step-size sequence $r_n\sim cn^{-\gamma}$ are set to $\gamma=0.9$ and $c=1$. Each simulation for a given scenario is run for $n_{\rm steps}=20000$ iterations, and this process is repeated independently for $n_{\rm sims}=5000$ times to obtain the empirical distribution of the final states.

	Figures \ref{fig1} and \ref{fig2} visualize the simulation results for the strong ($\alpha=0.8$) and weak ($\alpha=0.2$) self-reinforcement cases, respectively. Each figure presents a $3 \times 2$ grid of panels, where each panel displays the estimated probability densities of the final states. 
	The figures visually confirm several theoretical findings. First, within each panel, the density curves for the four agents are indistinguishable, providing strong evidence for the synchronization proven in Theorem \ref{th7}. Second, a row-wise comparison demonstrates the system's invariance to the initial states of the downstream group, which provides strong visual support for Theorem \ref{th11}. Finally, a column-wise comparison reveals the limit distribution's high sensitivity to the initial state of the leading subgroup, corroborating the weighted average structure established in Corollary \ref{co1}. Moreover, comparing the corresponding panels of Figure \ref{fig1} and Figure \ref{fig2} reveals that for the symmetric leading group under consideration, the distribution of the synchronization limit $Z_\infty$ is robust to the change of self-reinforcement parameter $\alpha$.

	\begin{figure}[htbp]
		\centering
		\includegraphics[width=0.9\textwidth]{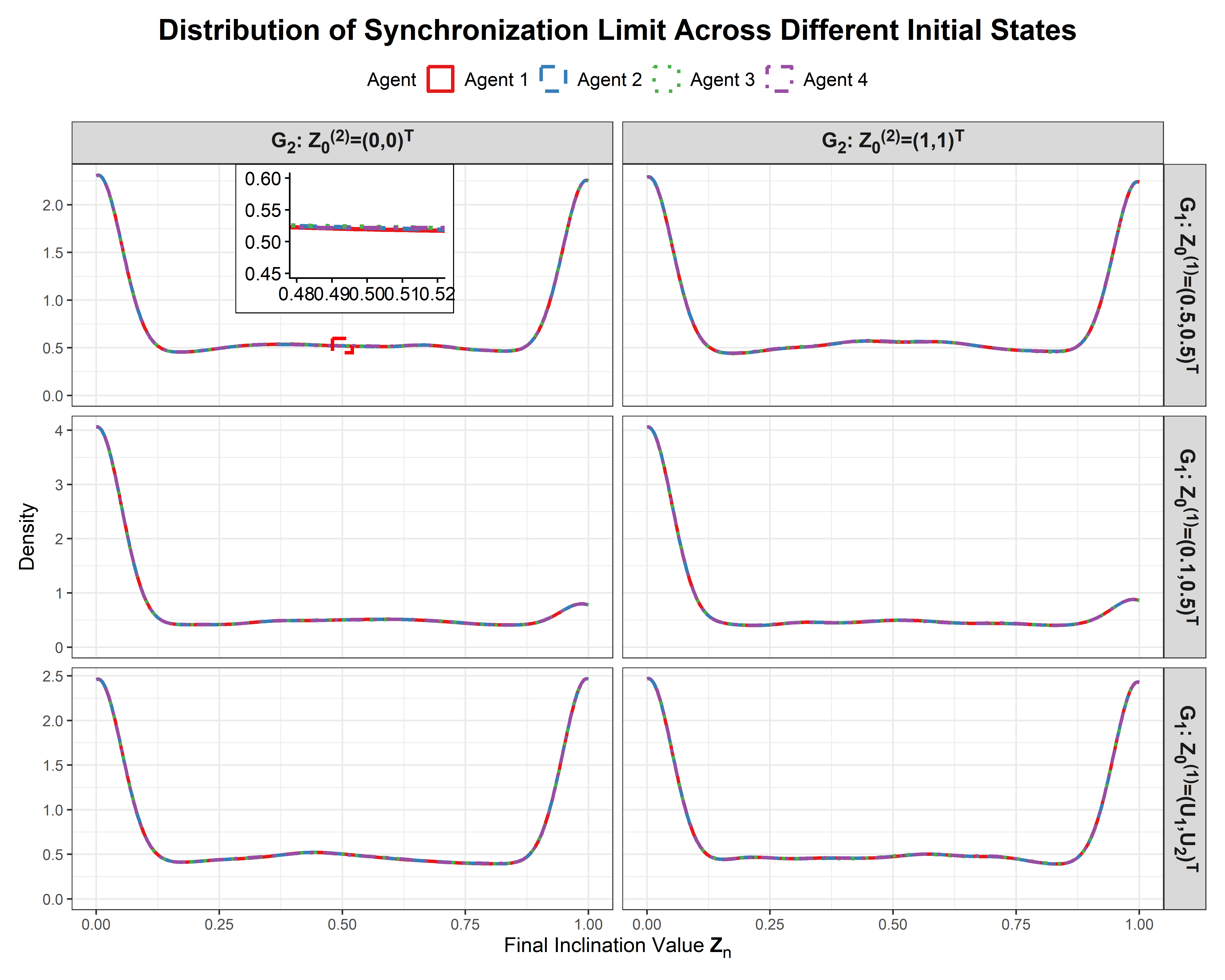}
		\vspace{-15pt}
		\caption{Distribution of the synchronization limit $Z_\infty$ for the two-group hierarchical network with strong self-reinforcement ($\alpha=0.8$). The six panels correspond to different initial states for the leading ${\cal G}_1$ and subsequent ${\cal G}_2$ subgroups. Within each panel, the four overlapping density curves represent the four agents.}
		\label{fig1}
	\end{figure}

	\begin{figure}[htbp]
		\centering
		\includegraphics[width=0.9\textwidth]{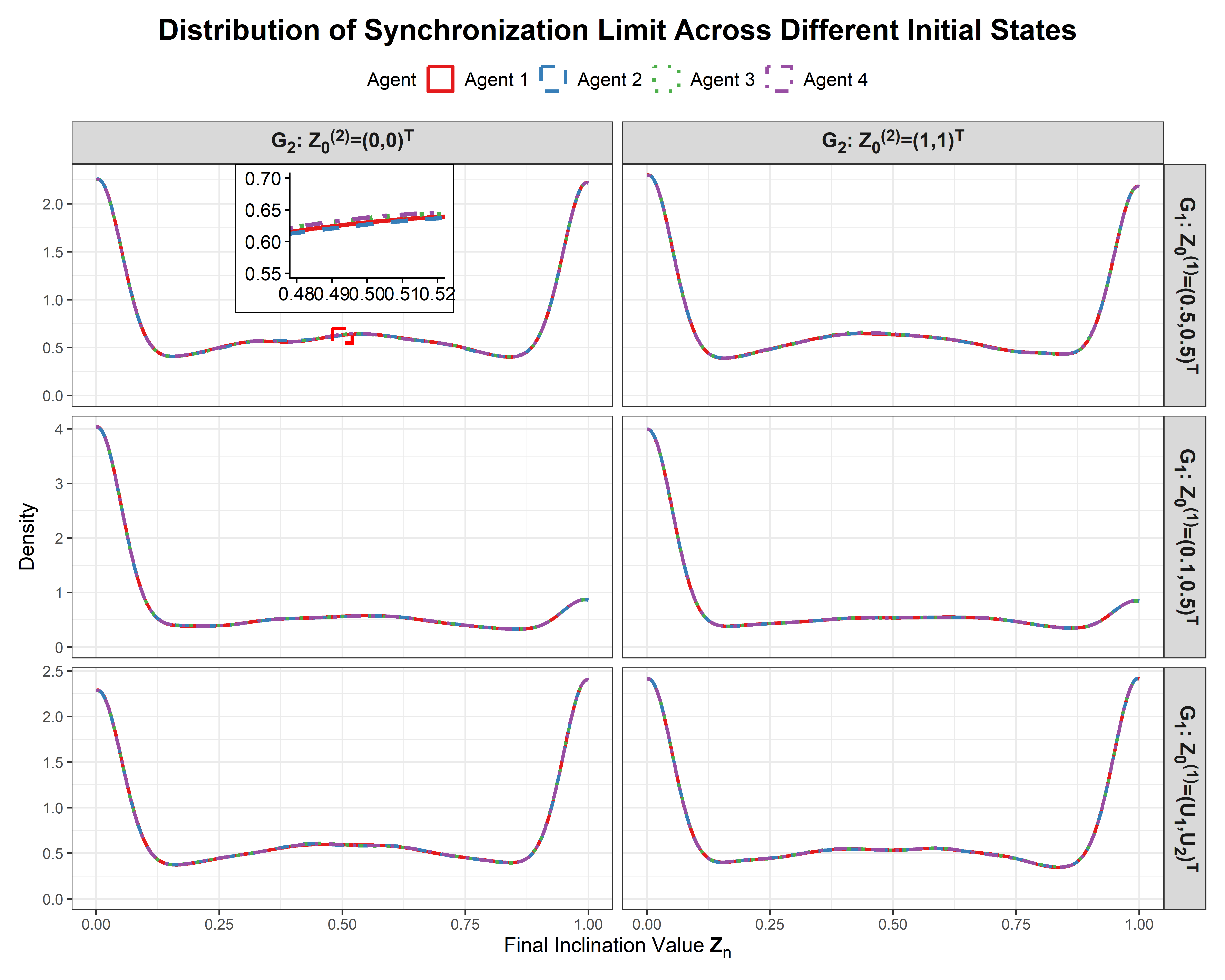}
		\vspace{-15pt}
		\caption{Distribution of the synchronization limit $Z_\infty$ for the two-group hierarchical network with weak self-reinforcement ($\alpha=0.2$). The six panels correspond to different initial states for the leading ${\cal G}_1$ and subsequent ${\cal G}_2$ subgroups. Within each panel, the four overlapping density curves represent the four agents.}
		\label{fig2}
	\end{figure}
	\begin{table}[htbp]
		\centering 
		\caption{Percentage of final states falling into boundary and central regions for different leading group structures $\alpha$ and initial states ${\bf Z}_0$. For the random cases, the initial states of the leading subgroup ($U_1, U_2$) are drawn independently from a ${\rm U}[0,1]$ distribution for each simulation run.}		
		\label{tab1}
		\begin{tabular}{llcccc}
			\toprule
			\multicolumn{2}{c}{Scenario} & \multicolumn{4}{c}{Percentage of final states in interval (\%)} \\
			\cmidrule(lr){1-2} \cmidrule(lr){3-6}
			$\alpha$ & ${\bf Z}_0^\top$ & $[0, 0.05]$ & $(0.05, 0.95)$ & $[0.95, 1]$ & At boundary (0 or 1) \\
			\midrule
			\multirow{6}{*}{\shortstack{\textbf{0.8} \\ (Strong)}} 
			& $(0.5,0.5,0,0)$    & 27.89 & 44.77 & 27.34 & 38.31 \\
			& $(0.5,0.5,1,1)$    & 28.47 & 44.58 & 26.95 & 38.20 \\
			& $(0.1,0.5,0,0)$    & 50.23 & 39.31 & 10.46 & 49.38 \\
			& $(0.1,0.5,1,1)$    & 49.66 & 40.20 & 10.14 & 38.81 \\
			& $(U_1,U_2,0,0)$ & 28.20 & 42.01 & 29.79 & 42.87 \\
			& $(U_1,U_2,1,1)$    & 29.96 & 40.09 & 29.95 & 43.86 \\
			\midrule
			\multirow{6}{*}{\shortstack{\textbf{0.2} \\ (Weak)}} 
			& $(0.5,0.5,0,0)$    & 28.42 & 44.00 & 27.58 & 38.19 \\
			& $(0.5,0.5,1,1)$    & 28.09 & 44.46 & 27.45 & 38.10 \\
			& $(0.1,0.5,0,0)$    & 50.49 & 40.27 & 9.24  & 48.43 \\
			& $(0.1,0.5,1,1)$    & 50.37 & 39.39 & 10.24 & 38.87 \\
			& $(U_1,U_2,0,0)$    & 30.21 & 39.29 & 30.50 & 44.08 \\
			& $(U_1,U_2,1,1)$    & 30.50 & 39.48 & 30.02 & 43.98 \\
			\bottomrule
		\end{tabular}
	\end{table}
	
	To quantitatively analyze the system's behavior, we report in Table \ref{tab1} the percentage of simulation runs where the final state falls into several key regions. As synchronization theory ensures all agents converge to the same limit, these percentages are computed from the pooled data of all four agents for each scenario to provide a more stable estimate of the $Z_\infty$ distribution.
	
	The results in Table \ref{tab1} lead to several key observations. First, they offer numerical evidence that polarization is a prevalent outcome. For instance, in the asymmetric case with self-reinforcement $\alpha=0.8$, the total proportions of runs converging to the boundaries are high. While the total percentages in the two paired scenarios ($49.38\%$ vs. $\ 38.81\%$) appear different, a closer look reveals that the constituent proportions converging to the $[0, 0.05]$ region ($50.23\%$ vs.\ $49.66\%$) and the $[0.95, 1]$ region ($10.46\%$ vs. $10.14\%$) are remarkably close. This slight discrepancy in the totals is attributable to finite-time effects, while the consistency of the components provides numerical corroboration for the conclusion of Theorem \ref{th11} that the limit distribution is independent of the downstream group's initial state. Second, the table also quantitatively confirms the robustness of the system to the self-reinforcement parameter $\alpha$, a phenomenon visually observed in Figures \ref{fig1} and \ref{fig2}. A comparison between the strong ($\alpha=0.8$) and weak ($\alpha=0.2$) cases shows the percentages in each corresponding interval being nearly identical across all initial conditions. This finding confirms that for the specific symmetric structure of the leading group considered, the strength of the self-reinforcement parameter $\alpha$ has a negligible impact on the final distribution of the synchronization limit.
	
	These results underscore a crucial practical implication: the system's ultimate consensus is critically dependent on the initial state of the leading group, while being remarkably robust to both the downstream followers' initial state and the leaders’ self-reinforcement strength $\alpha$.

	\begin{appendix}
    \section*{Auxiliary functions}\label{sec0}
		The auxiliary functions $\mathcal{H}(\cdot)$, $\mathcal{N}_i(\cdot)$, and $\mathcal{D}_i(\cdot)$, which appear in the covariance matrix components of Theorem \ref{th5}, are defined as follows. The functions $\mathcal{N}_i(\cdot)$ are
\begin{align*}
    {\cal N}_1(k,m,\lambda_b,c) &:= \sum_{q=0}^{k-m}[c(1-\lambda_b)]^q,  \\
    {\cal N}_2(k,\lambda_a,\lambda_b,c) &:= \sum_{q=0}^{k}{k+1\choose q}[c(1-\lambda_a)]^q[c(1-\lambda_b)-1]^{k-q},  \\
    {\cal N}_3(k,\lambda_a,c) &:= [c(1-\lambda_a)]^{k+1}, \\
    {\cal N}_4(k,\lambda_a,c) &:=\sum\limits_{q=0}^k[c(1-\lambda_a)-1]^{k-q}.
\end{align*}
The functions $\mathcal{D}_i(\cdot)$ are
\begin{align*}
    {\cal D}_1(k, m, \lambda_a, \lambda_b, c) &:= [c(1-\lambda_a)]^{k+1}[c(1-\lambda_b)]^{k+1-m}, \\
    {\cal D}_2(k, m, \lambda_a, \lambda_b, c) &:= [c(1-\lambda_a)]^{k+1}[c(1-\lambda_b)]^{k+1-m}[-1+c(2-\lambda_a-\lambda_b)]^{k+1}.
\end{align*}
Finally, the function $\mathcal{H}(\cdot)$ is
\begin{align*} 
{\cal H}(k,\lambda_a,\lambda_b,c;C_1,C_2,C_3)&:= k!\sum_{m=0}^{k}{k+1 \choose m}[c(1-\lambda_a)-1]^{k-m}\cdot \\
&\quad\Bigg\{\frac{C_1{\cal N}_1(k,m,\lambda_b,c)}{{\cal D}_1(k, m, \lambda_a, \lambda_b, c)}+\frac{C_2{\cal N}_2(k,\lambda_a,\lambda_b,c)+C_3{\cal N}_3(k,\lambda_a,c)}{{\cal D}_2(k,m,\lambda_a,\lambda_b,c)}\Bigg\}.
\end{align*}

\setcounter{section}{0}             
\renewcommand{\thesection}{\Alph{section}} 

		\section{Technical results}\label{secA}
		This section presents the detailed derivation of the leading-order terms in the second-order convergence analysis. The proofs involve carefully tracking the asymptotic behavior of higher-order remainders and establishing their dominance relations. For brevity, the detailed proofs of the following lemmas are deferred to the Supplementary Material.

		\begin{lemma}[Lemma A.4 of \cite{r1}]\label{le3}
			For $j\in\{1,2,\cdots,S\}$ and for any $\varepsilon \in(0,1)$, we have that
			\begin{equation*}
				|p_{n,j}|=\begin{cases}O\Big(\exp \Big[-(1-\varepsilon) \frac{c (1-{\rm Re}(\lambda_j))}{1-\gamma} n^{1-\gamma}\Big]\Big) & \text { for } 1/2<\gamma<1, \\
					O\left(n^{-(1-\varepsilon) c (1-{\rm Re}(\lambda_j))}\right) & \text { for } \gamma=1\end{cases}
			\end{equation*}
			and
			\begin{equation*}
				|\ell_{n,j}|= \begin{cases}O\Big(\exp \Big[(1+\varepsilon) \frac{c (1-{\rm Re}(\lambda_j))}{1-\gamma} n^{1-\gamma}\Big]\Big) & \text { for } 1/2<\gamma<1, \\ O\big(n^{(1+\varepsilon) c (1-{\rm Re}(\lambda_j))}\big) & \text { for } \gamma=1.\end{cases}	
			\end{equation*}
			Moreover, if we replace \eqref{eq11} with the condition
			\begin{equation*}
				n^\gamma r_n-c=O\left(n^{-\gamma}\right),
			\end{equation*}
			we have that
			\begin{equation}
				|p_{n,j}|= \begin{cases}O\Big(\exp \Big[-\frac{c (1-{\rm Re}(\lambda_j))}{1-\gamma} n^{1-\gamma}\Big]\Big) & \text { for } 1/2<\gamma<1, \\
					O\left(n^{-c (1-{\rm Re}(\lambda_j))}\right) & \text { for } \gamma=1\end{cases}.
			\end{equation}
			and
			\begin{equation}
				\left|\ell_{n, j}\right|= \begin{cases}O\Big(\exp \Big[\frac{c(1-{\rm Re}(\lambda_j))}{1-\gamma} n^{1-\gamma}\Big]\Big) & \text { for } 1/2<\gamma<1, \\ 
					O\left(n^{c (1-{\rm Re}(\lambda_j))}\right) & \text { for } \gamma=1.\end{cases}
			\end{equation}
		\end{lemma}

		\begin{lemma}\label{le2}
			For a fixed $q$, we have that 
			\begin{equation*}
				\sum\limits_{k+1\le j_1\neq\cdots\neq j_q\le n}r_{j_1}\cdots r_{j_q}=\begin{cases}
					O((n^{1-\gamma}-k^{1-\gamma})^q)\  &\text { for } 1/2<\gamma<1,\\
					O((\log n-\log k)^q)\  &\text { for } \gamma=1.
				\end{cases}
			\end{equation*}
		\end{lemma}

		\begin{lemma}\label{le1}
			Let ${\bf T}^{(s)}_{k+1,n}=\prod\limits_{j=k+1}^n[{\bf I}-r_j({\bf I}-{\bf J}^\top_s)]$. Then, for all $t\in\{1,\dots,\rho_s\}$, the diagonal entry  $[{\bf T}^{(s)}_{k+1,n}]_{t,t}$ is
			\begin{equation*}
				[{\bf T}^{(s)}_{k+1,n}]_{t,t}=\prod\limits_{j=k+1}^n[1-r_j(1-\lambda_s)]=p_{n,s}l_{k,s},
			\end{equation*}
			and for all $t\in\{1,\dots,\rho_s\}$, $q\in\{1,\dots,\rho_s-1\}$, the off-diagonal entry $[{\bf T}^{(s)}_{k+1,n}]_{t,t-q}$ is
			\begin{align*}
				[{\bf T}^{(s)}_{k+1,n}]_{t,t-q}=&\sum\limits_{k+1\le j_1\neq\cdots\neq j_q\le n}\frac{r_{j_1}\cdots r_{j_q} }{[1-r_{j_1}(1-\lambda_s)]\cdots[1-r_{j_q}(1-\lambda_s)]}\prod\limits_{j=k+1}^n[1-r_j(1-\lambda_s)]\\
				&=\sum\limits_{k+1\le j_1\neq\cdots\neq j_q\le n}\frac{r_{j_1}\cdots r_{j_q} }{[1-r_{j_1}(1-\lambda_s)]\cdots[1-r_{j_q}(1-\lambda_s)]}p_{n,s}l_{k,s}.
			\end{align*}
			Moreover, for a fixed constant $m_0$, it holds that
			\begin{equation}\label{eq17}
				|[{\bf T}^{(s)}_{m_0,n}]_{t,t-q}|=\begin{cases}
					O\Big(n^{(1-\gamma)q}\exp \Big[-(1-\varepsilon) \frac{c(1-{\rm Re}(\lambda_s))}{1-\gamma} n^{1-\gamma}\Big]\Big) & \text { for } 1/2<\gamma<1, \\
					O\left((\log n)^qn^{-(1-\varepsilon) c (1-{\rm Re}(\lambda_s))}\right) & \text { for } \gamma=1.
				\end{cases}
			\end{equation}
			Additionally, if the second condition of Assumption \ref{as2} holds, the $\varepsilon$ in the above expression can be removed.
		\end{lemma}

		\begin{lemma}\label{le13}
			The following holds:
			\begin{align*}
				R^{(q,u)}_{n,k}=&\sum\limits_{k+1\le j_1\neq \cdots\neq j_q\le n}\frac{r_{j_1}\cdots r_{j_q}}{[1-r_{j_1}(1-\lambda_u)]\cdots[1-r_{j_q}(1-\lambda_u)]}\\
				=&\begin{cases}
					\left(\frac{c}{1-\gamma}\right)^q(n^{1-\gamma}-k^{1-\gamma})^q\psi_1(k,n,\gamma)\ &\text{for}\ 1/2<\gamma<1,\\
					c^q(\log n-\log k)^q\psi_2(k,n,\gamma)\ &\text{for}\ \gamma=1.
				\end{cases}
			\end{align*}
			where $\psi_1$ and $\psi_2$ are functions such that $\psi_1\to1$ and $\psi_2\to1$ as $k\to\infty$.
		\end{lemma}

		We recall that $\alpha_u=1-\lambda_u$ for $u\in\{1,2,\cdots,T\}$.

		\begin{lemma}\label{leww1}
			Let the matrix ${\bf C}_{k+1,n}$ be defined as in \eqref{eq77}. Then, for all $1\le u\le T$, ${\cal I}_{u-1}\le i\le {\cal I}_u$, and $0\le t\le i-1$, $1\le s\le i-1$,
			\begin{align}
				&[{\bf C}^{11}_{k+1,n}]_{i,i-t}\sim c^t(\log n-\log k)^tF_{k+1,n}(\alpha_u),\label{eq79}\\
				&[{\bf C}^{33}_{k+1,n}]_{i,i}=c^{22}_{k+1,n}=F_{k+1,n}(c^{-1}),\label{eq80}\\
				&[{\bf C}^{31}_{k+1,n}]_{i,i}=\begin{cases}
					\frac{1-\alpha_u}{c\alpha_u-1}[F_{k+1,n}(c^{-1})-F_{k+1,n}(\alpha_u)]\ \ &\text{for  }c\alpha_j\neq 1,\\
					(1-c^{-1})F_{k+1,n}(c^{-1})(\log n-\log k)+O(n^{-1})\ \ &\text{for  }c\alpha_j=1,
				\end{cases}\label{eq81}\\
				&[{\bf C}^{31}_{k+1,n}]_{i,i-s}\sim
				\left[c^{s-1}(\log n-\log k)^{s-1}-(1-\alpha_u)c^s(\log n-\log k)^s\right]\cdot\\
				&\qquad\qquad\qquad	\begin{cases}\frac{1}{c\alpha_u-1}[F_{k+1,n}(c^{-1})-F_{k+1,n}(\alpha_u)]\ \ &\text{for  }c\alpha_j\neq 1,\\
					\frac{1-c^{-1}}{1-\alpha_u}F_{k+1,n}(c^{-1})(\log n-\log k)+O(n^{-1})\ \ &\text{for  }c\alpha_j=1.
				\end{cases}\label{eq82}
			\end{align}
		\end{lemma}

		To facilitate the subsequent analysis, we begin by introducing the following notation.
		Define 
		\begin{equation*}
			f_t(x)=x^{t-1-t\gamma},\ \ t=1,2,\cdots,q+1,
		\end{equation*}
		and define the symbol $f_t^{[s]}(x)$ as the coefficient of the $s$th derivative of the function $f_t(x)$. Moreover, we define the following two sequences. For any $s=1,2,\cdots$, 
		\begin{equation*}
			\widetilde{f}_{h,t}=\begin{cases}
				\frac{c}{(s-1)!}\alpha_i\alpha_jr^2_kf_t^{[s-1]}(x)\ \ &\text{\rm for}\ \  h=3s-2,\\
				-\frac{c(\alpha_i+\alpha_j)}{s!}r_kf_t^{[s]}(x)\ \ &\text{\rm for}\ \ h=3s-1,\\
				\frac{c}{(s+1)!}\alpha_i\alpha_jf_t^{[s+1]}(x)\ \ &\text{\rm for}\ \ h=3s,\\
			\end{cases}
		\end{equation*}
		where we denote $\frac{c}{s-1}=-c$ when $s=1$. Furthermore, we define 
		\begin{equation*}
			\widetilde{R}_{h,t}=\begin{cases}
				-\gamma-(s-1)+(t-1)(1-\gamma)\ \ &\text{\rm for}\ \ h=3s-2,\\
				-\gamma-1-(s-1)+(t-1)(1-\gamma)\ \ &\text{\rm for}\ \ h=3s-1,\\
				-\gamma-2-(s-1)+(t-1)(1-\gamma)\ \ &\text{\rm for}\ \ h=3s.\\
			\end{cases}
		\end{equation*}

		The covariance matrix of $\widehat{\bf Z}_n$ involves summation terms over $k$ of the form $(n^{1-\gamma} - k^{1-\gamma})^q r^2_k l_{k,i} l_{k,j}$ and $(\log n - \log k)^q r^2_k l_{k,i} l_{k,j}$. A binomial expansion of these expressions, followed by a term-by-term analysis of their convergence and convergence rates, reveals that each component shares the same rate of convergence. Moreover, after being multiplied by this rate, the limit of the sum of all terms is zero. This indicates that the convergence rate derived from the binomial expansion is faster than the exact rate of convergence. Therefore, it is necessary to identify the leading-order terms in order to determine the precise convergence rate. Based on the above definitions, we now present the following lemma.
		
		\begin{lemma}\label{le4}
			The following holds:
			
			{\rm (a)} If $1/2<\gamma<1$, for any positive integer $q$, we define		
			\begin{align*}
				S_{1,n}&=\sum\limits_{k=m_0}^{n-1}\frac{ck^{-\gamma}r_k}{p_{k,i}p_{k,j}},\ \ S_{2,n}=\sum\limits_{k=m_0}^{n-1}\frac{ck^{-\gamma}r_kk^{1-\gamma}}{p_{k,i}p_{k,j}},\ \ \cdots,\ \ S_{q+1,n}=\sum\limits_{k=m_0}^{n-1}\frac{ck^{-\gamma}r_kk^{q(1-\gamma)}}{p_{k,i}p_{k,j}},\\
				G_{1,k}&=\frac{c}{k^\gamma p_{k,i}p_{k,j}},\ \ G_{2,k}=\frac{ck^{1-\gamma}}{k^\gamma p_{k,i}p_{k,j}},\ \ \cdots,\ \ G_{q+1,k}=\frac{ck^{q(1-\gamma)}}{k^\gamma p_{k,i}p_{k,j}}.
			\end{align*}
			Then, for any $1\le t\le q+1$ and positive integer $p$, we have
			\begin{align}\label{eq33}
				S_{t,n}=\frac{G_{t,n}}{\alpha_i+\alpha_j}&-\frac{c(t-1-t\gamma)}{\alpha_i+\alpha_j}\sum\limits_{k=m_0}^{n-1}k^{(t-2)(1-\gamma)-2\gamma}l_{k,i}l_{k,j}   \non\\
				&-\frac{1}{\alpha_i+\alpha_j}\sum\limits_{k=m_0}^{n-1}\sum\limits_{h=1}^{q+p-2}\widetilde{f}_{h,t}k^{\widetilde{R}_{h,t}}l_{k,i}l_{k,j}+O\Big(\sum\limits_{k=m_0}^{n-1}k^{\widetilde{R}_{q+p-2,t}-1}l_{k,i}l_{k,j}\Big).
			\end{align}

			{\rm (b)} If $\gamma=1$ and $c(\alpha_i+\alpha_j)\neq1$, for any positive integer $q$, we define
			\begin{align*}
				P_{1,n}&=\sum\limits_{k=m_0}^{n-1}\frac{c^2k^{-2}}{p_{k,i}p_{k,j}},\ \ P_{2,n}=\sum\limits_{k=m_0}^{n-1}\frac{c^2k^{-2}\log k}{p_{k,i}p_{k,j}},\ \ \cdots,\ \ P_{q+1,n}=\sum\limits_{k=m_0}^{n-1}\frac{c^2k^{-2}(\log k)^q}{p_{k,i}p_{k,j}},\\
				D_{1,k}&=\frac{c^2}{k p_{k,i}p_{k,j}},\ \ D_{2,k}=\frac{c^2\log k}{k p_{k,i}p_{k,j}},\ \ \cdots,\ \ D_{q+1,k}=\frac{c^2(\log k)^q}{k p_{k,i}p_{k,j}}.
			\end{align*}
			Then, for any $1\le t\le q+1$,  we have
			\begin{equation}\label{eq29}
				P_{t,n}=\frac{D_{t,n}}{c(\alpha_i+\alpha_j)-1}-\frac{t-1}{c(\alpha_i+\alpha_j)-1}P_{t-1,n}+O\Big(\sum\limits_{k=m_0}^{n-1}k^{-3}(\log k)^{t-1}|\Delta P_{1,k}|\Big).
			\end{equation}

			{\rm (c)} If $\gamma=1$, $c(\alpha_i+\alpha_j)=1$, for any positive integer $q$, we further define 
			\begin{equation}\label{eq41}
				D_{\ln,1,k}=\frac{c^2\log k}{k p_{k,i}p_{k,j}},\ \ D_{\ln,2,k}=\frac{c^2(\log k)^2}{k p_{k,i}p_{k,j}},\ \ \cdots,\ \ D_{\ln,q+1,k}=\frac{c^2(\log k)^{q+1}}{k p_{k,i}p_{k,j}}.
			\end{equation}
			Then, for any $1\le t\le q+1$,  we have
			\begin{equation}\label{eq42}
				P_{t,n}=\frac{D_{\ln,t,n}}{t}+O\Big(\sum\limits_{k=m_0}^{n-1}k^{-3}(\log k)^tl_{k,i}l_{k,j}\Big).
			\end{equation}
		\end{lemma}

		\begin{remark}
			For $1/2<\gamma<1$ and $\gamma=1$, Taylor expansions of different orders are employed for $f_t(x)$ due to the distinct forms of the corresponding summation terms. Specifically, the expressions take the forms $(n^{1-\gamma}-k^{1-\gamma})^q r_k^2 l_{k,i}l_{k,j}$ and $(\log n-\log k)^q r_k^2 l_{k,i}l_{k,j}$, respectively.  
			In the case $1/2<\gamma<1$, the quantity $n^{(1-\gamma)q}$ grows relatively quickly, necessitating a higher-order Taylor expansion of $f_t(x)$ to accurately capture all potential leading-order contributions.  
			In contrast, for $\gamma=1$, it holds that $(\log n)^q=o(n^{-\epsilon})$ for any $\epsilon>0$, and hence a second-order expansion of $f_t(x)$ suffices.
		\end{remark}

		The following lemma establishes the convergence of the leading term in the second-order asymptotic behavior of $\widehat{\bf Z}_n$.
		\begin{lemma}\label{le8} The following holds:
			
			{\rm (a)} When $1/2<\gamma<1$, for all $i,j\in\{1,2,\cdots,T\}$ and integer $q\ge 0$, we have
			\begin{equation}\label{eq32}
				\lim\limits_{n\to\infty}n^\gamma p_{n,i}p_{n,j}\sum\limits_{k=m_0}^{n}ck^{-\gamma}r_k(n^{1-\gamma}-k^{1-\gamma})^ql_{n,i}l_{n,j}=
				\frac{q!(1-\gamma)^q}{c^{q-1}(\alpha_i+\alpha_j)^{q+1}}.
			\end{equation}
			
			{\rm (b)} When $\gamma=1$ and $c[{\rm Re}(\alpha_i)+{\rm Re}(\alpha_j)]>1$, for all $i,j\in\{1,2,\cdots,T\}$ and integers $q\ge0$, we have
			\begin{equation}\label{eq26}
				\lim\limits_{n\to\infty}n p_{n,i}p_{n,j}\sum\limits_{k=m_0}^{n-1}c^2k^{-2}\left(\log n-\log k\right)^ql_{n,i}l_{n,j}=\frac{c^2q!}{[-1+(\alpha_i+\alpha_j)c]^{q+1}}.
			\end{equation}
			
			{\rm (c)} When $\gamma=1$, $c[{\rm Re}(\alpha_i)+{\rm Re}(\alpha_j)]=1$, for all $i,j\in\{1,2,\cdots,T\}$ and integers $q\ge0$, we have
			\begin{equation}\label{eq30}
				\lim\limits_{n\to\infty}\frac{n}{(\log n)^{q+1}} p_{n,i}p_{n,j}\sum\limits_{k=m_0}^{n-1}c^2k^{-2}\left(\log n-\log k\right)^ql_{n,i}l_{n,j}=\begin{cases}
					0\   &\text{for}\ \ {\rm Im}(\alpha_i+\alpha_j)\neq0;\\
					\frac{c^2}{q+1}\  &\text{for}\ \ {\rm Im}(\alpha_i+\alpha_j)=0.
				\end{cases}
			\end{equation}
		\end{lemma}

		\begin{lemma}\label{le12}
			For all $u \in \mathbb{R}$ with $u \ge 1$ and any integer $q\ge0$, the following holds:
			
			{\rm (a)} If $1/2<\gamma<1$, for any $i,j\in\{1,2,\cdots,T\}$, we have
			\begin{equation*}
				|p_{n,i}|^u|p_{n,j}|^u\sum\limits_{k=m_0}^{n-1}c^uk^{-\gamma u}r^{u}_k(n^{1-\gamma}-k^{1-\gamma})^{qu}|l_{k,i}|^u|l_{k,j}|^u=O(n^{-\gamma(2u-1)}).
			\end{equation*}
			
			{\rm (b)} If $\gamma=1$ and $uc[{\rm Re}(\alpha_i)+{\rm Re}(\alpha_j)]>2u-1$, we have
			\begin{equation*}
				|p_{n,i}|^u|p_{n,j}|^u\sum\limits_{k=m_0}^{n-1}c^uk^{-\gamma u}r^{u}_k(\log n-\log k)^{qu}|l_{k,i}|^u|l_{k,j}|^u=O(n^{-(2u-1)}).
			\end{equation*}
			
			{\rm (c)} If $\gamma=1$ and $c[{\rm Re}(\alpha_i)+{\rm Re}(\alpha_j)]=1$, we have
			\begin{equation*}
				|p_{n,i}|^u|p_{n,j}|^u\sum\limits_{k=m_0}^{n-1}c^uk^{-\gamma u}r^{u}_k(\log n-\log k)^{qu}|l_{k,i}|^u|l_{k,j}|^u=\begin{cases}
					O(n^{-1}(\log n)^{q+1})\ &\text{for}\ u=1,\\
					O(n^{-u}(\log n)^{qu})\ &\text{for}\ u>1.
				\end{cases}
			\end{equation*}
		\end{lemma}

		\begin{lemma}\label{le14}
			The sequence $(v_{n,k})_{k=m_0}^n$ defined in equations \eqref{eq50} and \eqref{eq45} satisfies 
			\begin{equation*}
				\lim\limits_{n\to\infty}v_{n,k}=0,\ \	\sum\limits_{k=m_0}^n|v_{n,k}-v_{n,k-1}|=O(1).
			\end{equation*}
		\end{lemma}

			\section{Some auxiliary results}\label{secB}
		This section presents auxiliary results that support the proofs in this paper. We begin with a symmetric variant of Lemma A.1.4 in \cite{r4}.
		\begin{lemma}[Modification of the Jordan Space]\label{le7}
			Let ${\bf J}_\lambda$ be the Jordan block associated to the eigenvalue $\lambda$ and let ${\bf Q}_\lambda$ be a base of the generalized eigenspace associated to $\lambda$ such that $ {\bf WQ}_\lambda={\bf Q}_\lambda {\bf J}_\lambda$. Then, it is possible to replace ${\bf J}_\lambda$ and ${\bf Q}_\lambda$ by a new block ${\bf J}_{\beta, \lambda}$ and a new base ${\bf Q}_{\beta, \lambda}$ such that ${\bf WQ}_{\beta, \lambda}={\bf Q}_{\beta,\lambda}{\bf J}_{\beta,\lambda}$ and
			\begin{equation*}
				\left\|{\bf J}_{\beta, \lambda}\right\|_2 \leq \frac{1+|\lambda|}{2}<1 .
			\end{equation*}
		\end{lemma}
		\begin{proof}
			For any positive real number $\beta$, we define ${\bf D}_\beta={\rm Diag}\left(1, \frac{1}{\beta}, \frac{1}{\beta^2}, \cdots, \frac{1}{\beta^K}\right)$. We have
			\begin{align*}
				{\bf J}_\lambda {\bf D}_\beta= & \begin{pmatrix}
					\lambda & 1 & 0 & \cdots & 0 \\
					0 & \lambda & 1 & \cdots & 0 \\
					\vdots & \vdots & \vdots & \ddots & \vdots \\
					0 & 0 & 0 & \cdots & 1 \\
					0 & 0 & 0 & \cdots & \lambda
				\end{pmatrix} \begin{pmatrix}
					1 & 0 & 0 & \cdots & 0 \\
					0 & 1/{\beta} & 0 & \cdots & 0 \\
					\vdots & \vdots & \vdots & \ddots & \vdots \\
					0 & 0 & 0 & \cdots & 0 \\
					0 & 0 & 0 & \cdots & 1/{\beta^{K-1}}
				\end{pmatrix}\\
				=& \begin{pmatrix}
					1 & 0 & 0 & \cdots & 0 \\
					0 & 1/{\beta} & 0 & \cdots & 0 \\
					\vdots & \vdots & \vdots & \ddots & \vdots \\
					0 & 0 & 0 & \cdots & 0 \\
					0 & 0 & 0 & \cdots & 1/{\beta^{K-1}}
				\end{pmatrix}\begin{pmatrix}
					\lambda & 1/\beta & 0 & \cdots & 0 \\
					0 & \lambda & 1/\beta & \cdots & 0 \\
					\vdots & \vdots & \vdots & \ddots & \vdots \\
					0 & 0 & 0 & \cdots & 1/\beta \\
					0 & 0 & 0 & \cdots & \lambda
				\end{pmatrix}={\bf D}_\beta{\bf J}_{\beta,\lambda} 
			\end{align*}
			Given relations ${\bf WQ}_\lambda={\bf Q}_\lambda {\bf J}_\lambda$ and ${\bf J}_\lambda {\bf D}_\beta={\bf D}_\beta{\bf J}_{\beta,\lambda}$, setting the new base ${\bf Q}_{\beta,\lambda}={\bf Q}_\lambda{\bf D}_\beta$ leads to
			\begin{equation*}
				{\bf W}{\bf Q}_{\beta,\lambda}={\bf W}{\bf Q}_\lambda{\bf D}_\beta={\bf Q}_\lambda{\bf J}_\lambda{\bf D}_\beta={\bf Q}_\lambda{\bf D}_\beta{\bf J}_{\beta,\lambda}={\bf Q}_{\beta,\lambda}{\bf J}_{\beta,\lambda}.
			\end{equation*}
			
			Since $1/\beta$ becomes adequately small when $\beta$ is sufficiently large. Given that $|\lambda|<1$, we aim to select $1/\beta$ small enough to ensure that $\|{\bf J}_{\beta,\lambda}\|_2$ is sufficiently close to $|\lambda|$, thereby guaranteeing its modulus remains less than $1$. To achieve this, note that
			\begin{equation*}
				{\bf J}_{\beta,\lambda}\overline{{\bf J}}_{\beta,\lambda}^\top=\begin{pmatrix}
					|\lambda|+1/{|\beta|^2} & {\bar{\lambda}}/{\beta} & 0 & \cdots & 0 & 0 & 0 \\
					{\lambda}/{\beta} & |\lambda|+1/{|\beta|^2} & \;\;\bar{\lambda}/\beta & \cdots & 0 & 0 & 0 \\
					\vdots & \vdots & \vdots & & \vdots & \vdots & \vdots \\
					0 & 0 & 0 & \cdots & \;\;{\lambda}/{\beta} & \;\;|\lambda|+1/{|\beta|^2} & \;\;{\bar{\lambda}}/{\beta} \\
					0 & 0 & 0 & \cdots & 0 & {\lambda}/{\beta} & |\lambda|
				\end{pmatrix},
			\end{equation*}
			then we obtain
			\begin{equation*}
				\|{\bf J}_{\beta,\lambda}\|_2\le \max\limits_{j\in\{1,2,\cdots,K\}}\left\{\sum\limits_{i=1}^{K}|[{\bf J}_{\beta,\lambda}\overline{{\bf J}}_{\beta,\lambda}^\top]_{i,j}|\right\}\le\left(|\lambda|+1/\beta\right)^2
			\end{equation*}
			By choosing $1/\beta=\sqrt{\frac{1+|\lambda|}{2}}-|\lambda|>0$, we may now conclude that Lemma \ref{le7} holds.
		\end{proof}
		
		\begin{lemma}[Lemma A.1 in \cite{r9}]\label{le9}
			Let $\left(x_n\right)_n$ be a sequence of positive numbers that satisfies the following equation:
			\begin{equation*}
				x_{n+1}=\left(1-a r_n\right) x_n+K_n r_n^2,
			\end{equation*}
			where $a>0, r_n \geq 0$ and $0 \leq K_n \leq K$. Suppose that
			\begin{equation*}
				\sum_n r_n=+\infty \quad \text { and } \quad \sum_n r_n^2<+\infty.
			\end{equation*}
			Then $\lim _{n \to\infty} x_n=0$.
		\end{lemma}

		\begin{lemma}[Lemma 1 in \cite{r10}]\label{le11}
			Suppose that ${\cal C}_n$ and ${\cal D}_n$ are ${\cal S}$-valued raning) filtration satisfying for all $n$ :
			\begin{equation*}
				\sigma\left({\cal C}_n\right) \subseteq \mathcal{G}_n \quad \text { and } \quad \sigma\left({\cal D}_n\right) \subseteq \sigma\left(\bigcup_n \mathcal{G}_n\right).	
			\end{equation*}
			If ${\cal C}_n$ stably converges to ${\cal M}$ and ${\cal D}_n$ converges to ${\cal N}$ stably in the strong sense, with respect to $\mathcal{G}$, then
			\begin{equation*}
				\left[{\cal C}_n, {\cal D}_n\right] \longrightarrow {\cal M} \otimes {\cal N} \quad \text { stably. }
			\end{equation*}
			Here, ${\cal M} \otimes {\cal N}$ is the kernel on ${\cal S} \times {\cal S}$ such that $({\cal M} \otimes {\cal N})(\omega)={\cal M}(\omega) \otimes {\cal N}(\omega)$ for all $\omega$.
		\end{lemma}

		\begin{lemma}[Lemma B.1 of \cite{r2}]\label{le10}
			Let ${\cal H}=\left({\cal H}_n\right)_n$ be a filtration and $\left(Y_n\right)_n$ a ${\cal H}$-adaptea sequence of complex random variables such that $E\left[Y_n|{\cal H}_{n-1}\right] \rightarrow Y$ almost surely. Moreover, let $\left(c_n\right)_n$ be a sequence of strictly positive real numbers such that $\sum_n E\left[|Y_n|^2\right]/c_n^2<+\infty$ and let $\left\{v_{n,k}, 1\leq k \leq n\right\}$ be a triangular array of complex numbers such that $v_{n,k}\neq 0$ and
			\begin{align*}
				&\lim _n v_{n, k}=0, \quad \lim _n v_{n, n} \text { exists\ finite, } \quad \lim _n \sum_{k=1}^n \frac{v_{n, k}}{c_k}=\eta \in \mathbb{C}, \\
				&\sum_{k=1}^n \frac{|v_{n, k}|}{c_k}=O(1), \quad \sum_{k=1}^n|v_{n, k}-v_{n, k-1}|=O(1).
			\end{align*}
			Then $\sum_{k=1}^n v_{n, k} Y_k / c_k \xrightarrow{\text { a.s }} \eta Y$.
		\end{lemma}

		\begin{theorem}[Proposition 3.1 of \cite{r6}]\label{th6}
			Let $(\mathbf{T}_{n,k})_{n \geq 1,1 \leq k \leq k_n}$ be a triangular array of d-dimensional real random vectors, such that, for each fixed n, the finite sequence $(\mathbf{T}_{n, k})_{1 \leq k \leq k_n}$ is a martingale difference array with respect to a given filtration $(\mathcal{G}_{n, k})_{k \geq 0}$. Moreover, let $(t_n)_n$ be a sequence of real numbers and assume that the following conditions hold:
			
			{\rm (c1)} $\mathcal{G}_{n, k} \subseteq \mathcal{G}_{n+1, k}$ for each $n$ and $1 \leq k \leq k_n${\rm ;}
			
			{\rm (c2)} $\sum_{k=1}^{k_n}\left(t_n \mathbf{T}_{n, k}\right)\left(t_n \mathbf{T}_{n, k}\right)^{\top}=t_n^2 \sum_{k=1}^{k_n} \mathbf{T}_{n, k} \mathbf{T}_{n, k}^{\top} \xrightarrow{P} {\bm\Sigma}$, where ${\bm\Sigma}$ is a random positive semidefinite matrix{\rm ;}
			
			{\rm (c3)} $\sup _{1 \leq k \leq k_n}\left|t_n \mathbf{T}_{n, k}\right| \xrightarrow{L^1} 0$.\\
			Then $t_n \sum_{k=1}^{k_n} \mathbf{T}_{n, k}$ converges stably to the Gaussian kernel ${\cal N}(\mathbf{0}, {\bm\Sigma})$.
		\end{theorem}

	\end{appendix}
	
	\begin{acks}[Acknowledgments]
		Li Yang was supported by Key technologies for coordination and interoperation of power distribution service resource under Grant No. 2021YFB2401300. Dandan Jiang was partially supported by NSFC under Grant No. 12571311 and 12326606. Jiang Hu was partially supported by NSFC Grants No. 12171078, No. 12292980, No. 12292982, National Key R $\&$ D Program of China No. 2020YFA0714102, and Fundamental Research Funds for the Central Universities No. 2412023YQ003. Zhidong Bai was partially supported by NSFC Grants No.12171198, No.12271536, and Team Project of Jilin Provincial Department of Science and Technology No.20210101147JC. 

    \vspace{0.3cm} 
    Correspondence concerning this article should be addressed to Dandan Jiang, E-mail: jiangdd@mail.xjtu.edu.cn.
	\end{acks}

	\begin{supplement}
		\stitle{Supplement to ``Asymptotics for Reinforced Stochastic Processes on Hierarchical Networks''}
		\sdescription{This supplementary file provides the proofs of
			the technical lemmas in Appendix \ref{secA} and some computations used in Section \ref{sec4}.}
	\end{supplement}
	
	
	

\end{document}